\documentclass[11pt, reqno]{amsart}
\usepackage{natbib}
\usepackage{graphicx,amsmath,amssymb,epsfig}
\usepackage{rotating}
\usepackage{array}
\usepackage{multirow}
\usepackage{pdflscape}

\newtheorem{theorem}{Theorem}[section]

\newtheorem{lemma}{Lemma}[section]
\newtheorem{proposition}{Proposition}[section]
\newtheorem{assumption}{A\!\!}
\theoremstyle{definition}

\newtheorem{remark}{Comment}[section]

\newcommand{\eps}{\varepsilon}

\newcommand{\be}{\begin{eqnarray}}
\newcommand{\ee}{\end{eqnarray}}

\newcommand{\ba}{\begin{array}}
\newcommand{\ea}{\end{array}}
\newcommand{\bs}{\begin{align}\begin{split}\nonumber}
\newcommand{\bsnumber}{\begin{align}\begin{split}}
\newcommand{\es}{\end{split}\end{align}}

\newcommand{\MS}{\mathcal{S}}
\newcommand{\MH}{\mathcal{H}}

\renewcommand{\(}{\left(}
\renewcommand{\)}{\right)}
\renewcommand{\hat}{\widehat}

\newcommand{\Ep}{{\mathrm{E}}}

\renewcommand{\Pr}{{\mathrm{P}}}

\def\RR{ {\mathbb{R}}}

\renewcommand{\hat}{\widehat}
\renewcommand{\leq}{\leqslant}
\renewcommand{\geq}{\geqslant}
\newcommand{\sign}{ {\rm sign}}

\providecommand{\tabularnewline}{\\}

\renewcommand\[{\begin{equation}}
\renewcommand\]{\end{equation}}

\begin{document}

\title[Testing Regression Monotonicity]{Testing Regression Monotonicity in Econometric Models}


\author[Denis Chetverikov]{Denis Chetverikov}

\date{First version: March 2012. This version: \today. Email: chetverikov@econ.ucla.edu. I thank Victor
Chernozhukov for encouragement and guidance. I am also grateful to Anna Mikusheva, Isaiah Andrews, Andres Aradillas-Lopez, Moshe Buchinsky, Glenn Ellison, Jin Hahn, Bo Honore, Rosa Matzkin, Jose Montiel, Ulrich Muller, Whitney Newey, and Jack Porter for valuable comments. The first version
of the paper was presented at the Econometrics lunch at MIT in April,
2012.}

\begin{abstract}
Monotonicity is a key qualitative prediction of a wide array of economic models derived via robust comparative statics. It is therefore important to design effective and practical econometric methods for testing this prediction in empirical analysis. This paper develops a general nonparametric framework for testing monotonicity of a regression function. Using this framework, a broad class of new tests is introduced, which gives an empirical researcher a lot of flexibility to incorporate ex ante information she might have. The paper also develops new methods for simulating critical values, which are based on the combination of a bootstrap procedure and new selection algorithms. These methods yield tests that have correct asymptotic size and are asymptotically nonconservative. It is also shown how to obtain an adaptive rate optimal test that has the best attainable rate of uniform consistency against models whose regression function has Lipschitz-continuous first-order derivatives and that automatically adapts to the unknown smoothness of the regression function. Simulations show that the power of the new tests in many cases significantly exceeds that of some prior tests, e.g. that of Ghosal, Sen, and Van der Vaart (2000). An application of the developed procedures to the dataset of Ellison and Ellison (2011) shows that there is some evidence of strategic entry deterrence in pharmaceutical industry where incumbents may use strategic investment to prevent generic entries when their patents expire.
\end{abstract}

\maketitle

\section{Introduction}


The concept of monotonicity plays an important role in economics. For
example, monotone comparative statics has been a popular research
topic in economic theory for many years; see, in particular, the seminal
work on this topic by \cite{MilgromShannon} and \cite{Athey2002}. \cite{Matzkin94} mentions monotonicity as one of the most important implications of economic theory that can be used in econometric analysis. Given importance of monotonicity in economic theory, the natural
question is whether we observe monotonicity in the data. Although there do exist some methods for testing monotonicity in statistics, there is no general theory that would suffice for empirical analysis in economics. For example, I am not aware of any test of monotonicity that would allow for multiple covariates. In addition, there are currently no published results on testing monotonicity that would allow for endogeneity of covariates. Such a theory is provided in this paper. In particular, this paper provides a general nonparametric framework for testing monotonicity of a regression function. 
Tests of monotonicity developed in this paper can be used to evaluate
assumptions and implications of economic theory concerning monotonicity. In addition, as was recently noticed by \cite{EllisonEllison}, these tests can also be used to provide evidence
of existence of certain phenomena related to strategic behavior of economic agents that are difficult to detect otherwise.
Several motivating examples are presented in the next section. 

I start with the model 
\begin{equation}
Y=f(X)+\varepsilon\label{eq: Model}
\end{equation}
where $Y$ is a scalar dependent random variable, $X$
a scalar independent random variable, $f(\cdot)$ an unknown function,
and $\varepsilon$ an
unobserved scalar random variable satisfying $\Ep[\varepsilon|X]=0$ almost surely. Later in the paper, I extend the analysis to cover models with multivariate and endogenous $X$'s. I am interested in testing the null hypothesis, $\MH_{0}$,
that $f(\cdot)$ is nondecreasing against the alternative, $\MH_{a}$, that
there are $x_1$ and $x_2$ such that $x_1<x_2$ but $f(x_1)>f(x_2)$.
The decision is to be made based on the i.i.d. sample of size $n$, $\{X_i,Y_i\}_{1\leq i\leq n}$ from the distribution of $(X,Y)$.
I assume that $f(\cdot)$ is smooth but do not impose any parametric structure on it.
I derive a theory that yields tests with the correct asymptotic size. I also show how to obtain consistent tests and how to obtain a test with the optimal rate of uniform consistency
against classes of functions with Lipschitz-continuous first order derivatives. Moreover, the rate optimal test constructed in this paper is adaptive in the sense that it automatically adapts to the unknown smoothness of $f(\cdot)$.

This paper makes several contributions.
First, I introduce a general framework for testing monotonicity. This framework allows me to develop a broad class of new tests, which also includes some existing tests as special cases. This gives a researcher a lot of flexibility to incorporate ex ante information she might have. Second, I develop new methods to simulate the critical values for these tests that in many cases yield higher power than that of existing methods. Third, I consider the problem of testing monotonicity in models with multiple covariates for the first time in the literature. As will be explained in the paper, these models are more difficult to analyze and require a different treatment in comparison with the case of univariate $X$. Finally, I consider models with endogenous $X$ that are identified via instrumental variables, and I consider models with sample selection.

Providing a \textit{general} framework for testing monotonicity is a difficult problem. The problem arises because different test statistics studied in this paper have different limit distributions and require different normalizations. Some of the test statistics have $N(0,1)$ limit distribution, and some others have an extreme value limit distribution. Importantly, there are also many test statistics that are ``in between'', so that their distributions are far both from $N(0,1)$ and from extreme value distributions, and so their asymptotic approximations are difficult to obtain. Moreover, and equally important, the limit distribution of the statistic that leads to the rate optimal and adaptive test is \textit{unknown}. The main difficulty here is that the processes underlying the test statistic do not have an asymptotic equicontinuity property, and so classical functional central limit theorems, as presented for example in \cite{VaartWellner1996} and \cite{Dudley1999}, do not apply. This paper addresses these issues and provide bootstrap critical values that are valid uniformly over a large class of different test statistics and different data generating processes. Two previous papers, \cite{HallHeckman2000} and \cite{GSV2000}, used specific techniques to prove validity of their tests of monotonicity but it is difficult to generalize their techniques to make them applicable for other tests of monotonicity. In contrast, in this paper, I introduce a general approach that can be used to prove validity of many different tests of monotonicity. Other shape restrictions, such as concavity and super-modularity, can be tested by procedures similar to those developed in this paper.

Another problem is that test statistics studied in this paper have
some asymptotic distribution when $f(\cdot)$ is constant but diverge
if $f(\cdot)$ is strictly increasing. This discontinuity implies
that for some sequences of models $f(\cdot)=f_{n}(\cdot)$,
the limit distribution depends
on the local slope function, which is an unknown infinite-dimensional
nuisance parameter that can not be estimated consistently from the
data. A common approach in the literature to solve this problem is to calibrate
the critical value using the case when the type I error is maximized (the least favorable model),
i.e. the model with constant $f(\cdot)$.\footnote{The exception is \cite{WangAndMeyer2011}
who use the model with an isotonic estimate of $f(\cdot)$ to simulate
the critical value. They do not prove whether their test maintains
the required size, however.} In contrast, I develop two selection procedures
that estimate the set where $f(\cdot)$ is not strictly increasing, and then adjust the critical value to account for this set. The estimation is conducted so that no violation of
the asymptotic size occurs. The critical values obtained using these
selection procedures yield important power improvements in comparison
with other tests if $f(\cdot)$ is strictly increasing over some subsets
of the support of $X$. The first selection procedure, which is based on the one-step approach, is
related to those developed in \cite{ChernozhukovLeeRosen2009}, \cite{AndrewsandShi2010},
and \cite{Chetverikov2012}, all of which deal with the problem of testing
conditional moment inequalities. The second selection procedure is novel and is based on the step-down approach. It is somewhat related to methods developed in \cite{RomanoWolf2005} and \cite{RomanoShaikh2010} but the details are rather different.

Further, an important issue that applies to nonparametric testing in general is how
to choose a smoothing parameter for the test. In theory, the optimal smoothing
parameter can be derived for many smoothness classes of functions
$f(\cdot)$. In practice, however, the smoothness class that $f(\cdot)$
belongs to is usually unknown. I deal with this problem by employing the adaptive testing approach. This allows me to obtain tests with good power properties when
the information about smoothness of the function $f(\cdot)$ possessed
by the researcher is absent or limited.
More precisely, I construct a test statistic using many different weighting functions that correspond to many different values of the smoothing parameter so that the distribution of the test statistic is mainly determined by the
optimal weighting function. I provide a basic set of weighting functions that yields a rate optimal and adaptive test and show how the researcher can change this set in order to incorporate ex ante information. Importantly, the approach taken in this paper does not require ``under-smoothing''. This feature of my approach is important because, to the best of my knowledge, all procedures in the literature to achieve ``under-smoothing'' are ad hoc and do not have a sound theoretical justification.

The literature on testing monotonicity of a nonparametric regression function is quite large. The tests of \cite{Gijbels2000} and
\cite{GSV2000} (from now on, GHJK and GSV, respectively) are based on the signs of $(Y_{i+k}-Y_{i})(X_{i+k}-X_{i})$.
\cite{HallHeckman2000} (from now on, HH) developed a test based on the slopes of local
linear estimates of $f(\cdot)$. The list of other papers includes \cite{Schlee82},
\cite{Bowman98}, \cite{Dumbgen2001}, \cite{Durot2003}, \cite{Beraud2005},
and \cite{WangAndMeyer2011}. In a contemporaneous work, \cite{LSW1} derive another approach to testing monotonicity based on $L_p$-functionals. The results in this paper complement the results of that paper. An advantage of their method is that the asymptotic distribution of their test statistic in the least favorable model under $\MH_0$ turns out to be $N(0,1)$, so that obtaining a critical value for their test is computationally very simple. A disadvantage of their method, however, is that their test is not adaptive. Results in this paper are also different from those in \cite{RomanoWolf11} who also consider the problem of testing monotonicity. In particular, they assume that $X$ is non-stochastic and discrete, which makes their problem semi-parametric and substantially simplifies proving validity of critical values, and they test the null hypothesis that $f(\cdot)$ is \textit{not} weakly increasing against the alternative that it is weakly increasing. \cite{LLW2009} and \cite{DelgadoEscanciano2010} derived tests of stochastic monotonicity, which is a related but different problem. Specifically, stochastic monotonicity
means that the conditional cdf of $Y$ given $X$, $F_{Y|X}(y,x)$,
is (weakly) decreasing in $x$ for any fixed $y$.

As an empirical application of the results developed in this paper, I
consider the problem of detecting strategic entry deterrence in the pharmaceutical
industry. In that industry, incumbents whose drug patents are about
to expire can change their investment behavior in order to prevent
generic entries after the expiration of the patent. Although there
are many theoretically compelling arguments as to how and why incumbents
should change their investment behavior (see, for example, \cite{Tirole_book}),
the empirical evidence is rather limited. \cite{EllisonEllison}
showed that, under certain conditions, the dependence of investment
on market size should be monotone if no strategic entry deterrence
is present. In addition, they noted that the entry deterrence motive
should be important in intermediate-sized markets and less
important in small and large markets. Therefore, strategic entry deterrence
might result in the non-monotonicity of the relation between market
size and investment. Hence, rejecting the null hypothesis of
monotonicity provides the evidence in favor of the existence of
strategic entry deterrence. I apply the tests developed in this paper
to Ellison and Ellison's dataset and show that there is some evidence of non-monotonicity
in the data. The evidence is rather weak, though.

The rest of the paper is organized as follows. Section \ref{sec: motivating examples} provides motivating examples. Section \ref{sec: test} describes
the general test statistic and gives several methods to simulate the critical value.
Section \ref{sec: theory under high level conditions} contains the main results under high-level conditions when there are no additional covariates. 
Since in most practically relevant
cases, the model also contains some additional covariates, Section
\ref{sec: multivariate models} studies the cases of fully nonparametric and partially linear models with multiple covariates. Section \ref{sec: endogenous covariates} extends the analysis to cover the case where $X$ is endogenous and identification is achieved via instrumental variables. Section \ref{sec: sample selection models} briefly explains how to test monotonicity in sample selection models.
Section \ref{sec: monte carlo} presents a small Monte Carlo simulation study. Section \ref{sec: empirical example}
describes the empirical application. Section \ref{sec: conclusion} concludes. All proofs
are contained in the Appendix. In addition, Appendix \ref{sec: algorithms} contains implementation details, and Appendix \ref{sec: verification of high level conditions} is devoted to the verification of high-level conditions under primitive assumptions.

\textit{Notation}. Throughout this paper, let $\{\epsilon_i\}$ denote a sequence of independent $N(0,1)$ random variables that are independent of the data. The sequence $\{\epsilon_i\}$ will be used in bootstrapping critical values.  The notation $i=\overline{1,n}$ is a shorthand for $i\in\{1,...,n\}$. For any set $\MS$, I denote the number of elements in this set by $|\MS|$. 

\section{Motivating Examples}\label{sec: motivating examples}
Many testable implications of economic theory are concerned with comparative
statics analysis. These implications most often take the form of qualitative
statements like ``Increasing factor $X$ will positively (negatively)
affect response variable $Y$''. The common approach to test such
implications on the data is to look at the corresponding coefficient in the linear
(or other parametric) regression. Relying on these strong parametric assumptions, however, can lead to highly misleading results. For example, the test based on the linear regression will not be consistent and the test based on the quadratic regression may severely over-reject if the model is misspecified.
In contrast, this paper provides a class of tests that are valid without these strong parametric assumptions.
The purpose of this section is to give three examples from the literature where tests developed in this paper can be applied.

\textbf{1. Detecting strategic effects.} Certain strategic effects,
the existence of which is difficult to prove otherwise, can be detected by testing for monotonicity. An example on strategic entry deterrence in
the pharmaceutical industry is described in the Introduction and is analyzed in
Section \ref{sec: empirical example}. Below I provide another
example concerned with the problem of debt pricing. This example is based on \cite{MorrisShin2001}. Consider a model
where investors hold a collateralized debt. The debt will yield a
fixed payment, say 1, in the future if it is rolled over and an underlying
project is successful. Otherwise the debt will yield nothing (0). Alternatively,
all investors have an option of not rolling over and getting the value
of the collateral, $\kappa\in(0,1)$, immediately. The probability that the project turns
out to be successful depends on the fundamentals, $\theta$, and on how many investors
roll over. Specifically, assume that the project is successful if $\theta$ exceeds the proportion of investors who roll over. Under \textit{global game reasoning}, if private information possessed by investors is sufficiently accurate, the project will succeed if and only if $\theta\geq \kappa$; see \cite{MorrisShin2001} for details. Then ex ante value of the debt is given by
$$
V(\kappa)=\kappa\cdot\Pr(\theta<\kappa)+1\cdot\Pr(\theta\geq \kappa),
$$
and the derivative of the ex ante debt value with respect to the collateral value is
$$
\frac{dV(\kappa)}{d\kappa}=\Pr(\theta<\kappa)-(1-\kappa)\frac{d\Pr(\theta<\kappa)}{d\kappa}
$$
The first and second terms on the right hand side of this equation represent direct and strategic effects, respectively. The strategic effect represents coordination failure among investors. It arises because high value
of the collateral leads investors to believe that many other investors
will not roll over, and the project will not be successful even though the project is profitable ($\kappa<1$). \cite{MorrisShin2004} argue that this
effect is important for understanding anomalies in empirical implementation
of the standard debt pricing theory of \cite{Merton74}. A natural
question is how to prove existence of this effect in the data. Note that in the absence of strategic effect, the relation between value of the debt and value of the collateral will be monotonically increasing. If strategic effect is sufficiently strong, however, it can cause non-monotonicity in this relation. Therefore, one can detect the existence of the strategic effect and coordination failure by testing whether conditional mean of the price of the debt
given the value of the collateral is a monotonically increasing function.
Rejecting the null hypothesis of monotonicity provides evidence
in favor of the existence of the strategic effect and coordination failure.

\textbf{2. Testing assumptions of treatment effect models.} Monotonicity
is often assumed in the econometrics literature on estimating treatment
effects. A widely used econometric model in this literature is as
follows. Suppose that we observe a sample of individuals, $i=\overline{1,n}$.
Each individual has a random response function $y_{i}(t)$ that gives
her response for each level of treatment $t\in T$. Let $z_{i}$ and
$y_{i}=y_{i}(z_{i})$ denote the realized level of the treatment and
the realized response, respectively (both are observable).
The problem is how to derive inference on $E[y_{i}(t)]$. To address this problem, \cite{ManskiPepper} introduced assumptions of
monotone treatment response, which imposes that $y_{i}(t_{2})\geq y_{i}(t_{1})$
whenever $t_{2}\geq t_{1}$, and monotone treatment selection, which imposes
that $E[y_{i}(t)|z_{i}=v]$ is increasing in $v$ for all $t\in T$.
The combination of these assumptions yields a testable prediction.
Indeed, for all $v_{2}\geq v_{1}$,
\begin{align*}
E[y_{i}|z_{i}=v_{2}] & =  E[y_{i}(v_{2})|z_{i}=v_{2}]\\
 & \geq E[y_{i}(v_{1})|z_{i}=v_{2}]
 \geq  E[y_{i}(v_{1})|z_{i}=v_{1}]
  = E[y_{i}|z_{i}=v_{1}].
\end{align*}
Since both $z_i$ and $y_i$ are observed, this prediction can be tested by the
procedures developed in this paper. Note that the tests of stochastic monotonicity as described in the Introduction do not apply here since the testable prediction is monotonicity of the conditional mean function.


\textbf{3. Testing the theory of the firm.}
A classical paper \cite{HolmstromMilgrom94} on the
theory of the firm is built around the observation that in multi-task
problems different incentive instruments are expected
to be complementary to each other. Indeed, increasing an incentive
for one task may lead the agent to spend too much time on
that task ignoring other responsibilities. This can be avoided if
incentives on different tasks are balanced with each other. To derive
testable implications of the theory, Holmstrom and Milgrom study
a model of industrial selling introduced in \cite{AndersonSchmittlein}
where a firm chooses between an in-house agent and an independent representative who divide their time into four tasks: (i) direct sales,
(ii) investing in future sales to customers, (iii) non-sale activities,
such as helping other agents, and (iv) selling the products of other
manufacturers. Proposition 4 in their paper states that under certain
conditions, the conditional probability of having an
in-house agent is a (weakly) increasing function of the marginal
cost of evaluating performance and is a (weakly) increasing function
of the importance of non-selling activities. These are hypotheses that
can be directly tested on the data by procedures developed in this
paper. This would be an important extension of linear regression
analysis performed, for example, in \cite{AndersonSchmittlein} and
\cite{PoppoZenger98}. Again, note that the tests of stochastic monotonicity as described in the Introduction do not apply here.



\section{The Test}\label{sec: test}

\subsection{The General Test Statistic}
Recall that I consider a model given in equation (\ref{eq: Model}), and the test should be based on the i.i.d. sample $\{X_i,Y_i\}_{1\leq i\leq n}$ of $n$ observations from the distribution of $(X,Y)$ where $X$ and $Y$ are independent and dependent random variables, respectively. In this section and in Section \ref{sec: theory under high level conditions}, I assume that $X$ is a scalar and there are no additional covariates $Z$. The case where additional covariates $Z$ are present is considered in Section \ref{sec: multivariate models}.

Let $Q(\cdot,\cdot):\RR\times\RR\rightarrow\RR$ be a weighting function satisfying $Q(x_1,x_2)=Q(x_2,x_1)$ and $Q(x_1,x_2)\geq 0$ for all $x_1,x_2\in\RR$, and let
$$
b=b(\{X_i,Y_i\})=(1/2)\sum_{1\leq i,j\leq n}(Y_i-Y_j)\sign(X_j-X_i)Q(X_i,X_j)
$$
be a test function. Since $Q(X_i,X_j)\geq 0$ and $\Ep[Y_i|X_i]=f(X_i)$, it is easy to see that under $\MH_0$, that is, when the function $f(\cdot)$ is non-decreasing, $\Ep[b]\leq 0$. On the other hand, if $\MH_0$ is violated and there exist $x_1$ and $x_2$ on the support of $X$ such that $x_1<x_2$ but $f(x_1)>f(x_2)$, then there exists a function $Q(\cdot,\cdot)$ such that $\Ep[b]>0$ if $f(\cdot)$ is smooth. Therefore, $b$ can be used to form a test statistic if I can find an appropriate function $Q(\cdot,\cdot)$. For this purpose, I will use the adaptive testing approach developed in statistics literature. Even though this approach has attractive features, it is almost never used in econometrics. A notable exception is \cite{Horowitz2001}, who used it for specification testing.

The idea behind the adaptive testing approach is to choose $Q(\cdot,\cdot)$ from a large set of potentially useful weighting functions that maximizes the studentized version of $b$.
Formally, let $\MS_n$ be some general set that depends on $n$ and is (implicitly) allowed to depend on $\{X_i\}$, and for $s\in\MS_n$, let $Q(\cdot,\cdot,s):\RR\times\RR\rightarrow\RR$ be some function satisfying $Q(x_1,x_2,s)=Q(x_2,x_1,s)$ and $Q(x_1,x_2,s)\geq 0$ for all $x_1,x_2\in\RR$. The functions $Q(\cdot,\cdot,s)$ are also (implicitly) allowed to depend on $\{X_i\}$.
In addition, let
\begin{equation}\label{eq: test function definition}
b(s)=b(\{X_i,Y_i\},s)=(1/2)\sum_{1\leq i,j\leq n}(Y_i-Y_j)\sign(X_j-X_i)Q(X_i,X_j,s)
\end{equation}
be a test function. Conditional on $\{X_i\}$, the variance of $b(s)$ is given by
\begin{equation}\label{eq: variance definition}
V(s)=V(\{X_i\},\{\sigma_i\},s)=\sum_{1\leq i\leq n}\sigma_i^2\(\sum_{1\leq j\leq n}\sign(X_j-X_i)Q(X_i,X_j,s)\)^2
\end{equation}
where $\sigma_i=(\Ep[\eps_i^2|X_i])^{1/2}$ and $\eps_i=Y_i-f(X_i)$. In general, $\sigma_i$'s are unknown, and have to be estimated from the data. Let $\hat{\sigma}_i$ denote some (not necessarily consistent) estimator of $\sigma_i$. Available estimators are discussed later in this section. Then the estimated conditional variance of $b(s)$ is
\begin{equation}\label{eq: variance estimated definition}
\hat{V}(s)=V(\{X_i\},\{\hat{\sigma}_i\},s)=\sum_{1\leq i\leq n}\hat{\sigma}_i^2\(\sum_{1\leq j\leq n}\sign(X_j-X_i)Q(X_i,X_j,s)\)^2.
\end{equation}
The general form of the test statistic that I consider in this paper is
\begin{equation}\label{eq: test statistic definition}
T=T(\{X_i,Y_i\},\{\hat{\sigma}_i\},\MS_n)=\max_{s\in\MS_n}\frac{b(\{X_i,Y_i\},s)}{\sqrt{\hat{V}(\{X_i\},\{\hat{\sigma}_i\},s)}}.
\end{equation}
Large values of $T$ indicate that the null hypothesis is violated. Later in this section, I will provide methods for estimating quantiles of $T$ under $\MH_0$ and for choosing a critical value for the test based on the statistic $T$.

The set $\MS_n$ determines adaptivity properties of the test, that is the ability of the test to detect many different deviations from $\MH_0$. Indeed, each weighting function $Q(\cdot,\cdot,s)$ is useful for detecting some deviation, and so the larger is the set of weighting functions $\MS_n$, the larger is the number of different deviations that can be detected, and the higher is adaptivity of the test. In this paper, I allow for \textit{exponentially} large (in the sample size $n$) sets $\MS_n$. This implies that the researcher can choose a huge set of weighting functions, which allows her to detect large set of different deviations from $\MH_0$. The downside of the adaptivity, however, is that expanding the set $\MS_n$ increases the critical value, and thus decreases the power of the test against those alternatives that can be detected by weighting functions already included in $\MS_n$. Fortunately, in many cases the loss of power is relatively small. In particular, it follows from Lemma \ref{lem: maximal inequality} and Borell's inequality (see Proposition A.2.1 in \cite{VaartWellner1996}) that the critical values for test developed below are bounded from above by a slowly growing $C(\log p)^{1/2}$ for some $C>0$ where $p=|\MS_n|$, the number of elements in the set $\MS_n$.

\subsection{Typical Weighting Functions}\label{sub: typical weighting functions}
Let me now describe typical weighting functions. Consider some compactly supported kernel function $K:\RR\rightarrow\RR$ satisfying $K(x)\geq 0$ for all $x\in \RR$. For convenience, I will assume that the support of $K$ is $[-1,1]$.
In addition, let $s=(x,h)$ where $x$ is a location point and $h$ is a bandwidth value (smoothing parameter). Finally, define
\begin{equation}\label{eq: weighting function}
Q(x_1,x_2,(x,h))=|x_1-x_2|^kK\left(\frac{x_1-x}{h}\right)K\left(\frac{x_2-x}{h}\right)
\end{equation}
for some $k\geq 0$. I refer to this $Q$ as a \textit{kernel weighting function}.\footnote{It is possible to extend the definition of kernel weighting functions given in (\ref{eq: weighting function}). Specifically, the term $|x_1-x_2|^k$ in the definition can be replaced by general function $\bar{K}(x_1,x_2)$ satisfying $\bar{K}(x_1,x_2)\geq 0$ for all $x_1$ and $x_2$. I thank Joris Pinkse for this observation.}

Assume that a test is based on kernel weighting functions and $\MS_n$ consists of pairs $s=(x,h)$ with many different values of $x$ and $h$. To explain why this test has good adaptivity properties, consider figure 1 that plots two regression functions. Both $f_1(\cdot)$ and $f_2(\cdot)$ violate $\MH_0$ but locations where $\MH_0$ is violated are different. In particular, $f_1(\cdot)$ violates $\MH_0$ on the interval $[x_1,x_2]$ and $f_2(\cdot)$ violates $\MH_0$ on the interval $[x_3,x_4]$. In addition, $f_1(\cdot)$ is relatively less smooth than $f_2(\cdot)$, and $[x_1,x_2]$ is shorter than $[x_3,x_4]$. To have good power against $f_1(\cdot)$, $\MS_n$ should contain a pair $(x,h)$ such that $[x-h,x+h]\subset[x_1,x_2]$. Indeed, if $[x-h,x+h]$ is not contained in $[x_1,x_2]$, then positive and negative values of the summand of $b$ will cancel out yielding a low value of $b$. In particular, it should be the case that $x\in[x_1,x_2]$. Similarly, to have good power against $f_2(\cdot)$, $\MS_n$ should contain a pair $(x,h)$ such that $x\in[x_3,x_4]$. Therefore, using many different values of $x$ yields a test that adapts to the location of the deviation from $\MH_0$. This is spatial adaptivity. Further, note that larger values of $h$ yield higher signal-to-noise ratio. So, given that $[x_3,x_4]$ is longer than $[x_1,x_2]$, the optimal pair $(x,h)$ to test against $f_2(\cdot)$ has larger value of $h$ than that to test against $f_1(\cdot)$. Therefore, using many different values of $h$ results in adaptivity with respect to smoothness of the function, which, in turn, determines how fast its first derivative is varying and how long the interval of non-monotonicity is.

If no ex ante information is available, I recommend using kernel weighting functions with $\MS_n=\{(x,h):x\in\{X_1,...,X_n\},h\in H_n\}$ where $H_n=\{h=h_{\max}u^l:h\geq h_{\min},l=0,1,2,...\}$ and $h_{\max}=\max_{1\leq i,j\leq n}|X_i-X_j|/2$. I also recommend setting $u=0.5$, $h_{\min}=0.4h_{\max}(\log n/n)^{1/3}$, and $k=0$ or $1$. I refer to this $\MS_n$ as a \textit{basic set} of weighting functions. This choice of parameters is consistent with the theory presented in this paper and has worked well in simulations. The basic set of weighting functions yields a rate optimal and adaptive test. The value of $h_{\min}$ is selected so that the test function $b(s)$ for any given $s$ uses no less than approximately 15 observations when $n=100$ and $X$ is distributed uniformly on some interval.

If some ex ante information is available, the general framework considered here gives the researcher a lot of flexibility to incorporate this information. In particular, if the researcher expects that the function $f(\cdot)$ is rather smooth, then the researcher can restrict the set $\MS_n$ by considering only pairs $(x,h)$ with large values of $h$ since in this case deviations from $\MH_0$, if present, are more likely to happen on long intervals. Moreover, if the smoothness of the function $f(\cdot)$ is known, one can find an optimal value of the smoothing parameter $\tilde{h}=\tilde{h}_n$ corresponding to this level of smoothness, and then consider kernel weighting functions with this particular choice of the bandwidth value, that is $\MS_n=\{(x,h): x\in\{X_1,...,X_n\},h=\tilde{h}\}$. Further, if non-monotonicity is expected at one particular point $\tilde{x}$, one can consider kernel weighting functions with $\MS_n=\{(x,h):x=\tilde{x},h=\tilde{h}\}$ or $\MS_n=\{(x,h):x=\tilde{x},h\in H_n\}$ depending on whether smoothness of $f(\cdot)$ is known or not. More broadly, if non-monotonicity is expected on some interval $\mathcal{X}$, one can use kernel weighting functions with $\MS_n=\{(x,h):x\in\{X_1,...,X_n\}\cap\mathcal{X},h\in\tilde{h}\}$ or $\MS_n=\{(x,h):x\in\{X_1,...,X_n\}\cap\mathcal{X},h\in H_n\}$ again depending on whether smoothness of $f(\cdot)$ is known or not. Note that all these modifications will increase the power of the test because smaller sets $\MS_n$ yield lower critical values.

\begin{figure}\label{fig: two deviations}
\caption{Regression Functions Illustrating Different Deviations from $\MH_0$}
\includegraphics[scale=0.5]{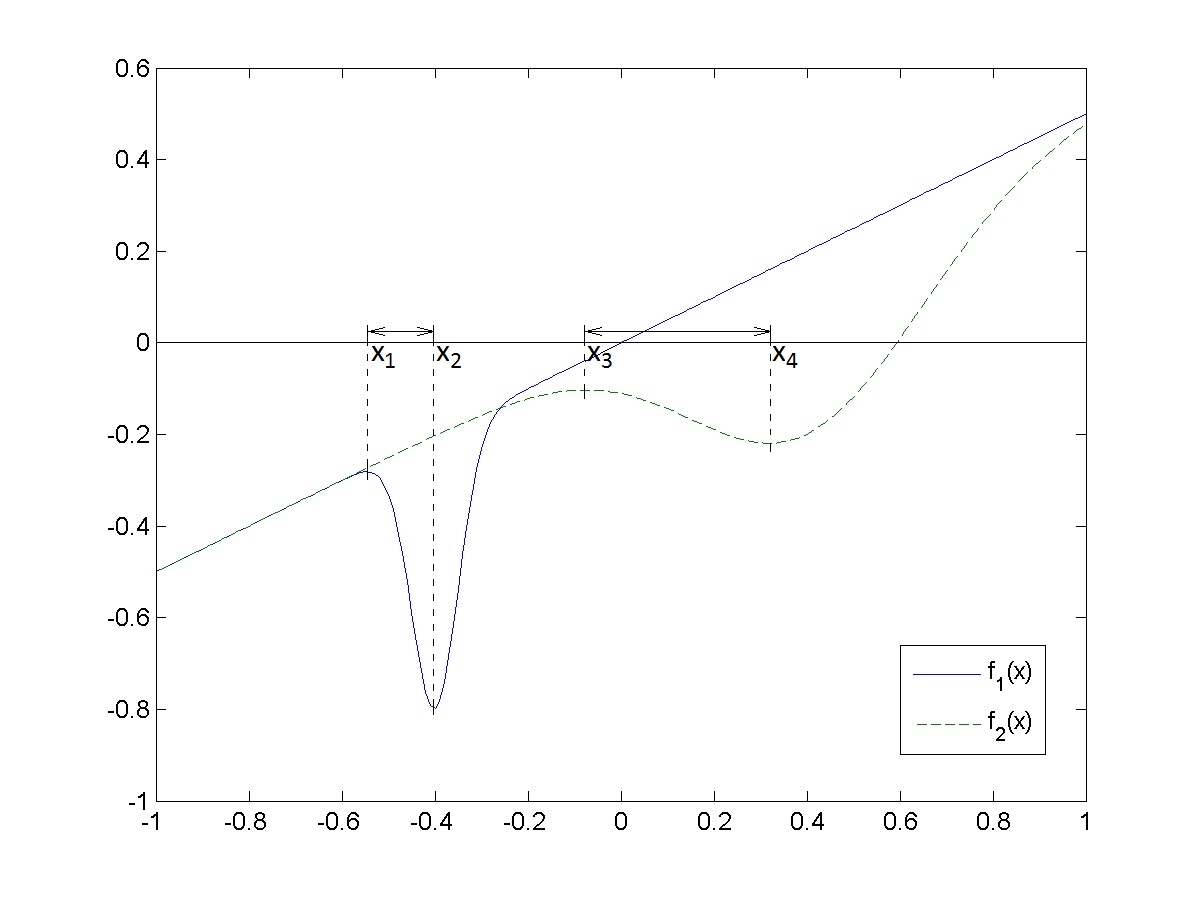}
\end{figure}

Another interesting choice of the weighting functions is
$$
Q(x_1,x_2,s)=\sum_{1\leq r\leq m}|x_1-x_2|^kK\left(\frac{x_1-x^r}{h}\right)K\left(\frac{x_2-x^r}{h}\right)
$$
where $s=(x^1,...,x^m,h)$. These weighting functions are useful if the researcher expects multiple deviations from $\MH_0$.

\subsection{Comparison with Other Known Tests}
I will now show that the general framework described above includes the Hall and Heckman's (HH) test statistic and a slightly modified version of the Ghosal, Sen, and van der Vaart's (GSV) test statistic as special cases that correspond to different values of $k$ in the definition of kernel weighting functions. 

GSV use the following test function:
$$
b(s)=(1/2)\sum_{1\leq i,j\leq n}\sign(Y_i-Y_j)\sign(X_j-X_i)K\left(\frac{X_i-x}{h}\right)K\left(\frac{X_j-x}{h}\right),
$$
whereas setting $k=0$ in equation (\ref{eq: weighting function}) yields
\begin{equation}\label{eq: GSV modification}
b(s)=(1/2)\sum_{1\leq i,j\leq n}(Y_i-Y_j)\sign(X_j-X_i)\left(\frac{X_i-x}{h}\right)K\left(\frac{X_j-x}{h}\right),
\end{equation}
and so the only difference is that I include the term $(Y_i-Y_j)$ whereas they use $\sign(Y_i-Y_j)$. It will be shown in the next section that my test is consistent. On the other hand, I claim that GSV test is not consistent under the presence of conditional heteroscedasticity. Indeed, assume that $f(X_i)=-X_i$, and that $\eps_i$ is $-2X_i$ or $2X_i$ with equal probabilities. Then $(Y_i-Y_j)(X_j-X_i)>0$ if and only if $(\eps_i-\eps_j)(X_j-X_i)>0$, and so the probability of rejecting $\MH_0$ for the GSV test is numerically equal to that in the model with $f(X_i)=0$ for $i=\overline{1,n}$. But the latter probability does not exceed the size of the test. This implies that the GSV test is not consistent since it maintains the required size asymptotically. Moreover, they consider a unique non-stochastic value of $h$, which means that the GSV test is nonadaptive with respect to the smoothness of the function $f(\cdot)$.

Let me now consider the HH test. The idea of this test is to make use of local linear estimates of the slope of the function $f(\cdot)$. Using well-known formulas for the OLS regression, it is easy to show that the slope estimate of the function $f(\cdot)$ given the data $\{X_i,Y_i\}_{i=s_1+1}^{s_2}$ with $s_1<s_2$ where $\{X_i\}_{i=1}^n$ is an increasing sequence is given by
\begin{equation}\label{eq: HH test statistic original}
b(s)=\frac{\sum_{s_1< i\leq s_2}Y_i\sum_{s_1< j\leq s_2}(X_i-X_j)}{(s_2-s_1)\sum_{s_1< i\leq s_2}X_i^2-(\sum_{s_1< i\leq s_2}X_i)^2},
\end{equation}
where $s=(s_1,s_2)$. Note that the denominator of (\ref{eq: HH test statistic original}) depends only on $X_i$'s, and so it disappears after studentization. In addition, simple rearrangements show that the numerator in (\ref{eq: HH test statistic original}) is up to the sign is equal to
\begin{equation}\label{eq: proportial term}
(1/2)\sum_{1\leq i,j\leq n}(Y_i-Y_j)(X_j-X_i)1\{x-h\leq X_i\leq x+h\}1\{x-h\leq X_j\leq x+h\}
\end{equation}
for some $x$ and $h$. On the other hand, setting $k=1$ in equation (\ref{eq: weighting function}) yields
\begin{equation}\label{eq: HH test statistic}
b(s)=(1/2)\sum_{1\leq i,j\leq n}(Y_i-Y_j)(X_j-X_i)K\left(\frac{X_i-x}{h}\right)K\left(\frac{X_j-x}{h}\right).
\end{equation}
Noting that expression in (\ref{eq: proportial term}) is proportional to that on the right hand side in (\ref{eq: HH test statistic}) with $K(\cdot)=1\{[-1,+1]\}(\cdot)$ implies that the HH test statistic is a special case of those studied in this paper.

\subsection{Estimating $\sigma_{i}$}\label{sub: estimating sigma}

In practice, $\sigma_{i}$'s are usually unknown, and, hence, have to
be estimated from the data. Let $\hat{\sigma}_{i}$ denote some
estimator of $\sigma_{i}$. I provide results for two types of
estimators. The first type of estimators is easier to implement but the second worked better in simulations.

First, $\sigma_{i}$ can be estimated by the residual $\hat{\eps}_{i}$. More precisely, let $\hat{f}(\cdot)$ be some
uniformly consistent estimator of $f(\cdot)$ with a polynomial rate of consistency
in probability, i.e. $\hat{f}(X_{i})-f(X_{i})=o_{p}(n^{-\kappa_1})$ uniformly over $i=\overline{1,n}$
for some $\kappa_1>0$, and let $\hat{\sigma}_{i}=\hat{\eps}_{i}$
where $\hat{\eps}_{i}=Y_{i}-\hat{f}(X_{i})$. Note that $\hat{\sigma}_i$ can be negative. Clearly, $\hat{\sigma}_{i}$
is not a consistent estimator of $\sigma_{i}$. Nevertheless, as I will show in Section \ref{sec: theory under high level conditions}
that this estimator leads to valid inference. Intuitively, it works because
the test statistic contains the weighted average sum of $\sigma_{i}^{2}$ over
$i=\overline{1,n}$, and the estimation error averages out. To obtain
a uniformly consistent estimator $\hat{f}(\cdot)$ of $f(\cdot)$, one can use a
series method (see \cite{Newey97}, theorem 1) or local polynomial
method (see \cite{Tsybakov_book}, theorem 1.8). If one prefers
kernel methods, it is important to use generalized kernels in order
to deal with boundary effects when higher order kernels are used; see, for example, \cite{Muller91}. Alternatively, one
can choose $\MS_n$ so that boundary points are excluded from the test statistic. In addition, if the researcher decides to impose some parametric structure on the set of potentially possible heteroscedasticity functions, then parametric methods like OLS will typically give uniform consistency with $\kappa_1$ arbitrarily close to $1/2$.

The second way of estimating $\sigma_i$ is to use a parametric or nonparametric estimator $\hat{\sigma}_i$ satisfying $\hat{\sigma}_{i}-\sigma_{i}=o_{p}(n^{-\kappa_2})$ uniformly over $i=\overline{1,n}$
for some $\kappa_2>0$. Many estimators of $\sigma_{i}$ satisfy
this condition. Assume that the observations $\{X_i,Y_i\}_{i=1}^n$ are arranged so that $X_i\leq X_j$ whenever $i\leq j$. Then the estimator of \cite{Rice1984}, given by
\begin{equation}
\hat{\sigma}=\left(\frac{1}{2n}\sum_{i=1}^{n-1}(Y_{i+1}-Y_{i})^{2}\right)^{1/2},\label{eq: Rice formula}
\end{equation}
is $\sqrt{n}$-consistent if $\sigma_i=\sigma$ for all $i=\overline{1,n}$ and $f(\cdot)$ is piecewise Lipschitz-continuous. 

The Rice estimator can be easily modified to allow for conditional heteroscedasticity. Choose a bandwidth value $b_{n}>0$.
For $i=\overline{1,n}$, let $J(i)=\{j=\overline{1,n}:|X_j-X_i|\leq b_{n}\}$.
Let
$|J(i)|$ denote the number of elements in $J(i)$. Then $\sigma_i$
can be estimated by
\begin{equation}\label{eq: local rice estimator}
\hat{\sigma}_i=\left(\frac{1}{2|J(i)|}\sum_{j\in J(i):j+1\in J(i)}(Y_{j+1}-Y_{j})^2\right)^{1/2}.
\end{equation}
I refer to (\ref{eq: local rice estimator}) as a local version of Rice's estimator. An advantage of this estimator is that it is adaptive
with respect to the smoothness of the function $f(\cdot)$. Proposition \ref{lem: consistency of rice estimator} in Appendix \ref{sec: verification of high level conditions} provides conditions that are sufficient for uniform consistency of this estimator with a polynomial rate. The key condition there is that $|\sigma_{j+1}-\sigma_j|\leq C|X_{j+1}-X_j|$ for some $C>0$ and all $j=\overline{1,n-1}$.
The intuition for consistency is as follows.
Note that $X_{j+1}$ is close to $X_{j}$.
So, if the function $f$ is continuous, then
$$
Y_{j+1}-Y_{j}=f(X_{j+1})-f(X_{j})+\varepsilon_{j+1}-\varepsilon_{j}\approx\varepsilon_{j+1}-\varepsilon_{j},
$$
so that
$$
\Ep[(Y_{j+1}-Y_{j})^2|\{X_i\}]\approx\sigma_{j+1}^2+\sigma_{j}^2
$$
since $\varepsilon_{j+1}$ is independent of $\varepsilon_{j}$.
Further, if $b_{n}$ is sufficiently small, then $\sigma_{j+1}^2+\sigma_{j}^2\approx2\sigma_{i}^2$
since $|X_{j+1}-X_{i}|\leq b_{n}$ and $|X_{j}-X_{i}|\leq b_{n}$, and so $\hat{\sigma}_i^2$ is close to $\sigma_i^2$. Other available estimators are presented, for example, in \cite{Muller1987}, \cite{FanYao1998}, \cite{Horowitz2001}, \cite{HardleTsybakov2007}, and \cite{Cai2008}.

\subsection{Simulating the Critical Value}

In this subsection, I provide three different methods for estimating quantiles of the null distribution of the test statistic $T$. These are plug-in, one-step, and step-down methods. All of these methods are based on the procedure known as the Wild bootstrap. The Wild bootstrap was introduced in \cite{Wu1986} and used, among many others, by \cite{Liu88}, \cite{Mammen93}, \cite{HardleMammen1993}, \cite{Horowitz2001}, and \cite{Chetverikov2012}. See also \cite{ChernozhukovChetverikov12}. The three methods are arranged in terms of increasing power and computational complexity. The validity of all three methods is established in theorem \ref{thm: size}.
 Recall that $\{\epsilon_i\}$ denotes a sequence of independent $N(0,1)$ random variables that are independent of the data.

\subsection*{Plug-in Approach}
Suppose that we want to obtain a test of level $\alpha$.
The plug-in approach is based on two observations. First, under $\MH_0$,
\begin{align}
b(s)&=(1/2)\sum_{1\leq i,j\leq n}(f(X_i)-f(X_j)+\eps_i-\eps_j)\sign(X_j-X_i)Q(X_i,X_j,s)\\
&\leq
(1/2)\sum_{1\leq i,j\leq n}(\eps_i-\eps_j)\sign(X_j-X_i)Q(X_i,X_j,s)
\end{align}
since $Q(X_i,X_j,s)\geq 0$ and $f(X_i)\geq f(X_j)$ whenever $X_i\geq X_j$ under $\MH_0$, and so the $(1-\alpha)$ quantile of $T$ is bounded from above by the $(1-\alpha)$ quantile of $T$ in the model with $f(x)=0$ for all $x\in\RR$, which is the least favorable model under $\MH_0$. Second, it will be shown that the distribution of $T$ asymptotically
depends on the distribution of noise $\{\eps_i\}$ only through
$\{\sigma_{i}^{2}\}$.
These two observations suggest that the critical value for the test can be obtained by simulating the conditional $(1-\alpha)$ quantile of $T^{\star}=T(\{X_i,Y_i{^\star}\},\{\hat{\sigma}_i\},\MS_n)$ given $\{X_i\}$, $\{\hat{\sigma}_i\}$, and $\MS_n$ where $Y_i^{\star}=\hat{\sigma}_i\epsilon_i$ for $i=\overline{1,n}$. This is called the plug-in critical value $c_{1-\alpha}^{PI}$. See Section \ref{sec: algorithms} of the Appendix for detailed step-by-step instructions.

\subsection*{One-Step Approach}
The test with the plug-in critical value is computationally rather simple. It has, however, poor power properties. Indeed, the distribution of $T$ in general depends on $f(\cdot)$ but the plug-in approach is based on the least favorable regression function $f(x)=0$ for all $x\in\mathbb{R}$, and so it is too conservative when $f(\cdot)$ is strictly increasing. More formally, suppose for example that kernel weighting functions are used, and that $f(\cdot)$ is strictly increasing in $h$-neighborhood
of $x_1$ but is constant in $h$-neighborhood of $x_2$. Let $s_1=s(x_1,h)$ and $s_2=s(x_2,h)$. Then $b(s_1)/(\hat{V}(s_1))^{1/2}$ is no greater than $b(s_2)/(\hat{V}(s_2))^{1/2}$
with probability approaching one. On the other hand, $b(s_1)/(\hat{V}(s_1))^{1/2}$
is greater than $b(s_2)/(\hat{V}(s_2))^{1/2}$ with nontrivial probability in the
model with $f(x)=0$ for all $x\in\RR$, which is used to obtain $c_{1-\alpha}^{PI}$. Therefore, $c_{1-\alpha}^{PI}$ overestimates
the corresponding quantile of $T$. The natural idea to overcome the conservativeness of the plug-in approach is to simulate a critical value using not all elements of $\MS_n$ but only those that are relevant for the given sample.
In this paper, I develop two selection procedures that are used to decide what elements of $\MS_n$
should be used in the simulation. The main difficulty here
is to make sure that the selection procedures do not distort the size
of the test. The simpler of these two procedures is the one-step approach.

Let $\{\gamma_{n}\}$ be a sequence of positive numbers converging
to zero, and let $c_{1-\gamma_n}^{PI}$ be the $(1-\gamma_n)$ plug-in critical value. In addition, denote 
$$
\MS_n^{OS}=\MS_n^{OS}(\{X_i,Y_i\},\{\hat{\sigma}_i\},\MS_n)=\{s\in\MS_n:b(s)/(\hat{V}(s))^{1/2}>-2c_{1-\gamma_n}^{PI}\}.
$$
Then the one-step critical value $c_{1-\alpha}^{OS}$ is the conditional $(1-\alpha)$ quantile of the simulated statistic $T^{\star}=T(\{X_i,Y_i{^\star}\},\{\hat{\sigma}_i\},\MS_n^{OS})$ given $\{X_i\}$, $\{\hat{\sigma}_i\}$, and $\MS_n^{OS}$ where $Y_i^{\star}=\hat{\sigma}_i\epsilon_i$ for $i=\overline{1,n}$.\footnote{If $\MS_n^{OS}$ turns out to be empty, assume that $\MS_n^{OS}$ consists of one randomly chosen element of $\MS_n$.} 
Intuitively, the one-step critical value works because the weighting functions corresponding to elements of the set $\MS_n\backslash\MS_n^{OS}$ have an asymptotically negligible influence on the distribution of $T$ under $\MH_0$. Indeed, it will be shown that the probability that at least one element $s$ of $\MS_n$ such that
\begin{equation}\label{eq: threshold value}
(1/2)\sum_{1\leq i,j\leq n}(f(X_i)-f(X_j))\sign(X_j-X_i)Q(X_i,X_j,s)/(\hat{V}(s))^{1/2}> -c_{1-\gamma_n}^{PI}
\end{equation}
belongs to the set $\MS_n\backslash\MS_n^{OS}$ is at most $\gamma_n+o(1)$.
On the other hand, the probability that at least one element $s$ of $\MS_n$ such that inequality (\ref{eq: threshold value}) does not hold for this element gives $b(s)/(\hat{V}(s))^{1/2}> 0$ is again at most $\gamma_n+o(1)$.
Since $\gamma_n$ converges to zero, this suggests that the critical value can be simulated using only elements of $\MS_n^{OS}$. In practice, one can set $\gamma_n$ as a small fraction of $\alpha$. For example, the Monte Carlo simulations presented in this paper use $\gamma_n=0.01$ with $\alpha=0.1$.\footnote{More formally, it is shown in the proof of Theorem \ref{thm: size} that the probability of rejecting $\MH_0$ under $\MH_0$ in large sample is bounded from above by $\alpha+2\gamma_n$. This suggests that if the researcher does not agree to tolerate small size distortions, she can use the test with level $\tilde{\alpha}=\alpha-2\gamma_n$ instead. On the other hand, I note that $\alpha+2\gamma_n$ is only an \textit{upper} bound on the probability of rejecting $\MH_0$, and in many cases the true probability of rejecting $\MH_0$ is smaller than $\alpha+2\gamma_n$.}

\subsection*{Step-down Approach}
The one-step approach, as the name suggests, uses only one step to cut out those elements of $\MS_n$ that have negligible influence on the distribution of $T$. It turns out that this step can be iterated using the step-down procedure and yielding second-order improvements in the power. 
The step-down procedures were developed in the literature on multiple hypothesis testing; see, in particular, \cite{Holm1979}, \cite{RomanoWolf2005}, \cite{RomanoWolf2005ECMA}, and \cite{RomanoShaikh2010}. See also \cite{LehmannRomano2005} for a textbook introduction. The use of step-down method in this paper, however, is rather different. 

To explain the step-down approach, let me define the sequences $(c_{1-\gamma_n}^{l})_{l=1}^{\infty}$ and $(\MS_n^l)_{l=1}^{\infty}$.  Set $c_{1-\gamma_n}^1=c_{1-\gamma_n}^{OS}$ and $\MS_n^1=\MS_n^{OS}$. Then for $l>1$, let $c_{1-\gamma_n}^l$ be the conditional $(1-\gamma_n)$ quantile of $T^{\star}=T(\{X_i,Y_i^{\star}\},\{\hat{\sigma}_i\},\MS_n^l)$ given $\{\hat{\sigma}_i\}$ and $\MS_n^l$ where $Y_i^{\star}=\hat{\sigma}_i\epsilon_i$ for $i=\overline{1,n}$ and
$$
\MS_n^l=\MS_n^l(\{X_i,Y_i\},\{\hat{\sigma}_i\},\MS_n)=\{s\in\MS_n:b(s)/(\hat{V}(s))^{1/2}>-c_{1-\gamma_n}^{PI}-c_{1-\gamma_n}^{l-1}\}.
$$
It is easy to see that $(c_{1-\gamma_n}^l)_{l=1}^{\infty}$ is a decreasing sequence, and so $\MS_n^l\supseteq\MS_n^{l+1}$ for all $l\geq 1$. Since $\MS_n^1$ is a finite set, $\MS_n^{l(0)}=\MS_n^{l(0)+1}$ for some $l(0)\geq 1$ and $\MS_n^l=\MS_n^{l+1}$ for all $l\geq l(0)$. Let $\MS_n^{SD}=\MS_n^{l(0)}$. Then the step-down critical value $c_{1-\alpha}^{SD}$ is the conditional $(1-\alpha)$ quantile of $T^{\star}=T(\{X_i,Y_i^{\star}\},\{\hat{\sigma}_i\},\MS_n^{SD})$ given $\{X_i\}$, $\{\hat{\sigma}_i\}$, and $\MS_n^{SD}$ where $Y_i^{\star}=\hat{\sigma}_i\epsilon_i$ for $i=\overline{1,n}$. 

Note that $\MS_n^{SD}\subset\MS_n^{OS}\subset\MS_n$, and so $c_\eta^{SD}\leq c_\eta^{OS}\leq c_\eta^{PI}$ for any $\eta\in(0,1)$. This explains that the three methods for simulating the critical values are arranged in terms of increasing power.

\section{Theory under High-Level Conditions}\label{sec: theory under high level conditions}

This section describes the high-level assumptions used in the paper and presents
the main results under these assumptions.

Let $c_1$, $C_1$, $\kappa_1$, $\kappa_2$, and $\kappa_3$ be strictly positive constants. The size properties of the test will be obtained under the following assumptions.

\begin{assumption}\label{as: disturbances}
$\Ep[|\eps_i|^4|X_i]\leq C_1$ and $\sigma_i\geq c_1$ for all $i=\overline{1,n}$.
\end{assumption}
\noindent
This is a mild assumption on the moments of disturbances. The condition $\sigma_i\geq c_1$ for all $i=\overline{1,n}$ precludes the existence of super-efficient estimators.

Recall that the results in this paper are obtained for two types of estimators of $\sigma_i$. When $\hat{\sigma}_i=\hat{\epsilon}_i=Y_i-\hat{f}(X_i)$ for some estimator $\hat{f}(\cdot)$ of $f(\cdot)$, I will assume

\begin{assumption}\label{as: function estimation}
(i) $\hat{\sigma}_i=Y_i-\hat{f}(X_i)$ for all $i=\overline{1,n}$ and (ii) $\hat{f}(X_i)-f(X_i)=o_p(n^{-\kappa_1})$ uniformly over $i=\overline{1,n}$.
\end{assumption}
\noindent
This assumption is satisfied for many parametric and nonparametric estimators of $f(\cdot)$; see, in particular, Subsection \ref{sub: estimating sigma}. When $\hat{\sigma}_i$ is some consistent estimator of $\sigma_i$, I will assume

\begin{assumption}\label{as: sigma}
$\hat{\sigma}_i-\sigma_i=o_p(n^{-\kappa_2})$ uniformly over $i=\overline{1,n}$.
\end{assumption}
\noindent
See Subsection \ref{sub: estimating sigma} for different available estimators. See also Proposition \ref{lem: consistency of rice estimator} in Appendix \ref{sec: verification of high level conditions} where Assumption A\ref{as: sigma} is proven for the local version of Rice's estimator.

\begin{assumption}\label{as: V}
$(\hat{V}(s)/V(s))^{1/2}-1=o_p(n^{-\kappa_3})$ and $(V(s)/\hat{V}(s))^{1/2}-1=o_p(n^{-\kappa_3})$ uniformly over $s\in\MS_n$.
\end{assumption}
\noindent
This is a high-level assumption that is verified for kernel weighting functions under primitive conditions in Appendix \ref{sec: verification of high level conditions} (Proposition \ref{lem: kernel function}).

Let
\begin{equation}\label{eq: sensitivity parameter}
A_n=\max_{s\in\MS_n}\max_{1\leq i\leq n}\left|\sum_{1\leq j\leq n}\sign(X_j-X_i)Q(X_i,X_j,s)/(V(s))^{1/2}\right|.
\end{equation}
I refer to $A_n$ as a sensitivity parameter. It provides an upper bound on how much any test function depends on a particular observation. Intuitively, approximation of the distribution of the test statistic is possible only if $A_n$ is sufficiently small.
\begin{assumption}\label{as: growth condition}
(i) $nA_n^4(\log (pn))^{7}=o_p(1)$ where $p=|\MS_n|$, the number of elements in the set $\MS_n$; (ii) if A\ref{as: function estimation} holds, then $\log p/n^{(1/4)\wedge\kappa_1\wedge\kappa_3}=o_p(1)$, and if A\ref{as: sigma} holds, then $\log p/n^{\kappa_2\wedge\kappa_3}=o_p(1)$.
\end{assumption}
\noindent
This is a key growth assumption that restricts the choice of the weighting functions and, hence, the set $\MS_n$. Note that this condition includes $p$ only through $\log p$, and so it allows an exponentially large (in the sample size $n$) number of weighting functions. Proposition \ref{lem: kernel function} in Appendix \ref{sec: verification of high level conditions} provides an upper bound on $A_n$ for kernel weighting functions, allowing me to verify this assumption under primitive conditions for the basic set of weighting functions.

Let $\mathcal{M}$ be a class of models given by equation (\ref{eq: Model}), regression function $f(\cdot)$, joint distribution of $X$ and $\varepsilon$ such that $\Ep[\varepsilon|X]=0$ almost surely, weighting functions $Q(\cdot,\cdot,s)$ for $s\in\MS_n$, and estimators $\{\hat{\sigma}_i\}$ such that uniformly over this class, (i) Assumptions A\ref{as: disturbances}, A\ref{as: V}, and A\ref{as: growth condition} are satisfied, and (ii) either Assumption A\ref{as: function estimation} or A\ref{as: sigma} is satisfied.\footnote{Assumptions A\ref{as: function estimation}, A\ref{as: sigma}, A\ref{as: V}, and A\ref{as: growth condition} contain statements of the form $Z=o_p(n^{-\kappa})$ for some random variable $Z$ and $\kappa>0$. I say that these assumptions hold uniformly over a class of models if for any $C>0$, $\Pr(|Z|>Cn^{-\kappa})=o(1)$ uniformly over this class. Note that this notion of uniformity is weaker than uniform convergence in probability; in particular, it applies to random variables defined on different probability spaces.} For $M\in\mathcal{M}$, let $\Pr_M(\cdot)$ denote the probability measure generated by the model $M$.
\begin{theorem}[Size properties of the test]\label{thm: size}
Let $P=PI$, $OS$, or $SD$. Let $\mathcal{M}_0$ denote the set of all models $M\in\mathcal{M}$ satisfying $\MH_0$. Then
$$
\inf_{M\in\mathcal{M}_0}\Pr_M(T\leq c_{1-\alpha}^P)\geq 1-\alpha+o(1)\text{ as }n\rightarrow \infty.
$$
In addition, let $\mathcal{M}_{00}$ denote the set of all models $M\in\mathcal{M}_0$ such that $f(x)=C$ for some constant $C$ and all $x\in\RR$. Then
$$
\sup_{M\in\mathcal{M}_{00}}\Pr(T\leq c_{1-\alpha}^P)=1-\alpha+o(1)\text{ as }n\rightarrow \infty.
$$
\end{theorem}
\begin{remark}
(i) This theorem states that the Wild Bootstrap combined with the selection procedures developed in this paper yields valid critical values. Moreover, critical values are valid uniformly over the class of models $\mathcal{M}_0$. The second part of the theorem states that the test is nonconservative in the sense that its level converges to the nominal level $\alpha$.

\noindent
(ii) The proof technique used in this theorem is based on \textit{finite sample} approximations that are built on the results of \cite{ChernozhukovChetverikov12} and \cite{ChernozhukovKato2011}. In particular, the validity of the bootstrap is established \textit{without} refering to the asymptotic distribution of the test statistic. 

\noindent
(iii) The standard techniques from empirical process theory as presented, for example, in \cite{VaartWellner1996} can not be used to prove the results of Theorem \ref{thm: size}. The problem is that it is not possible to embed the process $\{b(s)/(V(s))^{1/2}:s\in\MS_n\}$ into asymptotically equicontinuous process since, for example, when the basic set of kernel weighting functions is used, random variables $b(x_1,h)/(V(x_1,h))^{1/2}$ and $b(x_2,h)/(V(x_2,h))^{1/2}$ for fixed $x_1<x_2$ become asymptotically independent as $h\rightarrow 0$.

\noindent
(iv) To better understand the importance of the finite sample approximations used in the proof of this theorem, suppose that one would like to use \textit{asymptotic} approximations based on the limit distribution of the test statistic $T$ instead. The difficulty with this approach would be to \textit{derive} the limit distribution of $T$. Indeed, note that the class of test statistics considered in this theorem is large and different limit distributions are possible. For example, if $\MS_n$ is a singleton with $Q(\cdot,\cdot,s)$ being a kernel weighting function, so that $s=(x,h)$, and $x$ being fixed and $h=h_n$ converging to zero with an appropriate rate (this statistic is useful if the researcher wants to test monotonicity in a neighborhood of a particular point $x$), then $T\Rightarrow N(0,1)$ where $\Rightarrow$ denotes weak convergence. On the other hand, if kernel weighting functions are used but $\MS_n=\{(x,h):x\in\{X_1,...,X_n\},h=h_n\}$ with $h=h_n$ converging to zero with an appropriate rate (this statistic is useful if the researcher wants to test monotonicity on the whole support of $X$ but the smoothness of the function $f(\cdot)$ is known from the ex ante considerations), then it is possible to show that $a_n(T-b_n)\Rightarrow \mathcal{G}$ where $\mathcal{G}$ is the Gumbel distribution for some $a_n$ and $b_n$ that are of order $(\log n)^{1/2}$. This implies that not only limit distributions vary among test statistics in the studied class but also appropriate normalizations are different. Further, note that the theorem also covers test statistics that are ``in between'' the two statistics described above (these are statistics with $\MS_n=\{(x,h):x\in\{X_{l(1)},...,X_{l(k)}\},h=h_n\}$ where $l(j)\in\{1,...,n\}$, $j=\overline{1,k}$, $k\in\{1,...,n\}$; these statistics are useful if the researcher wants to test monotonicity on some subset pf the support of $X$). The distribution of these statistics can be far both from $N(0,1)$ and from $\mathcal{G}$ in finite samples for \textit{any} sample size $n$. Finally, if the basic set of weighting functions is used, the limit distribution of the corresponding test statistic $T$ is \textit{unknown} in the literature, and one can guess that this distribution is quite complicated. In contrast, Theorem \ref{thm: size} shows that the critical values suggested in this paper are valid uniformly over the whole class of test statistics under consideration.

\noindent
(v) Note that $T$ asymptotically has a form of U-statistic. The analysis of such statistics typically requires a preliminary Hoeffding projection. An advantage of the approximation method used in this paper is that it applies directly to the test statistic with no need for the Hoeffding projection, which greatly simplifies the analysis.

\noindent
(vi) To obtain a particular application of the general result presented in this theorem, assume that the basic set of weighting functions introduced in Subsection \ref{sub: typical weighting functions} is used. Suppose that Assumption A\ref{as: disturbances} holds. In addition, suppose that either Assumption A\ref{as: function estimation} or Assumption A\ref{as: sigma} holds. Then Assumptions A\ref{as: V} and A\ref{as: growth condition} hold by Proposition \ref{lem: kernel function} in Appendix \ref{sec: verification of high level conditions} (under mild conditions on $K(\cdot)$ stated in Proposition \ref{lem: kernel function}), and so the result of Theorem \ref{thm: size} applies. Therefore, the basic set of weighting functions yields a test with the correct asymptotic size, and so it can be used for testing monotonicity. An advantage of this set is that, as will follow from Theorems \ref{thm: uniform consistency rate} and \ref{thm: minimax lower bound}, it gives an adaptive test with the best attainable rate of uniform consistency in the minimax sense against alternatives with regression functions that have Lipschitz-continuous first order derivatives.

\noindent
(vii) The theorem can be extended to cover more general versions of the test statistic $T$. In particular, the theorem continues to hold if the test statistic $T$ is replaced by
$$
\widetilde{T}=\widetilde{T}(\{X_i,Y_i\},\{\hat{\sigma}_i\},\MS_n)=\max_{s\in\MS_n}\frac{w(\{X_i\},s)b(\{X_i,Y_i\},s)}{\sqrt{\hat{V}(\{X_i\},\{\hat{\sigma}_i\},s)}}
$$
where for some $c,C>0$, $w(\{X_i\},\cdot):\MS_n\rightarrow [c,C]$ is the function that can be used to give unequal weights to different weighting functions (if $\widetilde{T}$ is used, then critical values should be calculated using $\widetilde{T}(\cdot,\cdot,\cdot)$ as well). This gives the researcher additional flexibility to incorporate information on what weighting functions are more important ex ante.

\noindent
(viii) Finally, the theorem can be extended by equicontinuity arguments to cover certain cases with infinite $\MS_n$ where maximum over infinite set of weighting functions in the test statistic $T$ can be well approximated by maximum over some finite set of weighting functions. Note that since $p$ in the theorem can be large, approximation is typically easy to achieve. In particular, if $\mathcal{X}$ is the support of $X$, $\mathcal{X}$ has positive Lebesgue measure, the distribution of $X$ is absolutely continuous with respect to Lebesgue measure on $\mathcal{X}$, and the density of $X$ is bounded below from zero and from above on $\mathcal{X}$, then the theorem holds with kernel weighting functions and $\MS_n=\{(x,h):x\in\mathcal{X},h\in H_n\}$ where $H_n=\{h=h_{\max}u^l:h\geq h_{\min},l=0,1,2,\dots\}$, $h_{\max}=C$ and $h_{\min}=ch_{\max}(\log n/n)^{1/3}$ for some $c,C>0$. I omit the proof for brevity. Note, however, that using this set of weighting functions would require knowledge of the support $\mathcal{X}$ (or some method to estimate it). In addition, taking maximum in the test statistic $T$ would require using some optimization methods whereas for the basic set of weighting functions, where the maximum is taken only over $Cn\log n$ points for some $C>0$, calculating the test statistic $T$ is computationally straightforward.\qed
\end{remark}

Let $s_l=\inf\{\text{support of }X\}$ and $s_r=\sup\{\text{support of }X\}$. To prove consistency of the test and to derive the rate of consistency against one-dimensional alternatives, I will also incorporate the following assumptions.

\begin{assumption}\label{as: design simple}
For any interval $[x_1,x_2]\subset[s_l,s_r]$, $\Pr(X\in[x_1,x_2])>0$.
\end{assumption}
\noindent
This is a mild assumption stating that the support of $X$ is the interval $[s_l,s_r]$ and that for any subinterval of the support, there is a positive probability that $X$ belongs to this interval. Let $c_2$ and $C_2$ be strictly positive constants.

\begin{assumption}\label{as: weighting functions simple}
For any interval $[x_1,x_2]\subset[s_l,s_r]$, w.p.a.1, there exists $s\in\MS_n$ satisfying (i) the support of $Q(\cdot,\cdot,s)$ is contained in $[x_1,x_2]^2$, (ii) $Q(\cdot,\cdot,s)$ is bounded from above by $C_2$, and (iii) there exist non-intersecting subintervals $[x_{l1},x_{r1}]$ and $[x_{l2},x_{r2}]$ of $[x_1,x_2]$ such that $Q(y_1,y_2,s)\geq c_2$ whenever $y_1\in[x_{l1},x_{r1}]$ and $y_2\in[x_{l2},x_{r2}]$.
\end{assumption}

This assumption is satisfied, for example, if Assumption A\ref{as: design simple} holds, kernel functions are used, $\MS_n=\{(x,h):x\in\{X_1,\dots,X_n\},h\in H_n\}$ where $H_n=\{h=h_{\max}u_l:h\geq h_{\min},l=0,1,2,\dots\}$, $h_{\max}$ is bounded below from zero, and $h_{\min}\rightarrow 0$.

Let $\mathcal{M}_1$ be a subset of $\mathcal{M}$ consisting of all models satisfying Assumptions A\ref{as: design simple} and A\ref{as: weighting functions simple}. Then
\begin{theorem}[Consistency against fixed alternatives]\label{thm: consistency}
Let $P=PI$, $OS$, or $SD$. Then for any model $M$ from the class $\mathcal{M}_1$ such that $f(\cdot)$ is continuously differentiable and there exist $x_1,x_2\in[s_l,s_r]$ such that $x_1<x_2$ but $f(x_1)>f(x_2)$ ($\MH_0$ is false),
$$
\Pr_M(T\leq c_{1-\alpha}^P)\rightarrow 0\text{ as }n\rightarrow\infty.
$$
\end{theorem}
\begin{remark}\label{com: consistency}
(i) This theorem shows that the test is consistent against any fixed
continuously differentiable alternative.

\noindent
(ii) To compare the critical values based on the selection procedures
developed in this paper with the plug-in approach (no selection
procedure), assume that $f(\cdot)$ is continuously differentiable and strictly
increasing ($\MH_{0}$ holds). Then an argument like that used in the proof of Theorem 2 shows
that $\MS_{n}^{OS}$ and $\MS_n^{SD}$ will be singletons w.p.a.1,
which means that $\Pr\{c_{1-\alpha}^{OS}\leq C\}\rightarrow1$ and $\Pr\{c_{1-\alpha}^{SD}\leq C\}\rightarrow1$ for some $C>0$. On the other
hand, it follows from the Sudakov-Chevet Theorem (see, for example, Theorem 2.3.5 in \cite{Dudley1999}) that $\Pr(c_{1-\alpha}^{PI}>C)\rightarrow 1$ for any $C>0$. Finally, under Assumption A\ref{as: weighting functions difficult}, which is stated below, it follows from the proof of lemma 2.3.15 in \cite{Dudley1999}
that $\Pr\{c_{1-\alpha}^{PI}>c\sqrt{\log n}\}\rightarrow1$ for some $c>0$. This implies that there exist sequences of alternatives against which the tests with the one-step and step-down critical values are consistent but the test with the plug-in critical value is not.\qed  
\end{remark}

\begin{theorem}[Consistency against one-dimensional alternatives]\label{thm: consistency rate against one-dimensional alternatives}
Let $P=PI$, $OS$, or $SD$. Consider any model $M$ from the class $\mathcal{M}_1$ such that $f(\cdot)$ is continuously differentiable and there exist $x_1,x_2\in[s_l,s_r]$ such that $x_1<x_2$ but $f(x_1)>f(x_2)$ ($\MH_0$ is false). Assume that for every sample size $n$, the true model $M_n$ coincides with $M$ except that the regression function has the form $f_n(x)=l_nf(x)$ for all $x\in\RR$ and for some sequence $\{l_n\}$ of positive numbers converging to zero. Then
$$
\Pr_{M_n}(T\leq c_{1-\alpha}^P)\rightarrow 0\text{ as }n\rightarrow\infty
$$
as long as $\log p=o_p(l_n^2n)$.
\end{theorem}
\begin{remark}
(i) This theorem establishes the consistency of the test against one-dimensional local alternatives, which are often used in the literature to investigate the power of the test; see, for example, \cite{AndrewsandShi2010}, \cite{LeeandSongandWhang2011}, and the discussion in \cite{Horowitz2001}.

\noindent
(ii) Suppose that the basic set of weighting functions is used. Then $\log p\leq C\log n$ for some $C>0$, and so the test is consistent against one-dimensional local alternatives if $(\log n/n)^{1/2}=o(l_n)$.

\noindent
(iii) Now suppose that kernel weighting functions are used but $\MS_n$ is a maximal subset of the basic set such that for any $x_1,x_2,h$ satisfying $(x_1,h)\in\MS_n$ and $(x_2,h)\in\MS_n$, $|x_2-x_1|>2h$, and instead of $h_{\min}=0.4h_{\max}(\log n/n)^{1/3}$ we have $h_{\min}\rightarrow 0$ arbitrarily slowly. Then the test is consistent against one-dimensional local alternatives if $n^{-1/2}=o(l_n)$ (more precisely, for any sequence $l_n$ such that $n^{-1/2}=o(l_n)$, one can choose a sequence $h_{\min}=h_{\min,n}$ satisfying $h_{\min}\rightarrow 0$ sufficiently slowly so that the test is consistent). In words, this test is $\sqrt{n}$-consistent against such alternatives. I note however, that the practical value of this $\sqrt{n}$-consistency is limited because there is no guarantee that for any given sample size $n$ and given deviation from $\MH_0$, weighting functions suitable for detecting this deviation are already included in the test statistic. In contrast, it will follow from Theorem \ref{thm: uniform consistency rate} that the test based on the basic set of weighting functions does provide this guarantee.\qed
\end{remark}

Let $c_3,C_3,c_4,C_4,c_5$ be strictly positive constants. In addition, let $L>0$, $\beta\in(0,1]$, $k\geq 0$, and $h_n=(\log p/n)^{1/(2\beta+3)}$. To derive the uniform consistency rate against the classes of alternatives with Lipschitz derivatives, Assumptions A\ref{as: design simple} and A\ref{as: weighting functions simple} will be replaced by the following (stronger) conditions.

\begin{assumption}\label{as: design difficult}
For any interval $[x_1,x_2]\subset[s_l,s_r]$, $c_3(x_2-x_1)\leq \Pr(X\in[x_1,x_2])\leq C_3(x_2-x_1)$.
\end{assumption}
\noindent
This assumption is stronger than A\ref{as: design simple} in that it bounds the probabilities from above in addition to bounding probabilities from below and excludes mass points but is still often imposed in the literature.

\begin{assumption}\label{as: weighting functions difficult}
W.p.a.1, for all $[x_1,x_2]\subset[s_l,s_r]$ with $x_2-x_1=h_n$, there exists $s\in\MS_n$ satisfying (i) the support of $Q(\cdot,\cdot,s)$ is contained in $[x_1,x_2]^2$, (ii) $Q(\cdot,\cdot,s)$ is bounded from above by $C_4h_n^k$, and (iii) there exist non-intersecting subintervals $[x_{l1},x_{r1}]$ and $[x_{l2},x_{r2}]$ of $[x_1,x_2]$ such that $x_{r1}-x_{l1}\geq c_5h_n$, $x_{r2}-x_{l2}\geq c_5h_n$, $x_{l2}-x_{r1}\geq c_5h_n$, and $Q(y_1,y_2,s)\geq c_4h_n^k$ whenever $y_1\in[x_{l1},x_{r1}]$ and $y_2\in[x_{l2},x_{r2}]$.
\end{assumption}

\noindent
This condition is stronger than Assumption A\ref{as: weighting functions simple}; specifically, in Assumption A\ref{as: weighting functions difficult}, the qualifier ``w.p.a.1'' applies uniformly over all $[x_1,x_2]\subset[s_l,s_r]$ with $x_2-x_1=h_n$. Proposition \ref{lem: asumption 9 verification} shows that Assumption A\ref{as: weighting functions difficult} is satisfied under mild conditions on the kernel $K(\cdot)$ if Assumption A\ref{as: design difficult} holds and the basic set of weighting functions is used.

Let $f^{(1)}(\cdot)$ denote the first derivative of $f(\cdot)$.
\begin{assumption}\label{as: function smoothness}
For any $x_1,x_2\in [s_l,s_r]$, $|f^{(1)}(x_1)-f^{(1)}(x_2)|\leq L|x_1-x_2|^\beta$.
\end{assumption}
\noindent
This is a smoothness condition that requires that the regression function is sufficiently well-behaved.

Let $\mathcal{M}_2$ be the subset of $\mathcal{M}$ consisting of all models satisfying Assumptions A\ref{as: design difficult}, A\ref{as: weighting functions difficult}, and A\ref{as: function smoothness}. The following theorem gives the uniform rate of consistency.

\begin{theorem}[Uniform consistency rate]\label{thm: uniform consistency rate}
Let $P=PI$, $OS$, or $SD$. Consider any sequence of positive numbers $\{l_n\}$ such that $l_n\rightarrow\infty$, and let $\mathcal{M}_{2n}$ denote the subset of $\mathcal{M}_2$ consisting of all models such that the regression function $f$ satisfies $\inf_{x\in [s_l,s_r]}f^{(1)}(x)<-l_n(\log p/n)^{\beta/(2\beta+3)}$. Then
$$
\sup_{M\in\mathcal{M}_{2n}}\Pr_M(T\leq c_{1-\alpha}^P)\rightarrow 0\text{ as }n\rightarrow\infty.
$$
\end{theorem}
\begin{remark}\label{rem: comment uniform consistency}
(i) Theorem \ref{thm: uniform consistency rate} gives the rate of uniform consistency of the test against classes of functions with Lipschitz-continuous first order derivative with Lipschitz constant $L$ and order $\beta$. Importance
of \textit{uniform} consistency against sufficiently large classes
of alternatives like those considered here was previously emphasized
in \cite{Horowitz2001}. Intuitively, it guarantees that there are
no reasonable alternatives against which the test has low power if
the sample size is sufficiently large.

\noindent
(ii) Theorem \ref{thm: uniform consistency rate} shows that plug-in, one-step, and step-down critical values yield tests with the same rate of uniform consistency. Nonetheless, it does not mean that the selection procedures used in one-step and step-down critical values yield no power improvement in comparison with plug-in critical value. Specifically, it was shown in Comment \ref{com: consistency} that there exist sequences of alternatives against which tests with one-step and step-down critical values are consistent but the test with the plug-in  critical value is not.  

\noindent
(iii) Suppose that $\MS_n$ consists of the basic set of weighting functions. In addition, suppose that Assumptions A\ref{as: disturbances} and A\ref{as: design difficult} are satisfied. Further, suppose that either Assumption A\ref{as: function estimation} or A\ref{as: sigma} is satisfied. Then Assumptions A\ref{as: V} and A\ref{as: growth condition} hold by Proposition \ref{lem: kernel function} and Assumption A\ref{as: weighting functions difficult} holds by Proposition \ref{lem: asumption 9 verification} under mild conditions on the kernel $K(\cdot)$. Then Theorem \ref{thm: uniform consistency rate} implies that the test with this set of weighting functions is uniformly consistent against classes of functions with Lipschitz-continuous first order derivative with Lipschitz order $\beta$ whenever $\inf_{x\in[s_l,s_r]}f^{(1)}_n(x)<-l_n(\log n/n)^{\beta/(2\beta+3)}$ for some $l_n\rightarrow\infty$. On the other hand, it will be shown in Theorem \ref{thm: minimax lower bound} that no test can be uniformly consistent against models with $\inf_{x\in[s_l,s_r]}f^{(1)}_n(x)>-C(\log n/n)^{\beta/(2\beta+3)}$ for some sufficiently small $C>0$ if it controls size, at least asymptotically. Therefore, the test based on the basic set of weighting functions is rate optimal in the minimax sense.

\noindent
(iv) Note that the test is rate optimal in the minimax sense simultaneously against classes of functions with Lipschitz-continuous first order derivative with Lipschitz order $\beta$ for all $\beta\in(0,1]$. In addition, implementing the test does not require the knowledge of $\beta$. For these reasons, the test with the basic set of weighting functions is called adaptive and rate optimal.\qed
\end{remark}

To conclude this section, I present a theorem that gives a lower bound on the possible rates of uniform consistency against the class $\mathcal{M}_2$ so that no test that maintains asymptotic size can have a faster rate of uniform consistency. Let $\psi=\psi(\{X_i,Y_i\})$ be a generic test. In other words, $\psi(\{X_i,Y_i\})$ is the probability that the test rejects upon observing the data $\{X_i,Y_i\}$. Note that for any deterministic test $\psi=0$ or $1$.

\begin{theorem}[Lower bound on possible consistency rates]\label{thm: minimax lower bound}
For any test $\psi$ satisfying $\Ep_M[\psi]\leq \alpha+o(1)$ as $n\rightarrow\infty$ for all models $M\in\mathcal{M}_2$ such that $\MH_0$ holds,
there exists a sequence of models $M=M_n$ belonging to the class $\mathcal{M}_2$ such that $f(\cdot)=f_n(\cdot)$ satisfies $\inf_{x\in[s_l,s_r]}f_{n}^{(1)}(x)<-C(\log n/n)^{\beta/(2\beta+3)}$ for some sufficiently small constant $C>0$ and
$\Ep_{M_n}[\psi]\leq \alpha+o(1)$
as $n\rightarrow\infty$. Here $\Ep_{M_n}[\cdot]$ denotes the expectation under the distributions of the model $M_n$.
\end{theorem}
\begin{remark}
This theorem shows that no test can be uniformly consistent against models with $\inf_{x\in[s_l,s_r]}f^{(1)}_n(x)>-C(\log n/n)^{\beta/(2\beta+3)}$ for some sufficiently small $C>0$ if it controls size, at least asymptotically.\qed
\end{remark}

\section{Models with Multiple Covariates}\label{sec: multivariate models}

Most empirical studies contain additional covariates that should be
controlled for. In this section, I extend the results presented in Section \ref{sec: theory under high level conditions} to allow for this possibility. I consider cases of both nonparametric and partially linear models. For brevity, I will only consider the results concerning size properties of the test. The power properties can be obtained using the arguments closely related to those used in Theorems \ref{thm: consistency}, \ref{thm: consistency rate against one-dimensional alternatives}, and \ref{thm: uniform consistency rate}.

\subsection{Multivariate Nonparametric Model}
In this subsection, I consider a general nonparametric regression model, so that the model is given by
\begin{equation}\label{eq: nonparametric model}
Y=f(X,Z)+\eps
\end{equation}
where $Y$ is a scalar dependent random variable, $X$ a scalar independent random variable, $Z$ a vector in $\RR^d$ of additional independent random variables that should be controlled for, $f(\cdot)$ an unknown function, and $\eps$ an unobserved scalar random variable satisfying $\Ep[\eps|X,Z]=0$ almost surely. 

Let $S_z$ be some subset of $\RR^d$. The null hypothesis, $\MH_0$, to be tested is that for any $x_1,x_2\in\RR$ and $z\in S_z$, $f(x_1,z)\leq f(x_2,z)$ whenever $x_1\leq x_2$. The alternative, $\MH_a$, is that there are $x_1,x_2\in\RR$ and $z\in S_z$ such that $x_1< x_2$ but $f(x_1,z)>f(x_2,z)$.  The decision is to be made based on the i.i.d. sample of size $n$, $\{X_i,Z_i,Y_i\}_{1\leq i\leq n}$ from the distribution of $(X,Z,Y)$.

The choice of the set $S_z$ is up to the researcher and has to be made depending on the hypothesis to be tested. For example, if $S_z=\RR^d$, then $\MH_0$ means that the function $f(\cdot)$ is increasing in the first argument for all values of the second argument. If the researcher is interested in one particular value, say $z_0$, then she can set $S_z=z_0$, which will mean that under $\MH_0$, the function $f(\cdot)$ is increasing in the first argument when the second argument equals $z_0$.

The advantage of the nonparametric model studied in this subsection is that it is fully flexible and, in particular, allows for heterogeneous effects of $X$ on $Y$. On the other hand, the nonparametric model suffers from the curse of dimensionality and may result in tests with low power if the researcher has many additional covariates. In this case, it might be better to consider the partially linear model studied below.

To define the test statistic, let $\MS_n$ be some general set that depends on $n$ and is (implicitly) allowed to depend on $\{X_i,Z_i\}$. In addition, let $z:\MS_n\rightarrow S_z$ and $\ell:\MS_n\rightarrow(0,\infty)$ be some functions, and for $s\in\mathcal{S}_n$, let $Q(\cdot,\cdot,\cdot,\cdot,s):\mathbb{R}\times\mathbb{R}\times\mathbb{R}^d\times\mathbb{R}^d\rightarrow \mathbb{R}$ be weighting functions satisfying
$$
Q(x_1,x_2,z_1,z_2,s)=\bar{Q}(x_1,x_2,s)\bar{K}\left(\frac{z_1-z(s)}{\ell(s)}\right)\bar{K}\left(\frac{z_2-z(s)}{\ell(s)}\right)
$$
for all $x_1$, $x_2$, $z_1$, and $z_2$ where the functions $\bar{Q}(\cdot,\cdot,s)$ satisfy $\bar{Q}(x_1,x_2,s)=\bar{Q}(x_2,x_1,s)$ and $\bar{Q}(x_1,x_2,s)\geq 0$ for all $x_1$ and $x_2$. For example, $\bar{Q}(\cdot,\cdot,s)$ can be a kernel weighting function. The functions $Q(\cdot,\cdot,\cdot,\cdot,s)$ are also (implicitly) allowed to depend on $\{X_i,Z_i\}$. Here $\bar{K}:\RR^d\rightarrow\RR$ is some positive compactly supported auxiliary kernel function, and $\ell(s)$, $s\in\MS_n$, are auxiliary bandwidth values. Intuitively, $Q(\cdot,\cdot,\cdot,\cdot,s)$ are local-in-$z(s)$ weighting functions. It is important here that the auxiliary bandwidth values $\ell(s)$ depend on $s$. For example, if kernel weighting functions are used, so that $s=(x,h,z,\ell)$, then one has to choose $h=h(s)$ and $\ell=\ell(s)$ so that $n h\ell^{d+2}=o_p(1/\log p)$ and $1/(n h\ell^d)\leq Cn^{-c}$ w.p.a.1 for some $c,C>0$ uniformly over all $s=(x,h,z,\ell)\in\MS_n$; see discussion below.

Further, let
\begin{equation}\label{eq: test function definition multivariate}
b(s)=b(\{X_i,Z_i,Y_i\},s)=(1/2)\sum_{1\leq i,j\leq n}(Y_i-Y_j)\sign(X_j-X_i)Q(X_i,X_j,Z_i,Z_j,s)
\end{equation}
be a test function. Conditional on $\{X_i,Z_i\}$, the variance of $b(s)$ is given by
\begin{equation}\label{eq: variance definition multivariate}
V(s)=V(\{X_i,Z_i\},\{\sigma_i\},s)=\sum_{1\leq i\leq n}\sigma_i^2\(\sum_{1\leq j\leq n}\sign(X_j-X_i)Q(X_i,X_j,Z_i,Z_j,s)\)^2
\end{equation}
where $\sigma_i=(\Ep[\eps_i^2|X_i,Z_i])^{1/2}$, and estimated variance is
\begin{equation}\label{eq: variance estimated definition multivariate}
\hat{V}(s)=V(\{X_i,Z_i\},\{\hat{\sigma}_i\},s)=\sum_{1\leq i\leq n}\hat{\sigma}_i^2\(\sum_{1\leq j\leq n}\sign(X_j-X_i)Q(X_i,X_j,Z_i,Z_j,s)\)^2
\end{equation}
where $\hat{\sigma}_i$ is some estimator of $\sigma_i$.
Then the test statistic is
$$
T=T(\{X_i,Z_i,Y_i\},\{\hat{\sigma}_i\},\MS_n)=\max_{s\in\MS_n}\frac{b(\{X_i,Z_i,Y_i\},s)}{\sqrt{\hat{V}(\{X_i,Z_i\},\{\hat{\sigma}_i\},s)}}.
$$
Large values of $T$ indicate that $\MH_0$ is violated. The critical value for the test can be calculated using any of the methods described in Section \ref{sec: test}. For example, the plug-in critical value is defined as the conditional $(1-\alpha)$ quantile of $T^\star=T(\{X_i,Z_i,Y_i^\star\},\{\hat{\sigma}_i\},\MS_n)$ given $\{X_i,Z_i\}$, $\{\hat{\sigma}_i\}$, and $\MS_n$ where $Y_i^\star=\hat{\sigma}_i\epsilon_i$ for $i=\overline{1,n}$.

Let $c_{1-\alpha}^{PI}$, $c_{1-\alpha}^{OS}$, and $c_{1-\alpha}^{SD}$ denote the plug-in, one-step, and step-down critical values, respectively.
In addition, let 
\begin{equation}\label{eq: growth condition multivariate}
A_n=\max_{s\in\MS_n}\max_{1\leq i\leq n}\left|\sum_{1\leq j\leq n}\sign(X_j-X_i)Q(X_i,X_j,Z_i,Z_j,s)/(V(s))^{1/2}\right|
\end{equation}
be a sensitivity parameter. Recall that $p=|\MS_n|$.
To prove the result concerning multivariate nonparametric model, I will impose the following condition.
\begin{assumption}\label{as: multivariate nonparametric model}
(i) $\ell(s)\sum_{1\leq i,j\leq n}Q(X_i,X_j,Z_i,Z_j,s)/(V(s))^{1/2}=o_p(1/\sqrt{\log p})$ uniformly over $s\in\MS_n$, 
and (ii) the regression function $f$ has uniformly bounded first order partial derivatives.
\end{assumption}

Discussion of this assumption is given below.
In addition, I will need to modify Assumptions A\ref{as: disturbances}, A\ref{as: function estimation}, and A\ref{as: growth condition}:

\vspace{.2cm}
\noindent
\textbf{A1$'$} \textit{(i) $\Pr(|\varepsilon|\geq u|X,Z)\leq \exp(-u/C_1)$ for all $u>0$ and $\sigma_i\geq c_1$ for all $i=\overline{1,n}$.}
\vspace{.04cm}

\vspace{.2cm}
\noindent
\textbf{A2$'$} \textit{(i) $\hat{\sigma}_i=Y_i-\hat{f}(X_i,Z_i)$ for all $i=\overline{1,n}$ and (ii) $\hat{f}(X_i,Z_i)-f(X_i,Z_i)=o_p(n^{-\kappa_1})$ uniformly over $i=\overline{1,n}$.}
\vspace{.04cm}

\noindent
\textbf{A5$'$} \textit{(i) $A_n(\log(pn))^{7/2}=o_p(1)$, (ii) if A2$'$ holds, then $\log p/n^{(1/4)\wedge\kappa_1\wedge\kappa_3}=o_p(1)$, and if A3 holds, then $\log p/n^{\kappa_2\wedge\kappa_3}=o_p(1)$.}
\vspace{.04cm}

\noindent
Assumption A1$'$ imposes the restriction that $\eps_i$'s have sub-exponential tails, which is stronger than Assumption A\ref{as: disturbances}. It holds, for example, if $\eps_i$'s have normal distribution or $\eps_i$'s are uniformly bounded in absolute value.  Assumption A2$'$ is a simple extension of Assumption A2 to account for the vector of  additional covariates $Z$. Further, assume that $\bar{Q}(\cdot,\cdot,s)$ is a kernel weighting function for all $s\in\MS_n$, so that $s=(x,h,z,\ell)$, the joint density of $X$ and $Z$ is bounded below from zero and from above on its support, and $h\geq ((\log n)^2/(C n))^{1/(d+1)}$ and $\ell\geq ((\log n)^2/(C n))^{1/(d+1)}$ w.p.a.1 for some $C>0$ uniformly over all $s=(x,h,z,\ell)\in\MS_n$. Then it follows as in the proof of Proposition \ref{lem: kernel function}, that A\ref{as: multivariate nonparametric model}-i holds if $n h\ell^{d+2}=o_p(1/\log p)$ uniformly over all $s=(x,h,z,\ell)\in\MS_n$ and that A5$'$-i holds if $1/(n h\ell^d)\leq C n^{-c}$ w.p.a.1 for some $c,C>0$ uniformly over all $s=(x,h,z,\ell)\in\MS_n$.


The key difference between the multivariate case studied in this section and univariate case studied in Section \ref{sec: theory under high level conditions} is that now it is not necessarily the case that  $\Ep[b(s)]\leq0$ under $\MH_0$. The reason is that it can be the case under $\MH_0$ that $f(x_1,z_1)>f(x_2,z_2)$ even if $x_1<x_2$ unless $z_1=z_2$. This yields a bias term in the test statistic. Assumption A\ref{as: multivariate nonparametric model} ensures that this bias is asymptotically negligible relative to the concentration rate of the test statistic. The difficulty, however, is that this assumption contradicts the condition $n A_n^4(\log(pn))^7=o_p(1)$ imposed in Assumption A\ref{as: growth condition} and used in the theory for the case when $Z$ is absent. Indeed, under the assumptions specified in the paragraph above, the condition $n A_n^4(\log(p n))^7=o_p(1)$ requires $1/(nh^2\ell^{2d})=o_p(1)$ uniformly over all $s=(x,h,z,\ell)\in\MS_n$, which is impossible if $n h\ell^{d+2}=o_p(1/\log p)$ uniformly over all $s=(x,h,z,\ell)\in\MS_n$ and $d\geq 2$. To deal with this problem, I have to relax Assumption A\ref{as: growth condition}. This in turn requires imposing stronger conditions on the moments of $\varepsilon$. For these reasons, I replace Assumption A\ref{as: disturbances} by A1$'$.
This allows me to apply a powerful method developed in \cite{ChernozhukovChetverikov12} and replace Assumption A\ref{as: growth condition} by A5$'$.

Let $\mathcal{M}_{NP}$ denote the set of models given by equation (\ref{eq: nonparametric model}), function $f(\cdot)$, joint distribution of $X$, $Z$, and $\varepsilon$ satisfying $\Ep[\eps|X,Z]=0$ almost surely, weighting functions $Q(\cdot,\cdot,\cdot,\cdot,s)$ for $s\in\MS_n$ (possibly depending on $X_i$'s and $Z_i$'s), and estimators $\{\hat{\sigma}_i\}$ such that uniformly over this class, (i) Assumptions A1$'$, A4, A5$'$, and A11 are satisfied (where $V(s)$, $\hat{V}(s)$, and $A_n$ are defined in (\ref{eq: variance definition multivariate}), (\ref{eq: variance estimated definition multivariate}), and (\ref{eq: growth condition multivariate}), respectively), and (ii) either Assumption A2$'$ or A3 is satisfied. The following theorem shows that the test in the multivariate nonparametric model controls asymptotic size.


\begin{theorem}[Size properties in the multivariate nonparametric model]\label{thm: nonparametric model}
Let $P=PI$, $OS$, or $SD$. Let $\mathcal{M}_{NP,0}$ denote the set of all models $M\in\mathcal{M}_{NP}$ satisfying $\MH_0$. Then 
$$
\inf_{M\in\mathcal{M}_{NP,0}}\Pr_M(T\leq c_{1-\alpha}^P)\geq 1-\alpha+o(1)\text{ as }n\rightarrow \infty.
$$
In addition, let $\mathcal{M}_{NP,00}$ denote the set of all models $M\in\mathcal{M}_{NP,0}$ such that $f(x)= C$ for some constant $C$ and all $x\in\RR$. Then
$$
\sup_{M\in\mathcal{M}_{NP,00}}\Pr_M(T\leq c_{1-\alpha}^P)=1-\alpha+o(1)\text{ as }n\rightarrow \infty.
$$
\end{theorem}
\begin{remark}
(i) The result of this theorem is new and I am not aware of any similar or related result in the literature. Here I briefly comment on difficulties arising if one tries to obtain a result like that in Theorem \ref{thm: nonparametric model} by applying proof techniques that were previously used in the literature for the model where $Z$ is absent. The approach of \cite{GSV2000} consists of first providing a Gaussian coupling (strong approximation) for the whole process $\{b(s)/\hat{V}(s):s\in\MS_n\}$ and then employing results of the extreme value theory of Gaussian processes with one-dimensional domain (see, for example, \cite{Leadbetter_book} for a detailed description of these results) to derive a limit distribution of suprema of these processes. When $Z$ is present, however, one has to apply the results of the extreme value theory for Gaussian processes with multi-dimensional domain (these processes are refered in the literature as Gaussian random fields); see \cite{Piterbarg_book}. Although there do exist important applications of this theory in econometrics (see, for example, \cite{LLW2009}), it is not clear how to apply it in the setting like that studied in this section where the covariance structure of the process $\{\{b(s)/\hat{V}(s):s\in\MS_n\}$ is rather complicated, which is the case when kernel weighting functions are used with many different bandwidth values. \cite{HallHeckman2000} take a different approach: they first provide a Gaussian coupling and then use integration by parts of stochastic integrals to show validity of their critical values. Even when $Z$ is absent, they proved their results only when $X_i$'s are equidistant on some interval, but, more importantly, it is also unclear how to generalize their techniques based on integration by parts to multi-dimensional setting.

\noindent
(ii) There are other possible notions of monotonicity in the multivariate model (\ref{eq: nonparametric model}). Assume, for simplicity, that both $X$ and $Z$ are scalars. Then one might want to test the null hypothesis, $\mathcal{H}_0$, that $f(x_2,z_2)\geq f(x_1,z_1)$ for all $x_1,x_2,z_1$, and $z_2$ satisfying $x_2\geq x_1$ and $z_2\geq z_1$ against the alternative, $\mathcal{H}_a$, that there exist $x_1,x_2,z_1$, and $z_2$ such that $x_2\geq x_1$ and $z_2\geq z_1$ but $f(x_2,z_2)<f(x_1,z_1)$. To test this $\mathcal{H}_0$, one can consider the following test functions:
$$
b(s)=(1/2)\sum_{1\leq i,j\leq n}(Y_i-Y_j)\text{sign}(X_j-X_i)I\{Z_j\geq Z_i\}Q(X_i,X_j,Z_i,Z_j,s)
$$
where the function $Q$ satisfies $Q(x_1,x_2,z_1,z_2,s)\geq 0$ for all $x_1,x_2,z_1,z_2$, and $s$. Note that there is no bias term present in the test function, and so one does not have to consider local-in-$z$ weighting functions. The theory for this problem can be obtained along the same line as that used in Section \ref{sec: theory under high level conditions}.

\noindent
(iii) The same techniques that I use to test monotonicity in this paper can also be used to test other hypotheses that are of interest in economic theory. For example, one can test that the function $f(\cdot,\cdot)$ has \textit{increasing differences} or \textit{super-modularity} property. It is said that the function $f(\cdot,\cdot)$ has increasing differences or super-modularity property if for all $x_1<x_2$, $f(x_2,z)-f(x_1,z)$ is increasing in $z$. This property plays an important role in Robust comparative statics. Specifically, one can obtain a testing procedure based on the following test functions:
$$
b(s)=\sum_{1\leq j,k,l,m\leq n}\left((Y_j-Y_k)-(Y_l-Y_m)\right)Q(X_j,X_k,X_l,X_m,Z_j,Z_k,Z_l,Z_m,s)
$$
where the function $Q$ can take the form of kernel weighting function:
\begin{align*}
&Q(x_1,x_2,x_3,x_4,z_1,z_2,z_3,z_4,s)\\
&=K\left(\frac{x_1-x_r(s)}{h(s)}\right)K\left(\frac{x_2-x_l(s)}{h(s)}\right)K\left(\frac{x_3-x_r(s)}{h(s)}\right)K\left(\frac{x_4-x_l(s)}{h(s)}\right)\\
&\times K\left(\frac{z_1-z_r(s)}{h(s)}\right)K\left(\frac{z_2-z_r(s)}{h(s)}\right)K\left(\frac{z_3-z_l(s)}{h(s)}\right)K\left(\frac{z_4-z_l(s)}{h(s)}\right)
\end{align*}
and $s=(x_l,x_r,z_l,z_r,h)$ with $x_r>x_l$ and $z_r>z_l$ and some bandwidth value $h$.
\qed
\end{remark}

\subsection{Partially Linear Model}
In this model, additional covariates enter
the regression function as additively separable linear form. Specifically, the model is given by
\begin{equation}\label{eq: partially linear model}
Y=f(X)+Z^{T}\beta+\eps
\end{equation}
where $Y$ is a scalar dependent random variable, $X$ a scalar independent random variable, $Z$ a vector in $\RR^d$ of additional independent random variables that should be controlled for, $f(\cdot)$ an unknown function, $\beta$ a vector in $\RR^d$ of unknown parameters, and $\eps$ an unobserved scalar random variable satisfying $\Ep[\eps|X,Z]=0$ almost surely. As in Section \ref{sec: theory under high level conditions}, the problem is to test the
null hypothesis, $\MH_{0}$, that $f(\cdot)$ is nondecreasing against the
alternative, $\MH_{a}$, that there are $x_1$ and $x_2$ such that
$x_1<x_2$ but $f(x_1)>f(x_2)$. The decision is to be made based on the i.i.d. sample of size $n$, $\{X_i,Z_i,Y_i\}_{1\leq i\leq n}$ from the distribution of $(X,Z,Y)$. The regression function $f(\cdot)$ is assumed to be smooth but I do not impose any parametric structure on it.

An advantage of the partially linear model over the
fully nonparametric model is that it does not suffer from the curse
of dimensionality, which decreases the power of the test and may be a
severe problem if the researcher has many additional covariates to
control for. On the other hand, the partially linear model does not
allow for heterogeneous effects of the factor $X$, which might be
restrictive in some applications. It is necessary to have in mind
that the test obtained for the partially linear model will be inconsistent
if this model is misspecified.

Let me now describe the test. The main idea is to estimate $\beta$ by some $\hat{\beta}_n$ and to apply the methods described in Section \ref{sec: test} for the dataset $\{X_i,Y_i-Z_i^T\hat{\beta}_n\}$.
For example,
one can take an estimator of \cite{Robinson88}, which is given by
$$
\hat{\beta}_n=\left(\sum_{i=1}^{n}\hat{Z}_{i}\hat{Z}_{i}^{T}\right)^{-1}\left(\sum_{i=1}^{n}\hat{Z}_{i}\hat{Y}_{i}\right)
$$
where $\hat{Z}_{i}=Z_{i}-\hat{E}[Z|X=X_{i}]$, $\hat{Y}_{i}=Y_{i}-\hat{E}[Y|X=X_{i}]$,
and $\hat{E}[Z|X=X_{i}]$ and $\hat{E}[Y|X=X_{i}]$ are nonparametric estimators of $E[Z|X=X_{i}]$ and $E[Y|X=X_{i}]$,
respectively. Define $\tilde{Y}_{i}=Y_{i}-Z_{i}^T\hat{\beta}_n$, and let the test statistic be $T=T(\{X_i,\tilde{Y}_i\},\{\hat{\sigma}_i\},\MS_n)$ where $T(\cdot,\cdot,\cdot)$ is defined in (\ref{eq: test statistic definition}) and estimators $\hat{\sigma}_i$ of $\sigma_i=(\Ep[\eps_i^2|X_i])^{1/2}$ (here $\eps_i=Y_i-f(X_i)-Z_i^T\beta$) satisfy either $\hat{\sigma}_i=\hat{\eps}_i=Y_i-\hat{f}(X_i)-Z_i^T\hat{\beta}_n$ (here $\hat{f}(X_i)$ is some estimator of $f(X_i)$, which is uniformly consistent over $i=\overline{1,n}$) or $\hat{\sigma}_i$ is some uniformly consistent estimator of $\sigma_i$. The critical value for the test is simulated by one of the methods (plug-in, one-step, or step-down) described in Section \ref{sec: test} using the data $\{X_i,\tilde{Y}_i\}$, estimators $\{\hat{\sigma}_i\}$, and the set of weighting functions $\MS_n$. As in Section \ref{sec: test}, let $c_{1-\alpha}^{PI}$, $c_{1-\alpha}^{OS}$, and $c_{1-\alpha}^{SD}$ denote the plug-in, one-step, and step-down critical values, respectively.

Let $C_5>0$ be some constant. To obtain results for partially linear models, I will impose the following condition.
\begin{assumption}\label{as: partially linear model}
(i) $\Vert Z\Vert\leq C_5$ almost surely, (ii) $\|\hat{\beta}_n-\beta\|=O_p(n^{-1/2})$, and (iii) uniformly over all $s\in\MS_n$, $\sum_{1\leq i,j\leq n}Q(X_i,X_j,s)/V(s)^{1/2}=o_p(\sqrt{n/\log p})$.
\end{assumption}
\noindent
The assumption that the vector $Z$ is bounded is a mild regularity condition and can be easily relaxed. Further,
\cite{Horowitz_book} provides a set of conditions so that $\|\hat{\beta}_n-\beta\|=O_p(n^{-1/2})$ when $\hat{\beta}_n$ is the Robinson's estimator. Finally, it follows from the proof of Proposition \ref{lem: kernel function} that under Assumption A\ref{as: design difficult}, Assumption A\ref{as: partially linear model}-iii is satisfied if kernel weighting functions are used as long as $h_{\max}=o_p(1/\log p)$ and $h_{\min}\geq (\log n)^2/(Cn)$ w.p.a.1. for some $C>0$.
In addition, I will need to modify Assumption A\ref{as: function estimation} to account for the presence of the vector of additional covariates $Z$:

\vspace{.2cm}
\noindent
\textbf{A2$''$} \textit{(i) $\hat{\sigma}_i=Y_i-\hat{f}(X_i)-Z_i^T\hat{\beta}_n$ for all $i=\overline{1,n}$ and (ii) $\hat{f}(X_i)-f(X_i)=o_p(n^{-\kappa_1})$ uniformly over $i=\overline{1,n}$.}
\vspace{.04cm}

Let $\mathcal{M}_{PL}$ be a class of models given by equation (\ref{eq: partially linear model}), function $f(\cdot)$, parameter $\beta$, joint distribution of $X$, $Z$, and $\varepsilon$ satisfying $\Ep[\varepsilon|X,Z]=0$ almost surely, weighting functions $Q(\cdot,\cdot,s)$ for $s\in\MS_n$ (possibly depending on $X_i$'s), an estimator $\hat{\beta}_n$, and estimators $\{\hat{\sigma}_i\}$ such that uniformly over this class, (i) Assumptions A\ref{as: disturbances}, A\ref{as: V}, A\ref{as: growth condition}, and A\ref{as: partially linear model} are satisfied (where $V(s)$, $\hat{V}(s)$, and $A_n$ are defined in (\ref{eq: variance definition}), (\ref{eq: variance estimated definition}), and (\ref{eq: sensitivity parameter}), respectively), and (ii) either Assumption A\ref{as: function estimation}$''$ or A\ref{as: sigma} is satisfied.
The size properties of the test are given in the following theorem.
\begin{theorem}[Size properties in the partially linear model]\label{thm: partially linear model}
Let $P=PI$, $OS$, or $SD$. Let $\mathcal{M}_{PL,0}$ denote the set of all models $M\in\mathcal{M}_{PL}$ satisfying $\MH_0$. Then
$$
\inf_{M\in\mathcal{M}_{PL,0}}\Pr_M(T\leq c_{1-\alpha}^P)\geq 1-\alpha+o(1)\text{ as }n\rightarrow \infty.
$$
In addition, let $\mathcal{M}_{PL,00}$ denote the set of all models $M\in\mathcal{M}_{PL,0}$ such that $f(x)=C$ for some constant $C$ and all $x\in\RR$. Then
$$
\sup_{M\in\mathcal{M}_{PL,00}}\Pr_M(T\leq c_{1-\alpha}^P)=1-\alpha+o(1)\text{ as }n\rightarrow \infty.
$$
\end{theorem}
\begin{remark}\label{com: separately additive model}
The result of Theorem \ref{thm: partially linear model} can be extended to cover a general separately additive models. Specifically, assume that the data come from the model
\begin{equation}\label{eq: separately additive model}
Y=f(X)+g(Z)+\varepsilon
\end{equation}
where $g(\cdot)$ is some unknown smooth function and all other notation is the same as above. Then one can test $\MH_0$ that $f(\cdot)$ is nondecreasing against the alternative $\MH_a$ that there are $x_1$ and $x_2$ such that $x_1<x_2$ but $f(x_1)>f(x_2)$ using a method similar to that used above with the exception that now $\tilde{Y}_i=Y_i-\hat{g}(Z_i)$ where $\hat{g}(Z_i)$ is some nonparametric estimator of $g(Z_i)$. 
Specifically, consider the following conditions:

\vspace{.2cm}
\noindent
\textbf{A2$'''$} \textit{(i) $\hat{\sigma}_i=Y_i-\hat{f}(X_i)-\hat{g}(Z_i)$ for all $i=\overline{1,n}$ and (ii) $\hat{f}(X_i)-f(X_i)=o_p(n^{-\kappa_1})$ uniformly over $i=\overline{1,n}$.}
\vspace{.04cm}

\vspace{.2cm}
\noindent
\textbf{A12$'$} \textit{(i) $\max_{1\leq i\leq n}|\hat{g}(Z_i)-g(Z_i)|=O_p(\psi_n^{-1})$ for some sequence $\psi_n\rightarrow \infty$, and (ii) uniformly over all $s\in\MS_n$, $\sum_{1\leq i,j\leq n}Q(X_i,X_j,s)/V(s)^{1/2}=o_p(\sqrt{\psi_n/\log p})$.}
\vspace{.04cm}

Let $\mathcal{M}_{SA}$, $\mathcal{M}_{SA,0}$, and $\mathcal{M}_{SA,00}$ denote sets of models that are defined as $\mathcal{M}_{PL}$, $\mathcal{M}_{PL,0}$, and $\mathcal{M}_{PL,00}$ with (\ref{eq: partially linear model}) and Assumptions A\ref{as: partially linear model} and A2$''$ replaced by (\ref{eq: separately additive model}) and Assumptions A12$'$ and A2$'''$.
Then Theorem \ref{thm: partially linear model} applies with $\mathcal{M}_{PL}$, $\mathcal{M}_{PL,0}$, and $\mathcal{M}_{PL,00}$ replaced by $\mathcal{M}_{SA}$, $\mathcal{M}_{SA,0}$, and $\mathcal{M}_{SA,00}$. This result can be proven from the argument similar to that used to prove Theorem \ref{thm: partially linear model}.\qed
\end{remark}

\section{Models with Endogenous Covariates}\label{sec: endogenous covariates}
Empirical studies in economics often contain endogenous covariates. Therefore, it is important to extend the results on testing monotonicity obtained in this paper to cover this possibility. Allowing for endogenous covariates in nonparametric settings like the one considered here is challenging and several approaches have been proposed in the literature. In this paper, I consider the approach proposed in \cite{NPV99}. In their model,
\begin{align}
&Y=f(X)+W,\label{eq: endogenous model 1}\\
&X=\lambda(U)+Z\label{eq: endogenous model 2}
\end{align}
where $Y$ is a scalar dependent variable, $X$ a scalar covariate, $U$ a vector in $\mathbb{R}^d$ of instruments, $f(\cdot)$ and $\lambda(\cdot)$ unknown functions, and $W$ and $Z$ unobserved scalar random variables satisfying
\begin{equation}\label{eq: endogenous model 3}
\Ep[W|U,Z]=\Ep[W|Z]\text{ and }\Ep[Z|U]=0
\end{equation}
almost surely. In this setting, $X$ is endogenous in the sense that equation $\Ep[W|X]=0$ almost surely does not necessarily hold. The problem is to test the null hypothesis, $\MH_0$, that $f(\cdot)$ is nondecreasing against the alternative, $\MH_a$, that there are $x_1$ and $x_2$ such that $x_1<x_2$ but $f(x_1)>f(x_2)$. The decision is to be made based on the i.i.d. sample of size $n$, $\{X_i,U_i,Y_i\}_{1\leq i\leq n}$ from the distribution of $(X,U,Y)$.

It is possible to consider a more general setting where the function $f$ of interest depends on $X$ and $U_1$, that is $Y=f(X,U_1)+\varepsilon$ and $U=(U_1,U_2)$ but I refrain from including additional covariates for brevity. 

An alternative to the model defined in (\ref{eq: endogenous model 1})-(\ref{eq: endogenous model 3}) is the model defined by (\ref{eq: endogenous model 1}) and
\begin{equation}\label{eq: endogenous model alternative}
\Ep[W|U]=0
\end{equation}
almost surely. \cite{NPV99} noted that neither model is more general than the other one. The reason is that it does not follow from (\ref{eq: endogenous model 3}) that (\ref{eq: endogenous model alternative}) holds, and it does not follow from (\ref{eq: endogenous model alternative}) that (\ref{eq: endogenous model 3}) holds. Both models have been used in the empirical studies. The latter model have been studied in \cite{NP03}, \cite{HH05}, \cite{BCK07}, and \cite{DFR12}, among many others.

Let me now get back to the model defined by (\ref{eq: endogenous model 1})-(\ref{eq: endogenous model 3}). A key observation of \cite{NPV99} is that in this model,
$$
\Ep[Y|U,Z]=\Ep[f(X)|U,Z]+\Ep[W|U,Z]=\Ep[f(X)|U,X]+\Ep[W|Z]=f(X)+g(Z)
$$
where $g(z)=\Ep[W|Z=z]$. This implies that
\begin{equation}\label{eq: additively separable representation}
\Ep[Y|X,Z]=f(X)+g(Z).
\end{equation}
Therefore, the regression function of $Y$ on $X$ and $Z$ is separately additive in $X$ and $Z$, and so, if $Z_i=X_i-\lambda(U_i)$'s were known, both $f(\cdot)$ and $g(\cdot)$ could be estimated by one of the nonparametric methods suitable for estimating separately additive models. One particularly convenient method to estimate such models is a series estimator. Further, note that even though $Z_i$'s are unknown, they can be consistently estimated from the data as residuals from the nonparametric regression of $X$ on $U$. Then one can estimate both $f(\cdot)$ and $g(\cdot)$ by employing a nonparametric method for estimating separately additive models with $Z_i$'s replaced by their estimates.  Specifically, let $\hat{Z_i}=X_i-\hat{\Ep}[X|U=U_i]$ where $\hat{\Ep}[X|U=U_i]$ is a consistent nonparametric estimator of $\Ep[X|U=U_i]$. Further, for an integer $L>0$, let $r^{L,f}(x)=(r_{1L}^f(x),...,r_{LL}^f(x))'$ and $r^{L,g}(z)=(r_{1L}^g(z),...,r_{LL}^g(z))'$ be vectors of approximating functions for $f(\cdot)$ and $g(\cdot)$, respectively, so that
$$
f(x)\approx r^{L,f}(x)'\beta_f\text{ and }g(z)\approx r^{L,g}(z)'\beta_g
$$
for $L$ large enough where $\beta_f$ and $\beta_g$ are vectors in $\mathbb{R}^L$ of coefficients. In addition, let $r^L(x,z)=(r^{L,f}(x)',r^{L,g}(z)')'$ and $R=(r^L(X_1,\hat{Z}_1),...,r^L(X_n,\hat{Z}_n))$. Then the OLS estimator of $\beta_f$ and $\beta_g$ is
$$
\left(\begin{array}{c}
\hat{\beta}_{f}\\
\hat{\beta}_{g}
\end{array}\right)=\left(R'R\right)^{-1}\left(R'Y_1^n\right)
$$
where $Y_1^n=(Y_1,...,Y_n)'$, and the series estimators $\hat{f}(x)$ and $\hat{g}(x)$ of $f(x)$ and $g(z)$ are $r^{L,f}(x)'\hat{\beta}_f$ and $r^{L,g}(z)'\hat{\beta}_g$ for all $x$ and $z$, respectively.

Now, it follows from (\ref{eq: additively separable representation}) that
$$
Y=f(X)+g(Z)+\varepsilon
$$
where $\Ep[\varepsilon|X,Z]=0$ almost surely. This is exactly the model discussed in Comment \ref{com: separately additive model}, and since I have an estimator of $g(\cdot)$, it is possible to test $\MH_0$ using the same ideas as in Comment \ref{com: separately additive model}. Specifically, let $\tilde{Y}_i=Y-\hat{g}(\hat{Z}_i)$ and apply a test described in Section \ref{sec: test} with the data $\{X_i,\tilde{Y}_i\}_{1\leq i\leq n}$; that is, let the test statistic be $T=T(\{X_i,\tilde{Y}_i\},\{\hat{\sigma}_i\},\MS_n)$ where the function $T(\cdot,\cdot,\cdot)$ is defined in (\ref{eq: test statistic definition}) and estimators $\hat{\sigma}_i$ of $\sigma_i=(\Ep[\varepsilon_i^2|X_i])^{1/2}$ (here $\eps_i=Y_i-f(X_i)-g(Z_i)$) satisfy either $\hat{\sigma}_i=\hat{\eps}_i=Y_i-\hat{f}(X_i)-\hat{g}(\hat{Z}_i)$ or $\hat{\sigma}_i$ is some uniformly consistent estimator of $\sigma_i$.  Let $c_{1-\alpha}^{PI}$, $c_{1-\alpha}^{OS}$, and $c_{1-\alpha}^{SD}$ denote the plug-in, one-step, and step-down critical values, respectively, that are obtained as in Section \ref{sec: test} using the data $\{X_i,\tilde{Y}_i\}$, estimators $\{\hat{\sigma}_i\}$, and the set of weighting functions $\MS_n$.

To obtain results for this testing procedure, I will use the following modifications of Assumptions A2$'''$ and A12$'$:

\vspace{.2cm}
\noindent
\textbf{A2$''''$} \textit{(i) $\hat{\sigma}_i=Y_i-\hat{f}(X_i)-\hat{g}(\hat{Z_i})$ for all $i=\overline{1,n}$ and (ii) $\hat{f}(X_i)-f(X_i)=o_p(n^{-\kappa_1})$ uniformly over $i=\overline{1,n}$.}
\vspace{.04cm}

\vspace{.2cm}
\noindent
\textbf{A12$''$} \textit{(i) $\max_{1\leq i\leq n}|\hat{g}(\hat{Z_i})-g(\hat{Z_i})|=O_p(\psi_n^{-1})$ for some sequence $\psi_n\rightarrow \infty$, and (ii) uniformly over all $s\in\MS_n$, $\sum_{1\leq i,j\leq n}Q(X_i,X_j,s)/V(s)^{1/2}=o_p(\sqrt{\psi_n/\log p})$.}
\vspace{.04cm}

Assumption A2$''''$ is a simple extension of Assumption A2$'''$ that takes into account that $Z_i$'s have to be estimated by $\hat{Z}_i$'s. Further, Assumption A12$''$ requires that $\max_{1\leq i\leq n}|\hat{g}(\hat{Z}_i)-g(\hat{Z}_i)|=O_p(\psi_n^{-1})$ for some sequence $\psi_n\rightarrow \infty$. For the estimator $\hat{g}(\cdot)$ of $g(\cdot)$ described above, this requirement follows from Lemma 4.1 and Theorem 4.3 in \cite{NPV99} who also provide certain primitive conditions for this requirement to hold. Specific choices of $\psi_n$ depend on how smooth the functions $f(\cdot)$ and $g(\cdot)$ are and how the number of series terms, $L$, is chosen.

Let $\mathcal{M}_{EC}$ be a class of models given by equations (\ref{eq: endogenous model 1})-(\ref{eq: endogenous model 3}), functions $f(\cdot)$ and $\lambda(\cdot)$, joint distribution of $X$, $U$, $W$, and $Z$, weighting functions $Q(\cdot,\cdot,s)$ for $s\in\MS_n$ (possibly depending on $X_i$'s), an estimator $\hat{g}(\cdot)$ of $g(\cdot)$, and estimators $\{\hat{\sigma}_i\}$ of $\{\sigma_i\}$ such that uniformly over this class, (i) Assumptions A\ref{as: disturbances}, A\ref{as: V}, A\ref{as: growth condition}, and A12$''$ are satisfied (where $V(s)$, $\hat{V}(s)$, and $A_n$ are defined in (\ref{eq: variance definition}), (\ref{eq: variance estimated definition}), and (\ref{eq: sensitivity parameter}), respectively), and (ii) either Assumption A2$''''$ or A3 is satisfied. 
 
\begin{theorem}[Size properties in the model with endogenous covariates]\label{thm: model with endogenous covariates}
Let $P=PI$, $OS$, or $SD$. Let $\mathcal{M}_{EC,0}$ denote the set of all models $M\in\mathcal{M}_{EC}$ satisfying $\MH_0$. Then
$$
\inf_{M\in\mathcal{M}_{EC,0}}\Pr_M(T\leq c_{1-\alpha}^P)\geq 1-\alpha+o(1)\text{ as }n\rightarrow \infty.
$$
In addition, let $\mathcal{M}_{EC,00}$ denote the set of all models $M\in\mathcal{M}_{EC,0}$ such that $f\equiv C$ for some constant $C$. Then
$$
\sup_{M\in\mathcal{M}_{EC,00}}\Pr_M(T\leq c_{1-\alpha}^P)=1-\alpha+o(1)\text{ as }n\rightarrow \infty.
$$
\end{theorem}
\begin{remark}
Around the same time this paper became publicly available, \cite{Gutknecht2013} obtained a test of monotonicity of the function $f(\cdot)$ in the same model (\ref{eq: endogenous model 1})-(\ref{eq: endogenous model 3}). The test in that paper is a special case of the class of tests derived in this paper. Specifically, \cite{Gutknecht2013} obtained a test based on test functions (\ref{eq: GSV modification}). The major difference, however, is that his test is not adaptive because it is based on one non-stochastic bandwidth value.\qed
\end{remark}

\section{Sample Selection Models}\label{sec: sample selection models}

It is widely recognized in econometrics literature that sample selection problems can result in highly misleading inference in otherwise standard regression models. The purpose of this section is to briefly show that the same techniques that are used in the last section to deal with endogeneity issues can also be used to deal with sample selection issues. For concreteness, I consider a nonparametric version of the classical Heckman's sample selection model; see \cite{Heckman79}. The nonparametric version of the Heckman's model was previously analyzed in great generality in \cite{DNV03}. Specifically, I consider the model
\begin{align}
&Y^{*}=f(X)+W,\label{eq: sample selection 1}\\
&Y=DY^{*}\label{eq: sample selection 2}
\end{align}
where $Y^{*}$ is an unobserved scalar dependent random variable, $Y$ a scalar random variable, $X$ a covariate, $D$ a binary selection indicator, $W$ unobserved scalar random variable. The problem in this model arises when $W$ is not independent of $D$. To deal with this problem, let $Z$ denote a vector of random variables that affect selection $D$, and let $p(x,z)=\Ep[D|X=x,Z=z]$ be the propensity score. Further, assume that $Z$ is such that
\begin{equation}\label{eq: sample selection 3}
\Ep[W|X,Z,D=1]=\lambda(p(X,Z))
\end{equation}
where $\lambda(\cdot)$ is some unknown function. This condition is reasonable in many settings; see, in particular, discussion on p.35 of \cite{DNV03}.
I am interested in testing the null hypothesis that the function $f(\cdot)$ in the model (\ref{eq: sample selection 1})-(\ref{eq: sample selection 3}) is weakly increasing against the alternative that it is not weakly increasing based on the random sample $\{X_i,Z_i,Y_i\}_{1\leq i\leq n}$ from the distribution of $(X,Z,Y)$. 

A key observation of \cite{DNV03} is that under (\ref{eq: sample selection 3}), equations (\ref{eq: sample selection 1}) and (\ref{eq: sample selection 2}) imply that
$$
\Ep[Y|X,Z,D=1]=f(X)+\lambda(p(X,Z)),
$$
and so denoting $P=p(X,Z)$,
\begin{equation}\label{eq: sample selection 4}
\Ep[Y|X,P,D=1]=f(X)+\lambda(P).
\end{equation}
Therefore, if $P_i=p(X_i,Z_i)$'s were known, one could estimate both $f(\cdot)$ and $\lambda(\cdot)$ by one of the nonparametric methods suitable for estimating separately additive models applied to the regression of $Y$ on $X$ and $P$ based on the subsample of observations with $D_i=1$. Further, note that even though $P_i$'s are unknown, they can be estimated consistently from the data as the predicted values in the nonparametric regression of $D$ on $X$ and $Z$. Hence, one can estimate both $f(\cdot)$ and $\lambda(\cdot)$ by one of the non-parametric methods suitable for estimating additive models applied to the regression of $Y$ on $X$ and $P$ with $P_i$'s replaced by $\hat{P}_i$'s where $\hat{P}_i=\hat{p}(X_i,Z_i)$ and $\hat{p}(\cdot,\cdot)$ is a nonparametric estimator of $p(\cdot,\cdot)$.

Now, (\ref{eq: sample selection 4}) implies that
$$
Y=f(X)+\lambda(P)+\varepsilon
$$
where $\Ep[\varepsilon|X,P,D=1]=0$ almost surely. This is exactly the model discussed in Comment \ref{com: separately additive model}, and since an estimator $\hat{\lambda}(\cdot)$ of $\lambda(\cdot)$ is available, one can test monotonicity of $f(\cdot)$ by applying the results in Section \ref{sec: test} with the data $\{X_i,\tilde{Y}_i\}_{1\leq i\leq n}$ where $\tilde{Y}_i=Y-\hat{\lambda}(\hat{P}_i)$. For this test, one can state the theorem that is analogous to Theorem \ref{thm: model with endogenous covariates}.

\section{Monte Carlo Simulations}\label{sec: monte carlo}

In this section, I provide results of a small simulation study. The
aim of the simulation study is to shed some light on the size properties of the
test in finite samples and to compare its power with that of other tests
developed in the literature. In particular, I consider the tests of
\cite{Gijbels2000} (GHJK), \cite{GSV2000} (GSV), and \cite{HallHeckman2000} (HH).

I consider samples of size $n=100$, $200$, and $500$
with $X_{i}$'s uniformly distributed on the $[-1,1]$ interval, and regression functions of the form $f(x)=c_1x-c_2\phi(c_3x)$ where $c_1,c_2,c_3\geq 0$ and $\phi(\cdot)$ is the pdf of the standard normal distribution. I assume that $\eps_i$ is a zero-mean random variable that is independent of $X_i$ and has the standard deviation $\sigma$. Depending on the experiment, $\eps_i$ has either normal or continuous uniform distribution. Four combinations of parameters are studied: (1) $c_1=c_2=c_3=0$ and $\sigma=0.05$; (2) $c_1=c_3=1$, $c_2=4$, and $\sigma=0.05$; (3) $c_1=1$, $c_2=1.2$, $c_3=5$, and $\sigma=0.05$; (4) $c_1=1$, $c_2=1.5$, $c_3=4$, and $\sigma=0.1$. Cases 1 and 2 satisfy $\MH_0$ whereas cases 3 and 4 do not. In case 1, the regression function is flat corresponding to the maximum of the type I error. In case 2, the regression function is strictly increasing. Cases 3 and 4 give examples of the regression functions that are mostly increasing but violate $\MH_0$ in the small neighborhood near 0. All functions are plotted in figure 2. The parameters were chosen so that to have nontrivial rejection probability in most cases (that is, bounded from zero and from one).

\begin{figure}
\caption{Regression Functions Used in Simulations}
\includegraphics[scale=0.5]{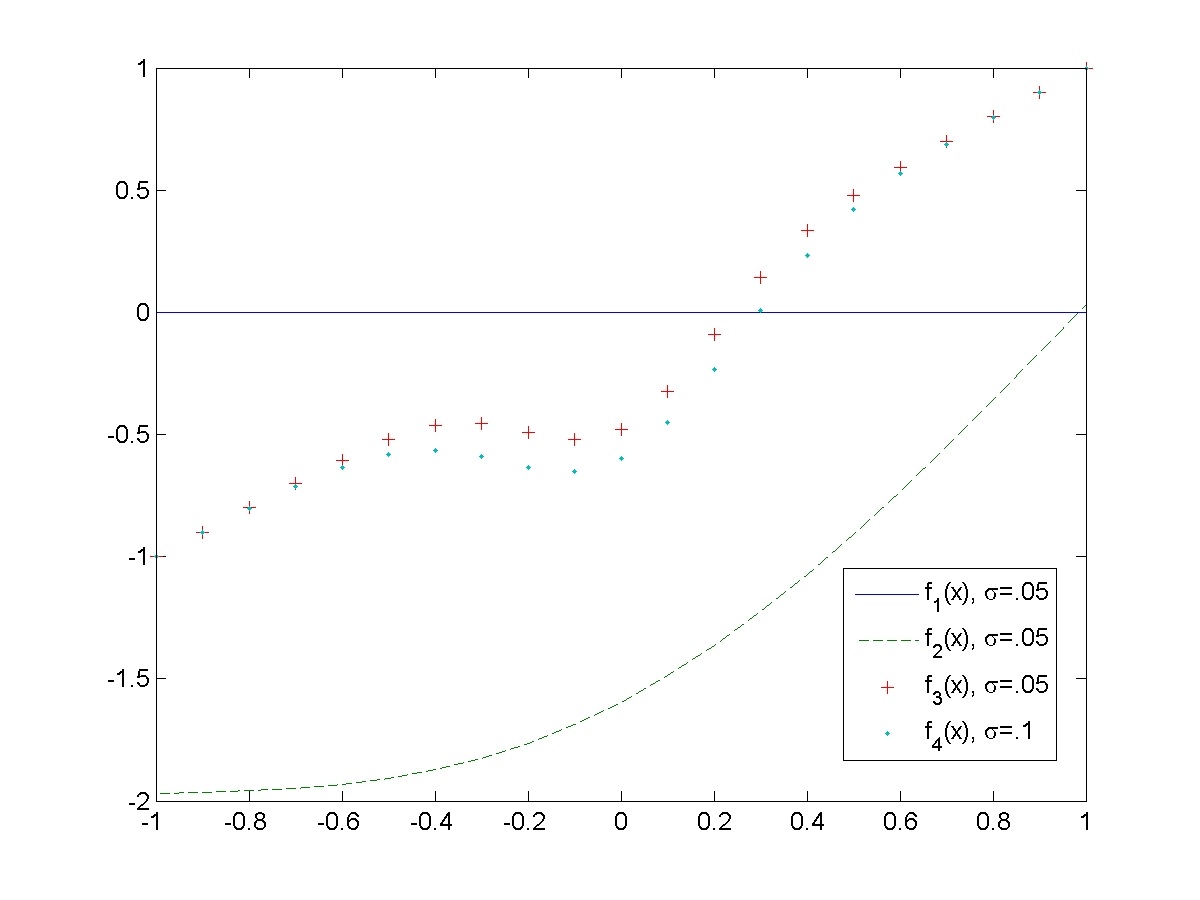}
\end{figure}

Let me describe the tuning parameters for all tests that are used
in the simulations. For the tests of GSV, GHJK,
and HH, I tried to follow their instructions
as closely as possible. For the test developed in this paper, I use kernel weighting functions with $k=0$, $\MS_n=\{(x,h):x\in\{X_1,...,X_n\},h\in H_n\}$, and the kernel $K(x)=0.75(1-x^{2})$
for $x\in(-1;+1)$ and $0$ otherwise. I use the set of
bandwidth values $H_n=\{h_{\max}u^l:h\geq h_{\min},l=0,1,2,...\}$, $u=0.5$, $h_{\max}=1$, $h_{\min}=0.4h_{\max}(\log n/n)^{1/3}$, and the truncation parameter $\gamma=0.01$. 
For the test of GSV, I use the same kernel $K$ with the bandwidth value $h_{n}=n^{-1/5}$, which was suggested in their paper, and I consider their sup-statistic. For the test of GHJK,
I use their run statistic maximized over $k\in\{10(j-1)+1:\, j=1,2,...0.2n\}$ (see the original paper for the explanation of the notation).
For the test of HH, local polynomial estimates
are calculated over $r\in nH_n$ at every design point $X_i$.
The set $nH_n$ is chosen so that to make the results comparable with those
for the test developed in this paper. Finally, I consider two versions of the test developed in this paper depending on how $\sigma_i$ is estimated. More precisely, I consider the test with $\sigma_i$ estimated by the Rice's method (see equation (\ref{eq: Rice formula})), which I refer to in the table below as CS (consistent sigma), and the test with $\hat{\sigma}_i=\hat{\eps}_i$ where $\hat{\eps}_i$ is obtained as the residual from estimating $f$ using the series method with polynomials of order 5, 6 and 8 whenever the sample size $n$ is 100, 200, and 500, respectively, which I refer to in the table below as IS (inconsistent sigma).

The rejection probabilities corresponding to nominal level $\alpha=0.1$
for all tests are presented in table 1. The results are based on 1000
simulations with 500 bootstrap repetitions in all cases excluding
the test of GSV where the asymptotic critical value is
used.

\begin{table}
\caption{Results of Monte Carlo Experiments}
\begin{tabular}{cccccccccccc}
\hline 
\multirow{2}{*}{{\small Noise}} & \multirow{2}{*}{{\small Case}} & \multirow{2}{*}{{\small Sample}} & \multicolumn{9}{c}{{\small Proportion of Rejections for}}\tabularnewline
\cline{4-12} 
&&& {\small GSV} & {\small GHJK} & {\small HH} & {\small CS-PI} & {\small CS-OS} & {\small CS-SD} & {\small IS-PI} & {\small IS-OS} & {\small IS-SD} \tabularnewline
\hline 
\multicolumn{1}{c}{} && {\small 100} & {\small .118} & {\small .078} & {\small .123} & {\small .128} & {\small .128} & {\small .128} & {\small .164} & {\small .164} & {\small .164}\tabularnewline
{\small normal}&{\small 1} & {\small 200} & {\small .091} & {\small .051} & {\small .108} & {\small .114} & {\small .114} & {\small .114} & {\small .149} & {\small .149} & {\small .149}\tabularnewline
&& {\small 500} & {\small .086} & {\small .078} & {\small .105} & {\small .114} & {\small .114} & {\small .114} & {\small .133} & {\small .133} & {\small .133}\tabularnewline
\hline
&& {\small 100} & {\small 0} & {\small .001} & {\small 0} & {\small .001} & {\small .008} & {\small .008} & {\small .008} & {\small .024} & {\small .024}\tabularnewline
{\small normal}&{\small 2}& {\small 200} & {\small 0} & {\small .002} & {\small 0} & {\small .001} & {\small .010} & {\small .010} & {\small .007} & {\small .017} & {\small .017}\tabularnewline
&& {\small 500} & {\small 0} & {\small .001} & {\small 0} & {\small .002} & {\small .007} & {\small .007} & {\small .005} & {\small .016} & {\small .016}\tabularnewline
\hline 
&& {\small 100} & {\small 0} & {\small .148} & {\small .033} & {\small .259} & {\small .436} & {\small .433} & {\small 0} & {\small 0} & {\small 0}\tabularnewline
{\small normal}&{\small 3}& {\small 200} & {\small .010} & {\small .284} & {\small .169} & {\small .665} & {\small .855} & {\small .861} & {\small .308} & {\small .633} & {\small .650}\tabularnewline
&& {\small 500} & {\small .841} & {\small .654} & {\small .947} & {\small .982} & {\small .995} & {\small .997} & {\small .975} & {\small .995} & {\small .995}\tabularnewline
\hline 
&& {\small 100} & {\small .037} & {\small .084} & {\small .135} & {\small .163} & {\small .220} & {\small .223} & {\small .023} & {\small .042} & {\small .043}\tabularnewline
{\small normal}&{\small 4}& {\small 200} & {\small .254} & {\small .133} & {\small .347} & {\small .373} & {\small .499} & {\small .506} & {\small .362} & {\small .499} & {\small .500}\tabularnewline
&& {\small 500} & {\small .810} & {\small .290} & {\small .789} & {\small .776} & {\small .825} & {\small .826} & {\small .771} & {\small .822} & {\small .822}\tabularnewline
\hline 
\multicolumn{1}{c}{} && {\small 100} & {\small .109} & {\small .079} & {\small .121} & {\small .122} & {\small .122} & {\small .122} & {\small .201} & {\small .201} & {\small .201}\tabularnewline
{\small uniform}&{\small 1} & {\small 200} & {\small .097} & {\small .063} & {\small .109} & {\small .121} & {\small .121} & {\small .121} & {\small .160} & {\small .160} & {\small .160}\tabularnewline
&& {\small 500} & {\small .077} & {\small .084} & {\small .107} & {\small .092} & {\small .092} & {\small .092} & {\small .117} & {\small .117} & {\small .117}\tabularnewline
\hline
&& {\small 100} & {\small .001} & {\small .001} & {\small 0} & {\small 0} & {\small .006} & {\small .007} & {\small .017} & {\small .032} & {\small .033}\tabularnewline
{\small uniform}&{\small 2}& {\small 200} & {\small 0} & {\small 0} & {\small 0} & {\small .001} & {\small .010} & {\small .010} & {\small .012} & {\small .022} & {\small .024}\tabularnewline
&& {\small 500} & {\small 0} & {\small .003} & {\small 0} & {\small .003} & {\small .011} & {\small .011} & {\small .011} & {\small .021} & {\small .021}\tabularnewline
\hline 
&& {\small 100} & {\small 0} & {\small .151} & {\small .038} & {\small .244} & {\small .438} & {\small .449} & {\small 0} & {\small 0} & {\small 0}\tabularnewline
{\small uniform}&{\small 3}& {\small 200} & {\small .009} & {\small .233} & {\small .140} & {\small .637} & {\small .822} & {\small .839} & {\small .290} & {\small .607} & {\small .617}\tabularnewline
&& {\small 500} & {\small .811} & {\small .582} & {\small .947} & {\small .978} & {\small .994} & {\small .994} & {\small .975} & {\small .990} & {\small .990}\tabularnewline
\hline 
&& {\small 100} & {\small .034} & {\small .084} & {\small .137} & {\small .155} & {\small .215} & {\small .217} & {\small .024} & {\small .045} & {\small .046}\tabularnewline
{\small uniform}&{\small 4}& {\small 200} & {\small .197} & {\small .116} & {\small .326} & {\small .357} & {\small .473} & {\small .478} & {\small .323} & {\small .452} & {\small .456}\tabularnewline
&& {\small 500} & {\small .803} & {\small .265} & {\small .789} & {\small .785} & {\small .844} & {\small .846} & {\small .782} & {\small .847} & {\small .848}\tabularnewline
\hline 
\end{tabular}

{\small Nominal Size is 0.1. GSV, GHJK, and HH stand for the tests of \cite{GSV2000},
\cite{Gijbels2000}, and \cite{HallHeckman2000} respectively. CS-PI, CS-OS, and CS-SD refer to the test developed in this paper with $\sigma_i$ estimated using Rice's formula and plug-in, one-step, and step-down critical values respectively. Finally, IS-PI, IS-OS, and IS-SD refer to the test developed in this paper with $\sigma_i$ estimated by $\hat{\sigma}_i=\hat{\eps}_i$ and plug-in, one-step, and step-down critical values respectively.}
\end{table}

The results of the simulations can be summarized as follows. First,
the results for normal and uniform disturbances are rather similar.
The test developed in this paper with $\sigma_i$ estimated using the Rice's method maintains the required size quite well (given the nonparametric structure of the problem) and yields size comparable with that of the GSV, GHJK, and HH tests. On the other hand, the test with $\hat{\sigma}_i=\hat{\eps}_i$ does pretty well in terms of size only when the sample size is as large as 500.
When the null hypothesis does not hold, the CS test with the step-down critical value yields the highest proportion of rejections in all cases. Moreover, in case 3 with the sample size $n=200$, this test has much higher power than that of GSV, GHJK, and HH. The CS test also has higher power than that of the IS test. Finally, the table shows that the one-step critical value gives a notable improvement in terms of power in comparison with plug-in critical value. For example, in case 3 with the sample size $n=200$, the one-step critical value gives additional 190 rejections out 1000 simulations in comparison with the plug-in critical value for the CS test and additional 325 rejections for the IS test. On the other hand, the step-down approach gives only minor improvements over the one-step approach. Overall, the results of the simulations are consistent with the theoretical findings in this paper. In particular, selection procedures yielding one-step and step-down critical values improve power with no size distortions. Additional simulation results are presented in the supplementary Appendix.

\section{Empirical Application}\label{sec: empirical example}

In this section, I review the arguments of \cite{EllisonEllison}
on how strategic entry deterrence might yield a non-monotone relation
between market size and investment in the pharmaceutical industry
and then apply the testing procedures developed in this paper to their dataset. I start
with describing their theory. Then I provide the details of the dataset.
Finally, I present the results.

In the pharmaceutical industry, incumbents whose patents are about
to expire can use investments strategically to prevent generic
entries after the expiration of the patent. In order to understand
how this strategic entry deterrence influences the relation between
market size and investment levels, \cite{EllisonEllison} developed
two models for an incumbent's investment. In the first model, potential entrants do not observe
the incumbent's investment but they do in the second
one. So, a strategic entry deterrence motive is absent in the former
model but is present in the latter one. Therefore, the difference in incumbent's investment between two models is explained by the strategic entry deterrence. Ellison and Ellison showed that in the former
model, the investment-market size relation is determined by a combination
of direct and competition effects. The direct effect is positive if
increasing the market size (holding entry probabilities
fixed) raises the marginal benefit from the investment more than it
raises the marginal cost of the investment. The competition effect
is positive if the marginal benefit of the investment is larger when
the incumbent is engaged in duopoly competition than it is when the
incumbent is a monopolist. The equilibrium investment is increasing
in market size if and only if the sum of two effects is positive.
Therefore, a sufficient condition for the monotonicity of investment-market
size relation is that both effects are of the same sign.\footnote{An interested reader can find a more detailed discussion in the original
paper.
} In the latter model, there is also a strategic entry deterrence
effect. The authors noted that this effect should be relatively less
important in small and large markets than it is in markets of intermediate
size. In small markets, there are not enough profits for potential
entrants, and there is no need to prevent entry. In large markets,
profits are so large that no reasonable investment levels will be
enough to prevent entries. As a result, strategic entry deterrence
might yield a non-monotonic relation between market size and investment
no matter whether the relation in the model with no strategic entry
deterrence is increasing or decreasing.

Ellison and Ellison studied three types of investment: detail
advertising, journal advertising, and presentation proliferation.
Detail advertising, measured as per-consumer expenditures, refers to sending representatives to doctors' offices.
Since both revenues and cost of detail advertising are likely to be
linear in the market size, it can be shown that the direct effect
for detail advertising is zero. The competition
effect is likely to be negative because detail advertising will benefit
competitors as well. Therefore, it is expected that detail advertising
is a decreasing function of the market size in the absence of strategic
distortions. Strategic entry deterrence should decrease detail advertising
for markets of intermediate size. Journal advertising is the placement
of advertisements in medical journals. Journal advertising is also measured as per-consumer expenditures. The competition effect for journal
advertising is expected to be negative for the same reason as for
detail advertising. The direct effect, however, may be positive because
the cost per potential patient is probably a decreasing function of
the market size. Opposite directions of these effects make journal
advertising less attractive for detecting strategic entry deterrence
in comparison with detail advertising. Nevertheless, following the original paper, I assume that journal advertising is a decreasing function of the market size in the absence of strategic distortions.
Presentation proliferation is selling a drug in
many different forms. Since the benefits of introducing a new form
is approximately proportional to the market size while the costs can
be regarded as fixed, the direct effect for presentation proliferation
should be positive. In addition, the competition effect is also likely
to be positive because it creates a monopolistic niche for the incumbent.
Therefore, presentation proliferation should be positively related
to market size in the absence of strategic distortions.

The dataset consists of 63 chemical compounds, sold under 71 different
brand names. All of these drugs lost their patent exclusivity between
1986 and 1992. There are four variables in the dataset: average revenue
for each drug over three years before the patent expiration (this
measure should be regarded as a proxy for market size), average costs
of detail and journal advertising over the same time span as revenues,
and a Herfindahl-style measure of the degree to which revenues are concentrated
in a small number of presentations (this measure should be regarded
as the inverse of presentation proliferation meaning that higher values
of the measure indicate lower presentation proliferation).

Clearly, the results will depend on how I define both dependent and
independent variables for the test. Following the strategy adopted
in the original paper, I use log of revenues as the independent variable
in all cases, and the ratio of advertising costs to revenues for detail
and journal advertising and the Herfindahl-style measure for presentation
proliferation as the dependent variable. The null hypothesis is that
the corresponding conditional mean function is decreasing.\footnote{In the original paper, \cite{EllisonEllison} test the null hypothesis
consisting of the union of monotonically increasing and monotonically
decreasing regression functions. The motivation for this modification
is that increasing regression functions contradict the theory developed
in the paper and, hence, should not be considered as evidence
of the existence of strategic entry deterrence. On the other hand,
increasing regression functions might arise if the strategic entry deterrence
effect over-weighs direct and competition effects even in small and
large markets, which could be considered as extreme evidence of
the existence of strategic entry deterrence.
}

I consider the test with kernel weighting functions with $k=0$ or $1$ and the kernel $K(x)=0.75(1-x^2)$ for $x\in(-1,1)$ and $0$ otherwise. I use the set of bandwidth values $H_n=\{0.5;1\}$ and the set of weighting functions $\MS_n=\{(x,h):x\in\{X_1,...,X_n\},h\in H_n\}$. Implementing the test requires estimating $\sigma_{i}^{2}$ for all
$i=1,...,n$. Since the test based on Rice's method outperformed that with $\hat{\sigma}_i=\hat{\eps}_i$ in the Monte Carlo simulations, I use this method in the benchmark procedure. I also check robustness
of the results using the following two-step procedure. First, I obtain
residuals of the OLS regression of $Y$ on a set of transformations
of $X$. In particular, I use polynomials in $X$ up to the third degree
(cubic polynomial). Second, squared residuals are projected onto the
same polynomial in $X$ using the OLS regression again. The resulting
projections are estimators $\hat{\sigma}_{i}^{2}$ of $\sigma_{i}^{2}$,
$i=1,...,n$.

The results of the test are presented in table 2. The table shows
the p-value of the test for each type of investment and each method
of estimating $\sigma_{i}^{2}$. In the table, method 1 corresponds
to estimating $\sigma_{i}^{2}$ using Rice's formula, and methods
2, 3, and 4 are based on polynomials of first, second, and third degrees
respectively. Note that all methods yield similar numbers, which reassures the robustness of the results. All the methods with $k=0$ reject the null hypothesis that journal advertising is decreasing in market size with 10\% confidence level. This may be regarded as evidence that pharmaceutical companies use strategic investment in the form of journal advertising to deter generic entries. On the other hand, recall that direct and competition effects probably have different signs for journal advertising, and so rejecting the null may also be due to the fact that the direct effect dominates for some values of market size. In addition, the test with $k=1$ does not reject the null hypothesis that journal advertising is decreasing in market size at the 10\% confidence level, no matter how $\sigma_i$ are estimated.
No method rejects the null hypothesis in the case of detail advertising and presentation proliferation. This may be (1) because firms do not use these types of investment for strategic entry deterrence, (2) because the strategic effect is too weak to yield non-monotonicity, or (3) because the sample size is not large enough. Overall, the results are consistent with those presented in \cite{EllisonEllison}.

\begin{table}

\caption{Incumbent Behavior versus Market Size: Monotonicity Test p-value}

\begin{tabular}{ccccccc}
\hline 
\multirow{3}{*}{{\small Method}} & \multicolumn{6}{c}{{\small Investment Type}}\tabularnewline
\cline{2-7} 
 & \multicolumn{2}{c}{{\small Detail Advertising}} & \multicolumn{2}{c}{{\small Journal Advertising}} & \multicolumn{2}{c}{{\small Presentation Proliferation}}\tabularnewline
\cline{2-7} 
 & {\small k=0} & {\small k=1} & {\small k=0} & {\small k=1} & {\small k=0} & {\small k=1}\tabularnewline
\hline 
{\small 1} & {\small .120} & {\small .111} & {\small .056} & {\small .120} & {\small .557} & {\small .661}\tabularnewline
{\small 2} & {\small .246} & {\small .242} & {\small .088} & {\small .168} & {\small .665} & {\small .753}\tabularnewline
{\small 3} & {\small .239} & {\small .191} & {\small .099} & {\small .195} & {\small .610} & {\small .689}\tabularnewline
{\small 4} & {\small .301} & {\small .238} & {\small .098} & {\small .194} & {\small .596} & {\small .695}\tabularnewline
\hline 
\end{tabular}

\end{table}

\section{Conclusion}\label{sec: conclusion}

In this paper, I have developed a general framework for testing monotonicity of a nonparametric regression function, and have given a broad class of new tests. A general test statistic uses many different weighting functions so that an approximately optimal weighting function
is determined automatically. In this sense, the test adapts to the properties of the model. I have also obtained new methods to simulate the critical
values for these tests. These are based on selection procedures. The procedures
are used to estimate what counterparts of the test statistic should
be used in simulating the critical value. They are constructed so that
no violation of the asymptotic size occurs. Finally, I have given tests suitable for models with multiple and endogenous covariates for the first time in the literature.

The new methods have numerous applications in economics. In particular, they can be applied to test qualitative predictions of comparative statics analysis including those derived via robust comparative statics. In addition, they are useful for evaluating monotonicity assumptions, which are often imposed in economic and econometric models, and for classifying economic objects in those cases where classification includes the concept of monotonicity (for example, normal/inferior and luxury/necessity goods). Finally, these methods can be used to detect strategic behavior of economic agents that might cause non-monotonicity in otherwise monotone relations.

The attractive properties of the new tests are demonstrated via Monte Carlo simulations.
In particular, it is shown that the rejection probability of the new
tests greatly exceeds that of other tests for some simulation designs. In addition, I applied the tests developed
in this paper to study entry deterrence effects in the pharmaceutical
industry using the dataset of \cite{EllisonEllison}. I showed that
the investment in the form of journal advertising seems to
be used by incumbents in order to prevent generic entries after the
expiration of patents. The evidence is rather weak, though.

\appendix

\section{Implementation Details}\label{sec: algorithms}
In this section, I provide detailed step-by-step instructions for implementing plug-in, one-step, and step-down critical values. The instructions are given for constructing a test of level $\alpha$. In all cases, let $B$ be a large integer denoting the number of bootstrap repetitions, and let $\{\epsilon_{i,b}\}_{i=1,b=1}^{n,B}$ be a set of independent $N(0,1)$ random variables. For one-step and step-down critical values, let $\gamma$ denote the truncation probability, which should be small relative to $\alpha$.

\subsection{Plug-in Approach}
\begin{enumerate}
\item For each $b=\overline{1,B}$ and $i=\overline{1,n}$, calculate $Y_{i,b}^\star=\hat{\sigma}_i\epsilon_{i,b}$.
\item For each $b=\overline{1,B}$, calculate the value $T_b^\star$ of the test statistic using the sample $\{X_i,Y_{i,b}^\star\}_{i=1}^n$.
\item Define the plug-in critical value, $c_{1-\alpha}^{PI}$, as the $(1-\alpha)$ sample quantile of $\{T_b^\star\}_{b=1}^B$.
\end{enumerate}

\subsection{One-Step Approach}
\begin{enumerate}
\item For each $b=\overline{1,B}$ and $i=\overline{1,n}$, calculate $Y_{i,b}^\star=\hat{\sigma}_i\epsilon_{i,b}$.
\item Using the plug-in approach, simulate $c_{1-\gamma}^{PI}$.
\item Define $\MS_n^{OS}$ as the set of values $s\in\MS_n$ such that $b(s)/(\hat{V}(s))^{1/2}>-2c_{1-\gamma}^{PI}$.
\item For each $b=\overline{1,B}$, calculate the value $T_b^\star$ of the test statistic using the sample $\{X_i,Y_{i,b}^\star\}_{i=1}^n$ and taking maximum only over $\MS_n^{OS}$ instead of $\MS_n$.
\item Define the one-step critical value, $c_{1-\alpha}^{OS}$, as the $(1-\alpha)$ sample quantile of $\{T_b^\star\}_{b=1}^B$.
\end{enumerate}

\subsection{Step-down Approach}
\begin{enumerate}
\item For each $b=\overline{1,B}$ and $i=\overline{1,n}$, calculate $Y_{i,b}^\star=\hat{\sigma}_i\epsilon_{i,b}$.
\item Using the plug-in and one-step approaches, simulate $c_{1-\gamma}^{PI}$ and $c_{1-\gamma}^{OS}$, respectively.
\item Denote $\MS_n^1=\MS_n^{OS}$, $c^1=c_{1-\gamma}^{OS}$, and set $l=1$.
\item For given value of $l\geq 1$, define $\MS_n^{l+1}$ as the set of values $s\in\MS_n^l$ such that $b(s)/(\hat{V}(s))^{1/2}>-c_{1-\gamma}^{PI}-c^l$.
\item For each $b=\overline{1,B}$, calculate the value $T_b^\star$ of the test statistic using the sample $\{X_i,Y_{i,b}^\star\}_{i=1}^n$ and taking the maximum only over $\MS_n^{l+1}$ instead of $\MS_n$.
\item Define $c^{l+1}$, as the $(1-\gamma)$  sample quantile of $\{T_b^\star\}_{b=1}^B$.
\item If $\MS_n^{l+1}=\MS_n^l$, then go to step (8). Otherwise, set $l=l+1$ and go to step (4).
\item For each $b=\overline{1,B}$, calculate the value $T_b^\star$ of the test statistic using the sample $\{X_i,Y_{i,b}^\star\}_{i=1}^n$ and taking the maximum only over $\MS_n^{l}$ instead of $\MS_n$.
\item Define $c_{1-\alpha}^{SD}$, as the $(1-\alpha)$ sample quantile of $\{T_b^\star\}_{b=1}^B$.
\end{enumerate}

\section{Verification of High-Level Conditions}\label{sec: verification of high level conditions}
This section verifies high level assumptions used in Section \ref{sec: theory under high level conditions} under primitive conditions.
Assumptions used in Sections \ref{sec: multivariate models}, \ref{sec: endogenous covariates}, and \ref{sec: sample selection models} can be verified by similar arguments.
 First, I consider Assumption A\ref{as: sigma}, which concerns the uniform consistency of the estimator $\hat{\sigma}_i$ of $\sigma_i$ over $i=\overline{1,n}$. Second, I verify Assumption A\ref{as: V} and derive a bound on the sensitivity parameter $A_n$ when kernel weighting functions are used. This allows me to verify Assumption A\ref{as: growth condition}. Finally, I provide primitive conditions for Assumption A\ref{as: weighting functions difficult}.

Let $C_6$ and $L^\prime$ be strictly positive constants. The following proposition proves consistency of the local version of Rice's estimator.
\begin{proposition}[Verifying Assumption A\ref{as: sigma}]\label{lem: consistency of rice estimator}
Suppose that $\hat{\sigma}_i$ is the local version of Rice's estimator of $\sigma_i$ given in equation (\ref{eq: local rice estimator}). Suppose also that (i) Assumptions A\ref{as: disturbances} and A\ref{as: design difficult} hold, (ii) $b_n\geq (\log n)^2/(C_6n)$, (iii) $|f(X_i)-f(X_j)|\leq L^\prime|X_i-X_j|$ for all $i,j=\overline{1,n}$, and (iv) $|\sigma_i^2-\sigma_j^2|\leq L^\prime|X_i-X_j|$ for all $i,j=\overline{1,n}$. Then 
$
\max_{1\leq i\leq n}|\hat{\sigma}_i-\sigma_i|=O_p(b_n^2+\log n/(b_nn^{1/2}))
$, and so Assumption A\ref{as: sigma} holds with any $\kappa_2$ satisfying  $b_n^2+\log n/(b_nn^{1/2})=o(n^{-\kappa_2})$.
\end{proposition}

Next, I consider restrictions on the weighting functions to ensure that Assumption A\ref{as: V} holds and give an upper bound on the sensitivity parameter $A_n$.

\begin{proposition}[Verifying Assumptions A\ref{as: V} and A\ref{as: growth condition}]\label{lem: kernel function}
Suppose that kernel weighting functions are used. In addition, suppose that (i) Assumptions A\ref{as: disturbances} and A\ref{as: design difficult} hold, (ii) $K(\cdot)$ has the support $[-1,+1]$, is continuous, and strictly positive on the interior of its support, (iii) $x\in[s_l,s_r]$ for all $(x,h)\in\MS_n$, (iv) $h_{\min}\geq (\log n)^2/(C_6n)$ w.p.a.1 where $h_{\min}=\min_{(x,h)\in\MS_n}h$, and (v) $h_{\max}\leq (s_r-s_l)/2$ where $h_{\max}=\max_{(x,h)\in\MS_n}h$. Then (a) $A_n\leq C/(nh_{\min})^{1/2}$ w.p.a.1 for some $C>0$, and so Assumption A\ref{as: growth condition} holds for the basic set of weighting functions;  (b) if Assumption A\ref{as: sigma} is satisfied, then Assumption A\ref{as: V} holds with $\kappa_3=\kappa_2$; (c) if Assumption A\ref{as: function estimation} is satisfied, then Assumption A\ref{as: V} holds with $\kappa_3=\kappa_1$ as long as $\log p/(h_{\min}n^{1/2})=o_p(n^{-\kappa_1})$.
\end{proposition}

Restrictions on the kernel $K(\cdot)$ imposed in this lemma are satisfied for most commonly used kernel functions including uniform, triangular, Epanechnikov, biweight, triweight, and tricube kernels. Note, however, that these restrictions exclude higher order kernels since those are necessarily negative at some points on their supports. Assumption A\ref{as: design difficult} imposes the restriction that the density of $X$ is bounded below from zero on its support. This is needed to make sure that $V(s)$ is sufficiently well separated from zero since $V(s)$ appears in the denominator of the test statistic $T$. Alternatively, one can truncate $V(s)$ from below with the truncation parameter converging to zero with an appropriate rate as the sample size $n$ increases. I do not consider this possibility for brevity of the paper.
Finally, I verify Assumption A\ref{as: weighting functions difficult}.

\begin{proposition}[Verifying Assumption A\ref{as: weighting functions difficult}]\label{lem: asumption 9 verification}
Suppose that the basic set of weighting functions is used. In addition, suppose that (i) Assumption A\ref{as: design difficult} holds and (ii) $K(\cdot)$ has the support $[-1,+1]$, is continuous, and strictly positive on the interior of its support. Then Assumption A\ref{as: weighting functions difficult} holds.
\end{proposition}

\section{Additional Notation}\label{sub: additional notation}
I will use the following additional notation in Appendices C and D. Recall that $\{\epsilon_i\}$ is a sequence of independent $N(0,1)$ random variables that are independent of the data. Denote $e_i=\sigma_i\epsilon_i$ and $\hat{e}_i=\hat{\sigma}_i\epsilon_i$ for $i=\overline{1,n}$. Let
$$
w_i(s)=\sum_{1\leq j\leq n}\sign(X_j-X_i)Q(X_i,X_j,s),
$$
$$
a_i(s)=w_i(s)/(V(s))^{1/2}\text{ and }\hat{a}_i(s)=w_i(s)/(\hat{V}(s))^{1/2},
$$
$$
e(s)=\sum_{1\leq i\leq n}a_i(s)e_i\text{, and }\hat{e}(s)=\sum_{1\leq i\leq n}\hat{a}_i(s)\hat{e}_i,
$$
$$
\eps(s)=\sum_{1\leq i\leq n}a_i(s)\eps_i\text{ and }\hat{\eps}(s)=\sum_{1\leq i\leq n}\hat{a}_i(s)\eps_i,
$$
$$
f(s)=\sum_{1\leq i\leq n}a_i(s)f(X_i)\text{ and }\hat{f}(s)=\sum_{1\leq i\leq n}\hat{a}_i(s)f(X_i).
$$
Note that $T=\max_{s\in\MS_n}\sum_{1\leq i\leq n}\hat{a}_i(s)Y_i=\max_{s\in\MS_n}(\hat{f}(s)+\hat{\eps}(s))$. In addition, for any $\MS\subset\MS_n$, which may depend on $\{X_i\}_{1\leq i\leq n}$, and all $\eta\in(0,1)$, let $c_{\eta}^{\MS}$ denote the conditional $\eta$ quantile of $T^{\star}=T(\{X_i,Y^{\star}_i\},\{\hat{\sigma}_i\},\MS)$ given $\{X_i\}$, $\{\hat{\sigma}_i\}$, and $\MS$ where $Y^{\star}_i=\hat{\sigma}_i\epsilon_i$ for $i=\overline{1,n}$, and let $c_{\eta}^{\MS,0}$ denote the conditional $\eta$ quantile of $T^{\star}=T(\{X_i,Y^{\star}_i\},\{\sigma_i\},\MS)$ given $\{X_i\}$ and $\MS$ where $Y^{\star}_i=\sigma_i\epsilon_i$ for $i=\overline{1,n}$. Further, for $\eta\leq 0$, define $c_{\eta}^{\MS}$ and $c_{\eta}^{\MS,0}$ as $-\infty$, and for $\eta\geq 1$, define $c_{\eta}^{\MS}$ and $c_{\eta}^{\MS,0}$ as $+\infty$.

In the next section, I prove some auxiliary lemmas. For these lemmas, it will be convenient to use the following additional notation. Recall that $\MS_n$ is a set that is allowed to depend on $\{X_i\}$. In addition, I will need some subsets of $\MS_n$ that are selected from $\MS_n$ according to certain mappings that also depend $\{X_i\}$, i.e. $\MS=\MS(\MS_n,\{X_i\})\subset \MS_n$. Some results will be stated to hold uniformly over all mappings $\MS(\cdot,\cdot)$.

Let $\{\psi_n\}$ be a sequence of positive numbers converging to zero sufficiently slowly so that (i) $\log p/n^{\kappa_3}=o_p(\psi_n)$ (recall that by Assumption A\ref{as: growth condition}, $\log p/n^{\kappa_3}=o_p(1)$, and so such a sequence exists), (ii) uniformly over all $\MS(\cdot,\cdot)$ and $\eta\in(0,1)$, $\Pr(c_{\eta+\psi_n}^{\MS,0}<c_\eta^{\MS})=o(1)$ and $\Pr(c_{\eta+\psi_n}^{\MS}<c_\eta^{\MS,0})=o(1)$ where $\MS=\MS(\MS_n,\{X_i\})$ (Lemma \ref{lem: conditional and unconditional quantiles 1} establishes existence of such a sequence under Assumptions A\ref{as: disturbances}, A\ref{as: sigma}, A\ref{as: V}, and A\ref{as: growth condition} and Lemma \ref{lem: conditional and unconditional quantiles 3} establishes existence under Assumptions A\ref{as: disturbances}, A\ref{as: function estimation}, A\ref{as: V}, and A\ref{as: growth condition}). Let
$$
\MS_n^R=\{s\in\MS_n:f(s)>-c_{1-\gamma_n-\psi_n}^{\MS_n,0}\}.
$$
For $D=PI,OS,SD,R$, let $c_{\eta}^D=c_{\eta}^{\MS_n^D}$ and $c_{\eta}^{D,0}=c_{\eta}^{\MS_n^D,0}$ where $\MS_n^{PI}=\MS_n$. Note that $c_\eta^{PI,0}$ and $c_\eta^{R,0}$ are non-stochastic.

Moreover, I denote $\mathcal{V}=\max_{s\in\MS_n}(V(s)/\hat{V}(s))^{1/2}$. Finally, I denote the space of $k$-times continuously differentiable functions on $\RR$ by $\mathbb{C}^k(\RR,\RR)$. For $g\in\mathbb{C}^k(\RR,\RR)$, the symbol $g^{(r)}$ for $r\leq k$ denotes the $r$th derivative of $g$, and $\Vert g^{(r)}\Vert_\infty=\sup_{t\in\RR}|g^{(r)}(t)|$.

\section{Proofs for section \ref{sec: theory under high level conditions}}

In this Appendix, I first prove a sequence of auxiliary lemmas (Subsection
\ref{sec: auxiliary lemmas}). Then I present the proofs of the theorems
stated in Section \ref{sec: theory under high level conditions} (Subsection \ref{sec: proof of theorems}). In Subsection \ref{sec: auxiliary lemmas}, all results hold uniformly over all models $M\in\mathcal{M}$. For simplicity of notation, however, I drop index $M$ everywhere.

\subsection{Auxiliary Lemmas}\label{sec: auxiliary lemmas}

\begin{lemma}\label{lem: maximal inequality}
$\Ep[\max_{s\in\MS_n}|e(s)||\{X_i\}]\leq C(\log p)^{1/2}$ for some universal $C>0$.
\end{lemma}
\begin{proof}
Note that by construction, conditional on $\{X_i\}$, $e(s)$ is distributed as a $N(0,1)$ random variable, and $|\MS_n|=p$. So, the result follows from lemma 2.2.2 in \cite{VaartWellner1996}.
\end{proof}

\begin{lemma}\label{lem: anticoncentration}
For all $\MS(\cdot,\cdot)$ and $\Delta>0$,
$$
\sup_{t\in\RR}\Pr\left(\max_{s\in\MS}e(s)\in[t,t+\Delta]|\{X_i\}\right)\leq C\Delta(\log p)^{1/2}
$$
and for all $\MS(\cdot,\cdot)$ and $(\eta,\delta)\in(0,1)^2$, 
$$
c_{\eta+\delta}^{\MS,0}-c_{\eta}^{\MS,0}\geq C\delta/(\log p)^{1/2}
$$ 
for some universal $C>0$ where $\MS=\MS(\MS_n,\{X_i\})$.
\end{lemma}
\begin{proof}
Theorem 3 in \cite{ChernozhukovKato2011} shows that if $W_1,\dots,W_p$ is a sequence of $N(0,1)$ random (not necessarily independent) variables, then for all $\Delta>0$,
$$
\sup_{t\in\RR}\Pr\left(\max_{1\leq j\leq p}W_j\in[t,t+\Delta]\right)\leq 4\Delta\left(\Ep\left[\max_{1\leq j\leq p}W_j\right]+1\right).
$$
Therefore,
$$
\sup_{t\in\RR}\Pr\left(\max_{s\in\MS}e(s)\in[t,t+\Delta]|\{X_i\}\right)\leq 4\Delta\left(\Ep\left[\max_{s\in\MS}e(s)|\{X_i\}\right]+1\right),
$$
and so the first claim follows from Lemma \ref{lem: maximal inequality}. The second claim follows from the result in the first claim.
\end{proof}

\begin{lemma}\label{lem: multiplied quantiles}
For all $\MS(\cdot,\cdot)$, $\eta\in(0,1)$, and $t\in\RR$, 
$$
c_{\eta-C|t|\log p/(1-\eta)}^{\MS,0}\leq c_\eta^{\MS,0}(1+t)\leq c_{\eta+C|t|\log p/(1-\eta)}^{\MS,0}
$$
for some universal $C>0$ where $\MS=\MS(\MS_n,\{X_i\})$.
\end{lemma}
\begin{proof}
Recall that $c_\eta^{\MS,0}$ is the conditional $\eta$ quantile of $\max_{s\in\MS}e(s)$ given $\{X_i\}$, and so combining Lemma \ref{lem: maximal inequality} and Markov inequality shows that $c_\eta^{\MS,0}\leq C(\log p)^{1/2}/(1-\eta)$ for some universal $C>0$. Therefore, Lemma \ref{lem: anticoncentration} gives
\begin{equation}\label{eq: quantile with mistakes}
c_{\eta+C|t|\log p/(1-\eta)}^{\MS,0}-c_\eta^{\MS,0}\geq C|t|(\log p)^{1/2}/(1-\eta)\geq |t|c_\eta^{\MS,0}
\end{equation}
if $C>0$ is sufficiently large in (\ref{eq: quantile with mistakes}). The lower bound follows similarly.
\end{proof}

\begin{lemma}\label{lem: gaussian approximation}
Under Assumptions A\ref{as: disturbances} and A\ref{as: growth condition}, uniformly over all $\MS(\cdot,\cdot)$ and $\eta\in(0,1)$,
$$
\Pr\left(\max_{s\in\MS}\eps(s)\leq c_{\eta}^{\MS,0}\right)=\eta+o(1)\text{ and }\Pr\left(\max_{s\in\MS}(-\eps(s))\leq c_{\eta}^{\MS,0}\right)=\eta+o(1)
$$
where $\MS=\MS(\MS_n,\{X_i\})$.
\end{lemma}

\begin{proof}
Note that $\sum_{1\leq i\leq n}(a_i(s)\sigma_i)^2=1$. In addition, by Assumption A\ref{as: growth condition}, $nA_n^4(\log(pn))^7=o_p(1)$, and so there exists a sequence $\{\omega_n\}$ of positive numbers such that $nA_n^4(\log(pn))^7=o_p(\omega_n)$ and $\omega_n=o(1)$. Let $\mathcal{A}_n$ denote the event that $nA_n^4(\log(pn))^7>\omega_n$. Then $\Pr(\mathcal{A}_n)=o(1)$. Further, note that $\Ep[\eps_i^2/\sigma_i^2|X_i]=1$, and since $\Ep[\eps_i^4|X_i]\leq C_1$ and $\sigma_i\geq c_1$ by Assumption A\ref{as: disturbances}, $\Ep[\eps_i^4/\sigma_i^4|X_i]\leq C$ for some $C>0$. Therefore, applying Lemma \ref{lem: gaussian approximation original}, case (ii), with $z_{is}=\sqrt{n}a_i(s)\sigma_i$ and $u_i=\eps_i/\sigma_i$ conditional on $\{X_i\}$ shows that there exists a sequence $\{\omega'\}$ of positive numbers converging to zero and depending only on $\{\omega_n\}$ such that outside of $\mathcal{A}_n$,
$$
\left|\Pr\left(\max_{s\in\MS}\eps(s)\leq c_\eta^{\MS,0}|\{X_i\}\right)-\eta\right|\leq \omega_n'.
$$ 
Since $\Pr(\mathcal{A}_n)=o(1)$, this implies that
$$
\Pr\left(\max_{s\in\MS}\eps(s)\leq c_\eta^{\MS,0}\right)=\eta+o(1),
$$
which gives the first asserted claim. The second claim follows by replacing $\eps(s)$ by $-\eps(s)$.
Note that since the sequence $\{\omega_n'\}$ depends only on the sequence $\{\omega_n\}$, it follows that the result of this lemma holds uniformly over all models $M\in\mathcal{M}$.
\end{proof}

\begin{lemma}\label{lem: conditional and unconditional quantiles 1}
Under Assumptions A\ref{as: disturbances}, A\ref{as: sigma}, A\ref{as: V}, and A\ref{as: growth condition}, there exists a sequence $\{\psi_n\}$ of positive numbers converging to zero such that uniformly over all $\MS(\cdot,\cdot)$ and $\eta\in(0,1)$, $\Pr(c_{\eta+\psi_n}^{\MS,0}<c_\eta^{\MS})=o(1)$ and $\Pr(c_{\eta+\psi_n}^{\MS}<c_\eta^{\MS,0})=o(1)$
where $\MS=\MS(\MS_n,\{X_i\})$.
\end{lemma}
\begin{proof}
Denote
$$
T^{\MS}=\max_{s\in\MS}\hat{e}(s)=\max_{s\in\MS}\sum_{1\leq i\leq n}\hat{a}_{i}(s)\hat{\sigma}_{i}\epsilon_{i}\text{ and }T^{\MS,0}=\max_{s\in\MS}e(s)=\max_{s\in\MS}\sum_{1\leq i\leq n}a_{i}(s)\sigma_{i}\epsilon_{i}.
$$
In addition, denote
$$
p_1=\max_{s\in\MS}|e(s)|\max_{s\in\MS}\left|1-(V(s)/\hat{V}(s))^{1/2}\right|,
$$
$$
p_2=\max_{s\in\MS}\left|\sum_{1\leq i\leq n} a_i(s)(\hat{\sigma}_i-\sigma_i)\epsilon_i\right|\max_{s\in\MS}(V(s)/\hat{V}(s))^{1/2}.
$$
Then $|T^{\MS}-T^{\MS,0}|\leq p_1+p_2$. Combining Lemma \ref{lem: maximal inequality} and Assumption A\ref{as: V} gives
$$
p_1=o_p\left((\log p)^{1/2}n^{-\kappa_3}\right).
$$
Consider $p_2$. Conditional on $\{\hat{\sigma}_i\}$, $(\hat{\sigma}_i-\sigma_i)\epsilon_i$ is distributed as a $N(0,(\hat{\sigma}_i-\sigma_i)^2)$ random variable, and so applying the argument like that in Lemma \ref{lem: maximal inequality} conditional on $\{X_i,\hat{\sigma}_i\}$ and using Assumption A\ref{as: sigma} gives
$$
\max_{s\in\MS}\left|\sum_{1\leq i\leq n}a_i(s)(\hat{\sigma}_i-\sigma_i)\epsilon_i\right|=o_p\left((\log p)^{1/2}n^{-\kappa_2}\right).
$$
Since $\max_{s\in\MS}(V(s)/\hat{V}(s))^{1/2}\rightarrow_p 1$ by Assumption A\ref{as: V}, this implies that 
$$
p_2=o_p\left((\log p)^{1/2}n^{-\kappa_2}\right).
$$
Therefore, $T^{\MS}-T^{\MS,0}=o_p((\log p)^{1/2}n^{-\kappa_2\wedge\kappa_3})$, and so there exists a sequence $\{\tilde{\psi}_{n}\}$
of positive numbers converging to zero such that
$$
\Pr\left(|T^{\MS}-T^{\MS,0}|>(\log p)^{1/2}n^{-\kappa_2\wedge\kappa_3}\right)=o(\tilde{\psi}_{n}).
$$
Hence,
$$
\Pr\left(\Pr\left(|T^{\MS}-T^{\MS,0}|>(\log p)^{1/2}n^{-\kappa_2\wedge\kappa_3}|\{X_i,\hat{\sigma}_i\}\right)>\tilde{\psi}_{n}\right)\rightarrow0.
$$
Let $\mathcal{A}_{n}$ denote the event that 
$$
\Pr\left(|T^{\MS}-T^{\MS,0}|>(\log p)^{1/2}n^{-\kappa_2\wedge\kappa_3}|\{X_i,\hat{\sigma}_i\}\right)\leq\tilde{\psi}_{n}.
$$
I will take $\psi_{n}=\tilde{\psi}_{n}+C(\log p) n^{-\kappa_2\wedge\kappa_3}$ for a constant $C$ that is larger than that in the statement of Lemma \ref{lem: anticoncentration}. By assumption A\ref{as: growth condition},
$\psi_{n}\rightarrow0$. Then note that
$$
\Pr\left(T^{\MS,0}\leq c_{\eta}^{\MS,0}|\{X_i,\hat{\sigma}_i\}\right)\geq\eta\text{ and }\Pr\left(T^{\MS}\leq c_{\eta}^{\MS}|\{X_i,\hat{\sigma}_i\}\right)\geq\eta
$$
for any $\eta\in(0,1)$. 
So, on $\mathcal{A}_{n}$,
\begin{align*}
\eta+\tilde{\psi}_{n} & \leq  \Pr\left(T^{\MS,0}\leq c_{\eta+\tilde{\psi}_{n}}^{\MS,0}|\{X_i,\hat{\sigma}_i\}\right)\\
 & \leq  \Pr\left(T^{\MS}\leq c_{\eta+\tilde{\psi}_{n}}^{\MS,0}+(\log p)^{1/2}n^{-\kappa_2\wedge\kappa_3}|\{X_i,\hat{\sigma}_i\}\right)+\tilde{\psi}_{n}
  \leq  \Pr\left(T^{\MS}\leq c_{\eta+\psi_{n}}^{\MS,0}|\{X_i,\hat{\sigma}_i\}\right)+\tilde{\psi}_{n}
\end{align*}
where the last inequality follows from Lemma \ref{lem: anticoncentration}. Therefore,
on $\mathcal{A}_{n}$, $c_{\eta}^{\MS}\leq c_{\eta+\psi_{n}}^{\MS,0}$,
i.e. $\Pr(c_{\eta+\psi_{n}}^{\MS,0}<c_{\eta}^{\MS})=o(1)$.
The second claim follows similarly.
\end{proof}

\begin{lemma}\label{lem: conditional and unconditional quantiles 2}
Let $c_\eta^{\MS,1}$ denote the conditional $\eta$ quantile of $T^{\MS,1}=\max_{s\in\MS}\sum_{1\leq i\leq n}a_i(s)\eps_i\epsilon_i$ given $\{X_i,\eps_i\}$. Let Assumptions A\ref{as: disturbances}, A\ref{as: function estimation}, and A\ref{as: growth condition} hold. Then there exists a sequence $\{\tilde{\psi}_n\}$ of positive numbers converging to zero such that uniformly over all $\MS(\cdot,\cdot)$ and $\eta\in(0,1)$,  $\Pr(c_{\eta+\tilde{\psi}_n}^{\MS,0}< c_\eta^{\MS,1})=o(1)$ and $\Pr(c_{\eta+\tilde{\psi}_n}^{\MS,1}<c_\eta^{\MS,0})=o(1)$
where $\MS=\MS(\MS_n,\{X_i\})$.
\end{lemma}
\begin{proof}
Let $Z^1=\{\sum_{1\leq i\leq n}a_i(s)\eps_i\epsilon_i\}_{s\in\MS}$ and $Z^2=\{\sum_{1\leq i\leq n}a_i(s)\sigma_i\epsilon_i\}_{s\in\MS}$. Conditional on $\{X_i,\eps_i\}$, these are zero-mean Gaussian vectors in $\RR^{|\MS|}$ with covariances $\Sigma^1$ and $\Sigma^2$ given by
$$
\Sigma^1_{s_1s_2}=\sum_{1\leq i\leq n}a_i(s_1)a_i(s_2)\eps_i^2\text{ and }\Sigma^2_{s_1s_2}=\sum_{1\leq i\leq n}a_i(s_1)a_i(s_2)\sigma_i^2.
$$
Let $\Delta_{\Sigma}=\max_{s_1,s_2\in\MS}|\Sigma^1_{s_1s_2}-\Sigma^2_{s_1s_2}|$. Applying Lemma \ref{lem: maximal inequality CCK} shows that under Assumption A\ref{as: disturbances}, for some $C>0$,
$$
\Ep\left[\Delta_{\Sigma}|\{X_i\}\right]\leq C\left(A_n\sqrt{\log p}+A_n^2\log p\Ep\left[\max_{1\leq i\leq n}\varepsilon_i^4|\{X_i\}\right]^{1/2}\right).
$$
Further, $\Ep[\max_{1\leq i\leq n}\varepsilon_i^4|\{X_i\}]\leq \sum_{1\leq i\leq n}\Ep[\varepsilon_i^4|\{X_i\}]\leq Cn$ for some $C>0$ by Assumption A\ref{as: disturbances}, and so Assumption A\ref{as: growth condition} implies that
\begin{equation}\label{eq: delta sigma bound}
(\log p)^2\Delta_{\Sigma}=o_p(1)
\end{equation}
because $A_n^2(\log p)^3n^{1/2}=o_p(1)$ and $A_n(\log p)^{5/2}=o_p(1)$ (to verify the latter equation, recall that $\log p/n^{1/4}=o(1)$ by Assumption A\ref{as: growth condition}; taking square of this equation and multiplying it by $A_n^2(\log p)^3n^{1/2}=o_p(1)$ gives $A_n(\log p)^{5/2}=o_p(1)$).
Therefore, it follows that there exists a sequence $\{\tilde{\psi}_n\}$ of positive numbers converging to zero such that 
\begin{equation}\label{eq: some bounds 1}
(\log p)^2\Delta_{\Sigma}=o_p(\tilde{\psi}_n^4).
\end{equation}
Let $g\in\mathbb{C}^2(\RR,\RR)$ be a function satisfying $g(t)=1$ for $t\leq 0$, $g(t)=0$ for $t\geq 1$, and $g(t)\in[0,1]$ for $t\in[0,1]$, and let $g_n(t)=g((t-c_{\eta+\tilde{\psi}_n/2}^{\MS,0})/(c_{\eta+\tilde{\psi}_n}^{\MS,0}-c_{\eta+\tilde{\psi}_n/2}^{\MS,0}))$. Then for some $C>0$, Lemma \ref{lem: anticoncentration} shows that
\begin{eqnarray*}
&&\Vert g_n^{(1)}\Vert_\infty\leq \|g^{(1)}\|_{\infty}/\left(c_{\eta+\tilde{\psi}_n}^{\MS,0}-c_{\eta+\tilde{\psi}_n/2}^{\MS,0}\right)\leq C(\log p)^{1/2}/\tilde{\psi}_n,\\
&&\Vert g_n^{(2)}\Vert_\infty\leq \|g^{(2)}\|_{\infty}/\left(c_{\eta+\tilde{\psi}_n}^{\MS,0}-c_{\eta+\tilde{\psi}_n/2}^{\MS,0}\right)^2\leq C(\log p)/\tilde{\psi}_n^2.
\end{eqnarray*}
Applying Lemma \ref{lem: gaussian comparison} gives for some $C>0$,
\begin{align}
D_n&=\left|\Ep\left[g_n\(\max_{s\in\MS}Z^1_s\)-g_n\(\max_{s\in\MS}Z^2_s\)|\{X_i,\eps_n\}\right]\right|\label{eq: bound on D1}\\
&\leq C\((\log p)\Delta_\Sigma/\tilde{\psi}_n^2+(\log p)(\Delta_\Sigma)^{1/2}/\tilde{\psi}_n\)=o_p(\tilde{\psi}_n)\label{eq: bound on D2}
\end{align}
by equation (\ref{eq: some bounds 1}). Note that $\max_{s\in\MS}Z_s^1=T^{\MS,1}$ and, using the notation in the proof of Lemma \ref{lem: conditional and unconditional quantiles 1}, $\max_{s\in\MS}Z_s^2=T^{\MS,0}$. Therefore,
\begin{align}
\Pr\(T^{\MS,1}\leq c_{\eta+\tilde{\psi}_n}^{\MS,0}|\{X_i,\eps_i\}\)
&\geq_{(1)} \Ep\left[g_n(T^{\MS,1})|\{X_i,\eps_i\}\right]\nonumber\\
&\geq_{(2)}\Ep\left[g_n(T^{\MS,0})|\{X_i,\eps_i\}\right]-D_n\nonumber\\
&\geq_{(3)}\Pr\(T^{\MS,0}\leq c_{\eta+\tilde{\psi}_n/2}^{\MS,0}|\{X_i,\eps_i\}\)-D_n\nonumber\\
&=_{(4)}\Pr\(T^{\MS,0}\leq c_{\eta+\tilde{\psi}_n/2}^{\MS,0}|\{X_i\}\)-D_n
\geq \eta+\tilde{\psi}_n/2-D_n\label{eq: second line}
\end{align}
where (1) and (3) are by construction of the function $g_n$, (2) is by the definition of $D_n$, and (4) is because $T^{\MS,0}$ and $c_{\eta+\tilde{\psi}_n/2}^{\MS,0}$ are jointly independent of $\{\eps_i\}$. Finally, note that the right hand side of line (\ref{eq: second line}) is bounded from below by $\eta$ w.p.a.1. This implies that $\Pr(c_{\eta+\tilde{\psi}_n}^{\MS,0}<c_{\eta}^{\MS,1})=o(1)$, which is the first asserted claim. The second claim of the lemma follows similarly.
\end{proof}

\begin{lemma}\label{lem: conditional and unconditional quantiles 3}
Under Assumptions A\ref{as: disturbances}, A\ref{as: function estimation}, A\ref{as: V}, and A\ref{as: growth condition}, there exists a sequence $\{\psi_n\}$ of positive numbers converging to zero such that uniformly over all $\MS(\cdot,\cdot)$ and $\eta\in(0,1)$, $\Pr(c_{\eta+\psi_n}^{\MS,0}<c_{\eta}^{\MS})=o(1)$ and $\Pr(c_{\eta+\psi_n}^{\MS}<c_{\eta}^{\MS,0})=o(1)$ where $\MS=\MS(\MS_n,\{X_i\})$.\footnote{Note that Lemmas \ref{lem: conditional and unconditional quantiles 1} and \ref{lem: conditional and unconditional quantiles 3} provide the same results under two different methods for estimating $\sigma_i$.}
\end{lemma}
\begin{proof}
Lemma \ref{lem: conditional and unconditional quantiles 2} established that
$$
\Pr(c_{\eta+\tilde{\psi}_n}^{\MS,0}<c_{\eta}^{\MS,1})=o(1)\text { and } \Pr(c_{\eta+\tilde{\psi}_n}^{\MS,1}<c_{\eta}^{\MS,0})=o(1).
$$
Therefore, it suffices to show that
$$
\Pr(c_{\eta+\hat{\psi}_n}^{\MS}<c_{\eta}^{\MS,1})\text { and } \Pr(c_{\eta+\hat{\psi}_n}^{\MS,1}<c_{\eta}^{\MS})
$$
for some sequence $\{\hat{\psi}_n\}$ of positive numbers converging to zero. Denote
$$
p_1=\max_{s\in\MS}\left|\sum_{1\leq i\leq n}a_i(s)\eps_i\epsilon_i\right|\max_{s\in\MS}\left|1-(V(s)/\hat{V}(s))^{1/2}\right|,
$$
$$
p_2=\max_{s\in\MS}\left|\sum_{1\leq i\leq n}a_i(s)(\hat{\sigma}_i-\eps_i)\epsilon_i\right|\max_{s\in\MS}(V(s)/\hat{V}(s))^{1/2}.
$$
Note that $|T^{\MS}-T^{\MS,1}|\leq p_1+p_2$ and that by Lemmas \ref{lem: maximal inequality} and \ref{lem: conditional and unconditional quantiles 2}, $\max_{s\in\MS}|\sum_{1\leq i\leq n}a_i(s)\eps_i\epsilon_i|=O_p((\log p)^{1/2})$. Therefore, the result follows by the argument similar to that used in the proof of Lemma \ref{lem: conditional and unconditional quantiles 1} since $\hat{\sigma}_i-\eps_i=o_p(n^{-\kappa_1})$ by Assumption A\ref{as: function estimation}.
\end{proof}


\begin{lemma}\label{lem: inclusion R and SD}
Let Assumptions A\ref{as: disturbances}, A\ref{as: V}, and A\ref{as: growth condition} hold. In addition, let either Assumption A\ref{as: function estimation} or A\ref{as: sigma} hold. Then
$\Pr(\MS_n^R\subset\MS_n^{SD})\geq 1-\gamma_n+o(1)$ and $\Pr(\MS_n^R\subset\MS_n^{OS})\geq 1-\gamma_n+o(1)$.
\end{lemma}
\begin{proof}
By the definition of $\psi_n$, $\log p/n^{\kappa_3}=o_p(\psi_n)$, and so there exists a sequence $\{\omega_n\}$ of positive numbers converging to zero such that $\log p/n^{\kappa_3}=o_p(\omega_n\psi_n)$. Let $\mathcal{A}_n$ denote the event that $\log p/n^{\kappa_3}>\omega_n\psi_n$. Then $\Pr(\mathcal{A}_n)=o(1)$.

Further, suppose that $\MS_n^R\backslash\MS_n^{SD}\neq\emptyset$. Then there exists the smallest integer $l$ such that $\MS_n^R\backslash\MS_n^l\neq\emptyset$ and $\MS_n^R\subset\MS_n^{l-1}$ (if $l=1$, let $\MS_n^0=\MS_n$). Therefore, $c_{1-\gamma_n}^R\leq c_{1-\gamma_n}^{l-1}$. It follows that there exists an element $s$ of $\MS_n^R$ such that 
$$
\hat{f}(s)+\hat{\eps}(s)\leq -c_{1-\gamma_n}^{PI}-c_{1-\gamma_n}^{l-1}\leq -c_{1-\gamma_n}^{PI}-c_{1-\gamma_n}^{R},
$$
and so for some $C>0$,
\begin{align*}
\Pr(\MS_n^R\backslash\MS_n^{SD}\neq\emptyset)
&\leq
\Pr(\min_{s\in\MS_n^R}(\hat{f}(s)+\hat{\eps}(s))\leq -c_{1-\gamma_n}^{PI}-c_{1-\gamma_n}^R)\\
&\leq_{(1)}
\Pr((\min_{s\in\MS_n^R}(f(s)+\eps(s))\mathcal{V}\leq -c_{1-\gamma_n}^{PI}-c_{1-\gamma_n}^R)\\
&\leq_{(2)}
\Pr((\min_{s\in\MS_n^R}(f(s)+\eps(s))\mathcal{V}\leq -c_{1-\gamma_n-\psi_n}^{PI,0}-c_{1-\gamma_n-\psi_n}^{R,0})+o(1)\\
&\leq_{(3)}
\Pr((\min_{s\in\MS_n^R}(\eps(s)-c_{1-\gamma_n-\psi_n}^{PI,0})\mathcal{V}\leq -c_{1-\gamma_n-\psi_n}^{PI,0}-c_{1-\gamma_n-\psi_n}^{R,0})+o(1)\\
&=_{(4)}
\Pr((\max_{s\in\MS_n^R}(-\eps(s))\geq c_{1-\gamma_n-\psi_n}^{PI,0}(1/\mathcal{V}-1)+c_{1-\gamma_n-\psi_n}^{R,0}/\mathcal{V})+o(1)\\
&\leq_{(5)}
\Pr((\max_{s\in\MS_n^R}(-\eps(s))\geq c_{1-\gamma_n-\psi_n}^{R,0}/\mathcal{V}-C(\log p)^{1/2}n^{-\kappa_3}/(\gamma_n+\psi_n))+o(1)\\
&\leq_{(6)}
\Pr((\max_{s\in\MS_n^R}(-\eps(s))\geq c^{R,0}_{1-\gamma_n-\psi_n-C(\log p)n^{-\kappa_3}/(\gamma_n+\psi_n)})+o(1)\\
&\leq_{(7)}
\gamma_n+\psi_n+C\omega_n+\Pr(\mathcal{A}_n)+o(1)=_{(8)}\gamma_n+o(1)
\end{align*}
where (1) follows from the definitions of $\hat{f}(s)$ and $\hat{\eps}(s)$, (2) is by the definition of $\psi_n$, (3) is by the definition of $\MS_n^R$, (4) is rearrangement, (5) is by Lemma \ref{lem: maximal inequality} and Assumption A\ref{as: V}, (6) is by Lemmas \ref{lem: anticoncentration} and \ref{lem: multiplied quantiles}, (7) is by Lemma \ref{lem: gaussian approximation} and definitions of $\omega_n$ and $\mathcal{A}_n$, and (8) follows from the definition of $\psi_n$ again. The first asserted claim follows. The second claim follows from the fact that $\MS_n^{SD}\subset\MS_n^{OS}$.
\end{proof}

\begin{lemma}\label{lem: truncation}
Let Assumptions A\ref{as: disturbances}, A\ref{as: V}, and A\ref{as: growth condition} hold. In addition, let either Assumption A\ref{as: function estimation} or A\ref{as: sigma} hold. Then
$\Pr(\max_{s\in\MS_n\backslash\MS_n^R}(\hat{f}(s)+\hat{\eps}(s))\leq0)\geq1-\gamma_{n}+o(1)$.\end{lemma}
\begin{proof}
The result follows from
\begin{align*}
\Pr(\max_{s\in\MS_n\backslash\MS_n^R}(\hat{f}(s)+\hat{\eps}(s))\leq 0)&=
\Pr(\max_{s\in\MS_n\backslash\MS_n^R}(f(s)+\eps(s))\leq 0)\\
&\geq_{(1)}
\Pr(\max_{s\in\MS_n\backslash\MS_n^R}\eps(s)\leq c_{1-\gamma_n-\psi_n}^{PI,0})\\
&\geq
\Pr(\max_{s\in\MS_n}\eps(s)\leq c_{1-\gamma_n-\psi_n}^{PI,0})\\
&=_{(2)}
1-\gamma_n-\psi_n+o(1)=_{(3)}1-\gamma_n+o(1)
\end{align*}
where (1) follows from the definition of $\MS_n^R$, (2) is by Lemma \ref{lem: gaussian approximation}, and (3) is by the definition of $\psi_n$.
\end{proof}

\subsection{Proofs of Theorems}\label{sec: proof of theorems}
\begin{proof}[Proof of Theorem \ref{thm: size}]
All inequalities presented in this proof hold uniformly over all models $M\in\mathcal{M}$ because all of them only rely on assumptions and lemmas in Subsection \ref{sec: auxiliary lemmas}, which hold uniformly over all models $M\in\mathcal{M}$. For brevity of notation, I drop indexing by $M$.

By the definition of $\psi_n$, $\log p/n^{\kappa_3}=o_p(\psi_n)$, and so there exists a sequence $\{\omega_n\}$ of positive numbers converging to zero such that $\log p/n^{\kappa_3}=o_p(\omega_n\psi_n)$. Let $\mathcal{A}_n$ denote the event that $\log p/n^{\kappa_3}>\omega_n\psi_n$. Then $\Pr(\mathcal{A}_n)=o(1)$.

Further, note that for some $C>0$,
\begin{eqnarray*}
&&\Pr(T\leq c_{1-\alpha}^P)=\Pr(\max_{s\in\MS_n}(\hat{f}(s)+\hat{\eps}(s))\leq c_{1-\alpha}^P)
\geq_{(1)}
\Pr(\max_{s\in\MS_n^R}(\hat{f}(s)+\hat{\eps}(s))\leq c_{1-\alpha}^P)-\gamma_n+o(1)\\
&&\geq_{(2)}
\Pr(\max_{s\in\MS_n^R}(\hat{f}(s)+\hat{\eps}(s))\leq c_{1-\alpha}^R)-2\gamma_n+o(1)
\geq_{(3)}
\Pr(\max_{s\in\MS_n^R}\hat{\eps}(s)\leq c_{1-\alpha}^R)-2\gamma_n+o(1)\\
&&\geq_{(4)}
\Pr(\max_{s\in\MS_n^R}\eps(s)\mathcal{V}\leq c_{1-\alpha-\psi_n}^{R,0})-2\gamma_n+o(1)
=_{(5)}
\Pr(\max_{s\in\MS_n^R}\eps(s)\leq c_{1-\alpha-\psi_n}^{R,0}/\mathcal{V})-2\gamma_n+o(1)\\
&&\geq_{(6)}
\Pr(\max_{s\in\MS_n^R}\eps(s)\leq c_{1-\alpha-\psi_n}^{R,0}(1-n^{-\kappa_3}))-2\gamma_n+o(1)\\
&&\geq_{(7)}
\Pr(\max_{s\in\MS_n^R}\eps(s)\leq c^{R,0}_{1-\alpha-\psi_n-C(\log p)n^{-\kappa_3}/(\alpha+\psi_n)})-2\gamma_n+o(1)\\
&&\geq_{(8)}
1-\alpha-\psi_n-C\omega_n-\Pr(\mathcal{A}_n)-2\gamma_n+o(1)=_{(9)}1-\alpha+o(1)
\end{eqnarray*}
where (1) follows from Lemma \ref{lem: truncation}, (2) is by Lemma \ref{lem: inclusion R and SD}, (3) is because under $\MH_0$, $\hat{f}(s)\leq 0$, (4) follows from the definitions of $\hat{\eps}(s)$ and $\psi_n$, (5) is rearrangement, (6) is by Assumption A\ref{as: V}, (7) is by Lemma \ref{lem: multiplied quantiles}, (8) is by Lemma \ref{lem: gaussian approximation} and definitions of $\omega_n$ and $\mathcal{A}_n$, and (9) is by the definitions of $\psi_n$ and $\gamma_n$. The first asserted claim follows.

In addition, when $f(\cdot)$ is identically constant,
\begin{eqnarray*}
&&\Pr(T\leq c_{1-\alpha}^P)=_{(1)}\Pr(\max_{s\in\MS_n}\hat{\eps}(s)\leq c_{1-\alpha}^P)
\leq_{(2)}
\Pr(\max_{s\in\MS_n}\hat{\eps}(s)\leq c_{1-\alpha}^{PI})\\
&&\leq_{(3)}
\Pr(\max_{s\in\MS_n}\hat{\eps}(s)\leq c_{1-\alpha+\psi_n}^{PI,0})+o(1)
\leq_{(4)}
\Pr(\max_{s\in\MS_n}\eps(s)\leq c_{1-\alpha+\psi_n}^{PI,0}(1+n^{-\kappa_3}))+o(1)\\
&&\leq_{(5)}
\Pr(\max_{s\in\MS_n}\eps(s)\leq c^{PI,0}_{1-\alpha+\psi_n+C(\log p)n^{-\kappa_3}/(\alpha-\psi_n)})+o(1)\leq_{(6)} 1-\alpha+o(1)
\end{eqnarray*}
where (1) follows from the fact that $\hat{f}(s)=0$ whenever $f(\cdot)$ is identically constant, (2) follows from $\MS_n^P\subset\MS_n$, (3) is by the definition of $\psi_n$, (4) is by Assumption A\ref{as: V}, (5) is by Lemma \ref{lem: multiplied quantiles}, and (6) is from Lemma \ref{lem: gaussian approximation} and the definition of $\psi_n$. The second asserted claim follows.
\end{proof}

\begin{proof}[Proof of Theorem \ref{thm: consistency}]
Let $x_1,x_2\in[s_l,s_r]$ be such that $x_1<x_2$ but $f(x_1)>f(x_2)$. By the mean
value theorem, there exists $x_{0}\in(x_{1},x_{2})$ satisfying
$$
f^{\prime}(x_{0})(x_{2}-x_{1})=f(x_{2})-f(x_{1})<0.
$$
Therefore, $f^{\prime}(x_{0})<0$. Since $f^{\prime}(\cdot)$ is continuous,
$f^{\prime}(x)<f^{\prime}(x_{0})/2$ for any $x\in[x_{0}-\Delta_x,x_{0}+\Delta_x]$
for some $\Delta_x>0$. Apply Assumption A\ref{as: weighting functions simple} to the interval $[x_{0}-\Delta_x,x_{0}+\Delta_x]$ to obtain $s=s_n\in\MS_n$ and intervals $[x_{l1},x_{r1}]$ and $[x_{l2},x_{r2}]$ appearing in Assumption A\ref{as: weighting functions simple}-iii. By Assumption A\ref{as: weighting functions simple}-ii, 
\begin{equation}\label{eq: variance n3 bound}
V(s)\leq Cn^3
\end{equation}
for some $C>0$. In addition, by Assumption A\ref{as: design simple} and Chebyshev's inequality, there is some $C>0$ such that $|\{i=\overline{1,n}:X_i\in[x_{l1},x_{r1}]\}|\geq Cn$ and $|\{i=\overline{1,n}:X_i\in[x_{l2},x_{r2}]\}|\geq Cn$ w.p.a.1. Hence, by Assumptions A\ref{as: weighting functions simple}-(i) and A\ref{as: weighting functions simple}-(iii),
\begin{equation}\label{eq: something}
\sum_{1\leq i,j\leq n}(f(X_i)-f(X_j))\sign(X_j-X_i)Q(X_i,X_j,s)\geq Cn^2
\end{equation}
w.p.a.1 for some $C>0$. Combining (\ref{eq: variance n3 bound}) and (\ref{eq: something}) yields $f(s)\geq Cn^{1/2}$ w.p.a.1 for some $C>0$. Let $\mathcal{A}_n$ denote the event that $f(s)< Cn^{1/2}$, so that $\Pr(\mathcal{A}_n)=o(1)$.
Further, since $\sum_{1\leq i\leq n}a_i(s)^2\sigma_i^2=1$, Assumption A\ref{as: disturbances} implies $A_n\geq C/n^{1/2}$ for some $C>0$, and so Assumption A\ref{as: growth condition} gives $\log p=o_p(n)$, and so 
\begin{equation}\label{eq: log p bound via fs}
(\log p)^{1/2}=o_p(f(s)).
\end{equation}
Therefore, there exists a sequence $\{\omega_n\}$ of positive numbers converging to zero such that $(\log p)^{1/2}=o_p(\omega_nf(s))$, and so letting $\mathcal{B}_n$ denote the event that $(\log p)^{1/2}>\omega_nf(s)$ leads to $\Pr(\mathcal{B}_n)=o(1)$.
Hence, for some $C>0$,
\begin{align*}
\Pr(T\leq c_{1-\alpha}^P)&\leq_{(1)}\Pr(T\leq c_{1-\alpha}^{PI})\leq_{(2)} \Pr(T\leq c_{1-\alpha+\psi_n}^{PI,0})+o(1)\\
&\leq_{(3)}
\Pr(T\leq C(\log p)^{1/2})+o(1)
\leq_{(4)}
\Pr(\hat{f}(s)+\hat{\eps}(s)\leq C(\log p)^{1/2})+o(1)\\
&\leq_{(5)}
\Pr(f(s)+\eps(s)\leq C(\log p)^{1/2}(1+n^{-\kappa_3}))+o(1)\\
&\leq_{(6)}
\Pr(f(s)+\eps(s)\leq 2C(\log p)^{1/2})+o(1)\\
&\leq_{(7)}
\Pr(\eps(s)\leq 2C(\log p)^{1/2}-f(s))+o(1)\\
&\leq_{(8)}
\Pr(\eps(s)\leq (2C\omega_n-1)f(s))+\Pr(\mathcal{B}_n)+o(1)\\
&\leq_{(9)}
C/n+\Pr(\mathcal{A}_n)+\Pr(\mathcal{B}_n)+o(1)=_{(10)}o(1)
\end{align*}
where (1) follows from $\MS_n^P\subset\MS_n^{PI}$, (2) is by the definition of $\psi_n$, (3) is by Lemma \ref{lem: maximal inequality}, (4) is since $T=\max_{s\in\MS_n}(\hat{f}(s)+\hat{\eps}(s))$, (5) is by Assumption A\ref{as: V}, (6) is because $n^{-\kappa_3}\leq 1$ for large $n$, (7) is rearrangement, (8) follows from definitions of $\omega_n$ and $\mathcal{B}_n$, (9) follows by noting that $2C\omega_n\leq 1/2$ for large $n$ and applying Chebyshev's inequality conditional on $\{X_i\}$ with $\Ep[\varepsilon(s)^2|\{X_i\}]=1$, and (10) follows from $\Pr(\mathcal{A}_n)+\Pr(\mathcal{B}_n)=o(1)$. 
This completes the proof of the theorem.
\end{proof}

\begin{proof}[Proof of Theorem \ref{thm: consistency rate against one-dimensional alternatives}]
The proof follows from an argument similar to that used in the proof of Theorem \ref{thm: consistency} with equation (\ref{eq: something}) replaced by
$$
\sum_{1\leq i,j\leq n}(f(X_i)-f(X_j))\sign(X_j-X_i)Q(X_i,X_j,s)\geq Cl_nn^2,
$$
so that $f(s)\geq Cl_nn^{1/2}$ w.p.a.1 for some $C>0$,
and condition $\log p=o_p(n)$ replaced by $\log p=o_p(l_n^2n)$, so that $(\log p)^{1/2}=o_p(f(s))$.
\end{proof}

\begin{proof}[Proof of Theorem \ref{thm: uniform consistency rate}]
Since $\inf_{x\in[s_l,s_r]}f^{(1)}(x)<-l_n(\log p/n)^{\beta/(2\beta+3)}$, for sufficiently large $n$, there exists an interval $[x_{n,1},x_{n,2}]\subset[s_l,s_r]$ such that $|x_{n,2}-x_{n,1}|= h_n$ and for all $x\in[x_{n,1},x_{n,2}]$, $f^{(1)}(x)<-l_n(\log p/n)^{\beta/(2\beta+3)}/2$. Take $s=s_n\in\MS_n$ as in Assumption A\ref{as: weighting functions difficult} applied to the interval $[x_{n,1},x_{n,2}]$.
Assumptions A\ref{as: disturbances}, A\ref{as: design difficult}, and A\ref{as: weighting functions difficult}-(ii) yield $V(s)\leq C(n h_n)^3h_n^{2k}$ for some $C>0$ because by Lemma \ref{lem: technical lemma}, $|\{i=\overline{1,n}:X_i\in[x_{n,1},x_{n,2}]\}|\leq Cn h_n$ w.p.a.1. In addition, combining Assumptions A\ref{as: design difficult}, A\ref{as: weighting functions difficult}-(i), and A\ref{as: weighting functions difficult}-(iii),
$$
\sum_{1\leq i,j\leq n}(f(X_i)-f(X_j))\sign(X_j-X_i)Q(X_i,X_j,s)\geq l_n Ch_n^{1+\beta+k}(n h_n)^2
$$
w.p.a.1 for some $C>0$, and so $f(s)\geq Cl_n h_n^{1+\beta}(n h_n)^{1/2}$ w.p.a.1 by an argument similar to that used in the proof of Theorem \ref{thm: consistency}. From this point, since $\log p=o(l_n^2h_n^{2\beta+3}n)$, the argument like that used in the proof of Theorem \ref{thm: consistency} yields the result.
\end{proof}

\begin{proof}[Proof of Theorem \ref{thm: minimax lower bound}]
Let $s_l=0$ and $s_r=1$. Let $h=h_n=c(\log n/n)^{1/(2\beta+3)}$ for sufficiently small $c>0$. Let $L=[(s_r-s_l)/(4h)]$ where $[x]$ is the largest integer smaller or equal than $x$. For $l=\overline{1,L}$, let $x_l=4h(l-1)$ and define $f_l:[s_l,s_r]\rightarrow\RR$ by $f_l(s_l)=0$, $f_l^{(1)}(x)=0$ if $x\leq x_l$, $f_l^{(1)}(x)=-L(x-x_l)^{\beta}$ if $x\in(x_l,x_l+h]$, $f_l^{(1)}(x)=-L(x_l+2h-x)^{\beta}$ if $x\in(x_l+h,x_l+2h]$, $f_l^{(1)}(x)=L(x-x_l-2h)^{\beta}$ if $x\in(x_l+2h,x_l+3h]$, $f_l^{(1)}(x)=L(x_l+4h-x)^{\beta}$ if $x\in(x_l+3h,x_l+4h]$ and $f_l^{(1)}(x)=0$ otherwise. In addition, let $f_0(x)=0$ for all $x\in[s_l,s_r]$. Finally, let $\{\eps_i\}$ be a sequence of independent $N(0,1)$ random variables. 

For $l=\overline{0,L}$, consider a model $M_l=M_{n,l}$ with the regression function $f_l$, $X$ distributed uniformly over $[0,1]$, and $\varepsilon$ distributed as $N(0,1)$ random variable independently of $X$. Note that $M_0$ belongs to $\mathcal{M}_2$ and satisfies $\MH_0$. In addition, for $l\geq 1$, $M_l$ belongs to $\mathcal{M}_2$, does not satisfy $\MH_0$, and, moreover, has $\inf_{x\in[s_l,s_r]}f_l^{(1)}(x)<-C(\log n/n)^{\beta/(2\beta+3)}$ for some small $C>0$.

Consider any test $\psi=\psi(\{X_i,Y_i\})$ such that $\Ep_{M_0}[\psi]\leq \alpha+o(1)$. Then following the argument from \cite{Dumbgen2001} gives
\begin{align*}
\inf_{M\in\mathcal{M}_2}\Ep_M[\psi]-\alpha
&\leq
\min_{1\leq l\leq L}\Ep_{M_l}[\psi]-\Ep_{M_0}[\psi]+o(1)
\leq
\sum_{1\leq l\leq L}\Ep_{M_l}[\psi]/L-\Ep_{M_0}[\psi]+o(1)\\
&=
\sum_{1\leq l\leq L}\Ep_{M_0}[\psi\rho_l]/L-\Ep_{M_0}[\psi]+o(1)
=\sum_{1\leq l\leq L}\Ep_{M_0}[\psi(\rho_l-1)]/L+o(1)\\
&\leq
\Ep_{M_0}\left[\psi\left|\sum_{1\leq l\leq L}\rho_l/L-1\right|\right]+o(1)
\leq
\Ep_{M_0}\left[\left|\sum_{1\leq l\leq L}\rho_l/L-1\right|\right]+o(1)
\end{align*}
where $\rho_l$ is the likelihood ratio of observing $\{X_i,Y_i\}$ under the models $M_l$ and $M_0$. Further,
$$
\rho_l=\exp\(\sum_{1\leq i\leq n}Y_if_l(X_i)-\sum_{1\leq i\leq n}f_l(X_i)^2/2\)=\exp(\omega_{n,l}\xi_{n,l}-\omega_{n,l}^2/2)
$$
where $\omega_{n,l}=(\sum_{1\leq i\leq n}f_l(X_i)^2)^{1/2}$ and $\xi_{n,l}=\sum_{1\leq i\leq n}Y_if_l(X_i)/\omega_{n,l}$. Note that under the model $M_0$, conditional on $\{X_i\}$, $\{\xi_{n,l}\}_{1\leq l\leq L}$ is a sequence of independent $N(0,1)$ random variables, so that
$$
\Ep_{M_0}\left[\left|\sum_{1\leq l\leq L}\rho_l/L-1\right||\{X_i\}\right]\leq 2.
$$
In addition, by the construction of the functions $f_l$, Lemma \ref{lem: technical lemma} shows that $\omega_{n,l}\leq C(nh)^{1/2}h^{\beta+1}$ w.p.a.1 for some $C>0$, and so $\omega_{n,l}\leq (C\log n)^{1/2}$ w.p.a.1 for some $C>0$ where $C$ can be made arbitrarily small by selecting sufficiently small $c$ in the definition of $h_n$. Let $A_n$ denote the event $\omega_{n,l}\leq (C\log n)^{1/2}$, so that $\Pr(A_n)\rightarrow 1$. Moreover,
\begin{equation}\label{eq: expectation null model bound}
\Ep_{M_0}\left[\left|\sum_{1\leq l\leq L}\rho_l/L-1\right|\right]\leq \Ep_{M_0}\left[\left|\sum_{1\leq l\leq L}\rho_l/L-1\right||A_n\right]+2(1-\Pr(A_n)).
\end{equation}
Further, on $A_n$, 
\begin{align*}
\Ep_{M_0}\left[\left|\sum_{1\leq l\leq L}\rho_l/L-1\right||\{X_i\}\right]^2
&\leq
\Ep_{M_0}\left[\left(\sum_{1\leq l\leq L}\rho_l/L-1\right)^2|\{X_i\}\right]\\
&\leq
\sum_{1\leq l\leq L}\Ep_{M_0}[\rho_l^2/L^2|\{X_i\}]\\
&\leq
\sum_{1\leq l\leq L}\Ep_{M_0}[\exp(2\omega_{n,l}\xi_{n,l}-\omega_{n,l}^2)/L^2|\{X_i\}]\\
&\leq
\sum_{1\leq l\leq L}\exp(\omega_{n,l}^2)/L^2\leq \max_{1\leq l\leq L}\exp(\omega_{n,l}^2)/L\\
&\leq
\exp\left(C\log n-\log L\right)=o(1)
\end{align*}
because $C$ in the last line is arbitrarily small and $\log n\lesssim \log L$. Combining the last bound with (\ref{eq: expectation null model bound}) gives  $\inf_{M\in\mathcal{M}_2}\Ep_{M}[\psi]\leq \alpha+o(1)$, and so the result follows.
\end{proof}

\section{Proofs for Section \ref{sec: multivariate models}}
\begin{proof}[Proof of Theorem \ref{thm: nonparametric model}]
In the set-up of this theorem, Lemmas \ref{lem: maximal inequality}, \ref{lem: anticoncentration}, \ref{lem: multiplied quantiles}, \ref{lem: conditional and unconditional quantiles 1}, \ref{lem: conditional and unconditional quantiles 3}, \ref{lem: inclusion R and SD}, and \ref{lem: truncation} hold with $X_i$ replaced by $X_i,Z_i$ and Assumptions A\ref{as: disturbances}, A\ref{as: function estimation}, and A\ref{as: growth condition} replaced by Assumptions A1$'$, A2$'$, and A5$'$, respectively. Further, Lemma \ref{lem: gaussian approximation} now follows by applying Lemma \ref{lem: gaussian approximation original}, case (i), and assumptions A1 and A5 replaced by A1$'$ and A5$'$. The result of Lemma \ref{lem: conditional and unconditional quantiles 2} holds but the proof has to be modified. Recall that in the proof of Lemma \ref{lem: conditional and unconditional quantiles 2}, it was derived that
$$
\Ep\left[\Delta_{\Sigma}|\{X_i,Z_i\}\right]\leq C\left(A_n\sqrt{\log p}+A_n^2\log p\Ep\left[\max_{1\leq i\leq n}\varepsilon_i^4\right]^{1/2}\right).
$$
Now, $A_n(\log p)^{5/2}=o_p(1)$ by Assumption A5$'$. In addition, by a maximal inequality (Lemma 2.2.2 in \cite{VaartWellner1996}),
$$
\Ep\left[\max_{1\leq i\leq n}\varepsilon_i^4\right]\leq C(\log n)^4
$$
for some $C>0$ under Assumption A1$'$. Therefore,
$$
A_n^2(\log p)^3\Ep\left[\max_{1\leq i\leq n}\varepsilon_i^4\right]^{1/2}\leq CA_n^2(\log(p n))^5=o_p(1)
$$
for some $C>0$ under Assumption A5$'$. Hence, $(\log p)^2\Delta_{\Sigma}=o_p(1)$. The rest of the proof of Lemma \ref{lem: conditional and unconditional quantiles 2} is the same.

Now, the first claim of Theorem \ref{thm: nonparametric model} follows from an argument similar to that used in the proof of Theorem \ref{thm: size} by noting that under $\MH_0$, $(\log p)^{1/2}\max_{s\in\MS_n}\hat{f}(s)\leq o_p(1)$\footnote{Specifically, the only required change in the proof of the first claim of Theorem \ref{thm: size} is that starting from inequality (3) $\alpha$ should be replaced by $\alpha-o_p(1)$.}, which holds by Assumption A\ref{as: multivariate nonparametric model}. The second claim of the theorem again follows from an argument similar to that used in the proof of Theorem \ref{thm: size} by noting that when $f(\cdot)$ is identically constant, $(\log p)^{1/2}\max_{s\in\MS_n}|\hat{f}(s)|\leq o_p(1)$, which holds by Assumption A\ref{as: multivariate nonparametric model}.
\end{proof}

\begin{proof}[Proof of Theorem \ref{thm: partially linear model}]
Denote $Y_{i}^{0}=f(X_{i})+\eps_{i}$. Then $Y_{i}=Y_{i}^{0}+Z_{i}^{T}\beta$
and $\tilde{Y}_{i}=Y_{i}^{0}-Z_{i}^{T}(\hat{\beta}_n-\beta)$. So,
$|\tilde{Y}_{i}-Y_{i}^{0}|\leq\Vert Z_{i}\Vert\Vert\hat{\beta}_n-\beta\Vert$
uniformly over $i=1,...,n$ and all models in $\mathcal{M}_{PL}$. 
Further, note that since $\sum_{1\leq i,j\leq n}Q(X_i,X_j,s)/V(s)^{1/2}=o_p(\sqrt{n/\log p})$ uniformly over $s\in\MS_n$, there exists a sequence $\{\omega_n\}$ such that $\omega_n\rightarrow 0$ and $\sum_{1\leq i\leq n}|a_i(s)|=o_p(\omega_n\sqrt{n/\log p})$ uniformly over $s\in\MS_n$. Let $\mathcal{A}_n$ denote the event that $\sum_{1\leq i\leq n}|a_i(s)|>\omega_n\sqrt{n/\log p}$, so that $\Pr(\mathcal{A}_n)=o(1)$.
In addition, note that for any $C>0$,
\begin{align*}
\Pr(\max_{s\in\MS_n}\sum_{1\leq i\leq n}|\hat{a}_i(s)(\tilde{Y}_i-Y_i^0)|>C/\sqrt{\log p})&\leq \Pr(\max_{s\in\MS_n}\sum_{1\leq i\leq n}|a_i(s)(\tilde{Y}_i-Y_i^0)|>C/\sqrt{2\log p})+o(1)\\
&\leq \Pr(C_9\|\hat{\beta}-\beta\|\max_{s\in\MS_n}\sum_{1\leq i\leq n}|a_i(s)|>C/\sqrt{2\log p})+o(1)\\
&\leq \Pr(\|\hat{\beta}-\beta\|>C/(C_9\omega_n\sqrt{2n}))+\Pr(\mathcal{A}_n)+o(1)=o(1)
\end{align*}
where the last conclusion follows from Assumption A\ref{as: partially linear model} and from $\Pr(\mathcal{A}_n)=o(1)$.
Therefore, for any $C>0$,
$$
\Pr\left(|T-\max_{s\in\MS_n}\sum_{1\leq i\leq n}\hat{a}_i(s)Y_i^0|>C/\sqrt{\log p}\right)\leq \Pr\left(\max_{s\in\MS_n}\sum_{1\leq i\leq n}|\hat{a}_i(s)(\tilde{Y}_i-Y_i^0)|>C/\sqrt{\log p}\right)=o(1).
$$
From this point, in view of of Lemma \ref{lem: anticoncentration}, the result follows by the argument similar to that used in the proof of Theorem \ref{thm: size}.
\end{proof}

\section{Proof for Section \ref{sec: endogenous covariates}}
\begin{proof}[Proof of Theorem \ref{thm: model with endogenous covariates}]
The proof follows from the same argument as in the proof of Theorem \ref{thm: partially linear model} and Comment \ref{com: separately additive model}.
\end{proof}




\section{Proofs for Appendix \ref{sec: verification of high level conditions}}
\begin{proof}[Proof of Proposition \ref{lem: consistency of rice estimator}]
Since $\sigma_i\geq c_1$ by Assumption A\ref{as: disturbances}, it follows that
$$
|\hat{\sigma}_i-\sigma_i|=\left|\frac{\hat{\sigma}_i^2-\sigma_i^2}{\hat{\sigma}_i+\sigma_i}\right|\leq |\hat{\sigma}_i^2-\sigma_i^2|/c_1,
$$
so that it suffices to prove that $\hat{\sigma}_i^2-\sigma_i^2=O_p(b_n^2+\log n/(b_nn^{1/2}))$ uniformly over $i=\overline{1,n}$. Now, note that
$$
\hat{\sigma}_i^2-\sigma_i^2=\sum_{j\in J(i)^\prime}(Y_{j+1}-Y_j)^2/(2|J(i)|)-\sigma_i^2
$$
where $J(i)^\prime=\{j\in J(i):j+1\in J(i)\}$. Therefore,
$$
\hat{\sigma}_i^2-\sigma_i^2=p_{i,1}+p_{i,2}+p_{i,3}-p_{i,4}+O_p(b_n^2)
$$
where
\begin{align*}
&p_{i,1}=\sum_{j\in J(i)^\prime}(\eps_{j+1}^2-\sigma_j^2)/(2|J(i)|),\\
&p_{i,2}=\sum_{j\in J(i)^\prime}(\eps_{j}^2-\sigma_j^2)/(2|J(i)|),\\
&p_{i,3}=\sum_{j\in J(i)^\prime}(f(X_{j+1})-f(X_j))(\eps_{j+1}-\eps_j)/|J(i)|,\\
&p_{i,4}=\sum_{j\in J(i)^\prime}\eps_j\eps_{j+1}/|J(i)|.
\end{align*}
Since $b_n\geq (\log n)^2/(C_6n)$, Lemma \ref{lem: technical lemma} implies that $cb_nn\leq |J(i)|\leq Cb_nn$ for all $i=\overline{1,n}$ w.p.a.1 for some $c,C>0$. Note also that $J(i)$'s depend only on $\{X_i\}$. Therefore, applying Lemma \ref{lem: maximal inequality CCK} conditional on $\{X_i\}$ shows that $p_{i,1}$ and $p_{i,2}$ are both $O_p(\log n/(b_nn^{1/2}))$ uniformly over $i=\overline{1,n}$ since $\Ep[\max_{1\leq i\leq n}\eps_i^4|\{X_i\}]\leq C_1n$ by Assumption A\ref{as: disturbances}. Further, applying Lemma \ref{lem: maximal inequality CCK} conditional on $\{X_i\}$ separately to $\sum_{j\in J(i)^\prime:j\text{ odd}}\eps_j\eps_{j+1}/|J(i)|$, $\sum_{j\in J(i)^\prime:j\text{ even}}\eps_j\eps_{j+1}/|J(i)|$, $\sum_{j\in J(i)^\prime}(f(X_{j+1})-f(X_j))\eps_{j+1}/|J(i)|$, and $\sum_{j\in J(i)^\prime}(f(X_{j+1})-f(X_j))\eps_{j}/|J(i)|$ shows that $p_{i,3}$ and $p_{i,4}$ are also $O_p(\log n/(b_nn^{1/2}))$ uniformly over $i=\overline{1,n}$. Combining these bounds gives the claim of the proposition and completes the proof.
\end{proof}

\begin{proof}[Proof of Proposition \ref{lem: kernel function}]
Let $s=(x,h)\in\MS_n$. Since $h\leq (s_r-s_l)/2$, either $s_l+h\leq x$ or $x+h\leq s_r$ holds. Let $\MS_{n,1}$ and $\MS_{n,2}$ denote the subsets of those elements of $\MS_n$ that satisfy the former and latter cases, respectively, so that $\MS_n=\MS_{n,1}\cup\MS_{n,2}$. Consider $\MS_{n,1}$. Let $C_K\in(0,1)$. Since the kernel $K(\cdot)$ is continuous and strictly positive on the interior of its support, $\min_{t\in[-C_K,0]}K(t)>0$. In addition, since $K(\cdot)$ is bounded, it is possible to find a constant $c_K\in(0,1)$ such that and $c_K+C_K\leq 1$ and
\begin{equation}\label{eq: fancy kernel constraint}
6c_K^{k+1}C_3\max_{t\in[-1,-1+c_K]}K(t)\leq c_3(1-c_K)^k C_K\min_{t\in[-C_K,0]}K(t)
\end{equation}
where the constant $k$ appears in the definition of kernel weighting functions.

Denote
\begin{align*}
M_{n,1}(x,h)&=\{i=\overline{1,n}:X_i\in[x-C_K h,x]\},\\
M_{n,2}(x,h)&=\{i=\overline{1,n}:X_i\in[x-h,x-(1-c_K)h],\}\\
M_{n,3}(x,h)&=\{i=\overline{1,n}:X_i\in[x-(1-c_K/2)h,x-(1-c_K)h]\},\\
M_{n,4}(x,h)&=\{i=\overline{1,n}:X_i\in[x-h,x+h]\}.
\end{align*}
Since $h_{\min}\geq (\log n)^2/(C_6n)$ w.p.a.1 by assumption, Lemma \ref{lem: technical lemma} and Assumption \ref{as: design difficult} give
\begin{align}
(1/2)c_3C_K n h\leq|M_{n,1}(x,h)|&\leq (3/2)C_3C_K n h,\nonumber\\
(1/2)c_3c_K n h\leq |M_{n,2}(x,h)|&\leq (3/2)C_3c_K n h,\nonumber\\
(1/2)c_3(c_K/2)n h\leq |M_{n,3}(x,h)|&\leq (3/2)C_3(c_K/2)nh,\nonumber\\
(1/2)c_3n h\leq |M_{n,4}(x,h)|&\leq (3/2)C_32n h\label{eq: Mnh4}
\end{align}
w.p.a.1. uniformly over $s=(x,h)\in\MS_{n,1}$. Note also that (\ref{eq: Mnh4}) holds w.p.a.1 uniformly over $s=(x,h)\in\MS_{n,2}$ as well.
Then
\begin{eqnarray*}
\sum_{1\leq j\leq n}\sign(X_j-X_i)|X_j-X_i|^kK((X_j-x)/h)
&\geq&
\sum_{j\in M_{n,1}(x,h)}((1-c_K)h)^k K((X_j-x)/h) \\
&&- \sum_{j\in M_{n,2}(x,h)}(c_Kh)^kK((X_j-x)/h)\\
&\geq&
((1-c_k)h)^k(1/2)c_3C_knh\min_{t\in[-C_K,0]}K(t)\\
&& - (c_Kh)^k(3/2)C_3c_Knh\max_{t\in[-1,-1+c_k]}K(t)\\
&\geq&
((1-c_K)h)^kc_3C_Knh\min_{t\in[-C_K,0]}K(t)/4\\
&\geq& C n h^{k+1}
\end{eqnarray*}
w.p.a.1 uniformly over $X_i\in M_{n,3}(x,h)$ and $s=(x,h)\in\MS_{n,1}$ for some $C>0$ where the inequality preceding the last one follows from (\ref{eq: fancy kernel constraint}). Hence,
\begin{align*}
V(s)&=\sum_{1\leq i\leq n}\sigma_i^2\(\sum_{1\leq j\leq n}\sign(X_j-X_i)Q(X_i,X_j,s)\)^2\\
&=\sum_{1\leq i\leq n}\sigma_i^2K((X_i-x)/h)^2\(\sum_{1\leq j\leq n}\sign(X_j-X_i)|X_j-X_i|^kK((X_j-x)/h)\)^2\\
&\geq
\sum_{i\in M_{n,3}(x,h)}\sigma_i^2K((X_i-x)/h)^2\(\sum_{1\leq j\leq n}\sign(X_j-X_i)|X_j-X_i|^kK((X_j-x)/h)\)^2,
\end{align*}
and so $V(s)\geq C(nh)^3h^{2k}$ w.p.a.1 uniformly over $s=(x,h)\in\MS_{n,1}$ for some $C>0$. Similar argument gives $V(s)\geq C(nh)^3h^{2k}$ w.p.a.1 uniformly over $s=(x,h)\in\MS_{n,2}$ for some $C>0$. In addition,
\begin{equation}\label{eq: denominator one bound}
\left|\sum_{1\leq j\leq n}\sign(X_j-X_i)Q(X_i,X_j,s)\right|\leq (2h)^k|M_{n,4}(x,h)|\left(\max_{t\in[-1,+1]}K(t)\right)^2 \leq Cnh^{k+1}
\end{equation}
w.p.a.1 uniformly over $i=\overline{1,n}$ and $s=(x,h)\in\MS_{n}$ for some $C>0$. Combining (\ref{eq: denominator one bound}) with the bound on $V(s)$ above gives $A_n\leq C/(nh_{\min})$ w.p.a.1 for some $C>0$. Now, for the basic set of weighting functions, $\log p\leq C\log n$ and $h_{\min}\geq Cn^{-1/3}$ w.p.a.1 for some $C>0$, and so Assumption A\ref{as: growth condition} holds, which gives the claim (a).

To prove the claim (b), note that
\begin{align}
\sum_{1\leq i\leq n}\(\sum_{1\leq j\leq n}\sign(X_j-X_i)Q(X_i,X_j,s)\)^2&\leq (2h)^{2k}|M_{n,4}(x,h)|^3\left(\max_{t\in[-1,+1]}K(t)\right)^4\label{eq: variance calculations bound 1}\\
&\leq C(nh)^3h^{2k}\label{eq: variance calculations bound 2}
\end{align}
w.p.a.1 uniformly over $s=(x,h)\in\MS_n$ for some $C>0$ by (\ref{eq: Mnh4}). Therefore, under Assumption A\ref{as: sigma},
\begin{align*}
|\hat{V}(s)-V(s)|&\leq \max_{1\leq i\leq n}|\hat{\sigma}_i^2-\sigma_i^2|\sum_{1\leq i\leq n}\(\sum_{1\leq j\leq n}\sign(X_j-X_i)Q(X_i,X_j,s)\)^2\\
&\leq (n h)^3h^{2k}o_p(n^{-\kappa_2})
\end{align*}
uniformly over $s=(x,h)\in\MS_{n}$. Combining this bound with the lower bound for $V(s)$ established above shows that under Assumption A\ref{as: sigma}, $|\hat{V}(s)/V(s)-1|=o_p(n^{-\kappa_2})$, and so 
\begin{eqnarray*}
|(\hat{V}(s)/V(s))^{1/2}-1|&=&o_p(n^{-\kappa_2}),\\
|(V(s)/\hat{V}(s))^{1/2}-1|&=&o_p(n^{-\kappa_2})
\end{eqnarray*}
uniformly over $\MS_{n}$, which is the asserted claim (b).

To prove the last claim, note that for $s=(x,h)\in\MS_{n}$,
$$
|\hat{V}(s)-V(s)|\leq I_1(s)+I_2(s)
$$
where
\begin{align*}
I_1(s)&=\left|\sum_{1\leq i\leq n}(\eps_i^2-\sigma_i^2)\(\sum_{1\leq j\leq n}\sign(X_j-X_i)Q(X_i,X_j,s)\)^2\right|,\\
I_2(s)&=\left|\sum_{1\leq i\leq n}(\hat{\sigma}_i^2-\eps_i^2)\(\sum_{1\leq j\leq n}\sign(X_j-X_i)Q(X_i,X_j,s)\)^2\right|.
\end{align*}
Consider $I_1(s)$. Combining (\ref{eq: denominator one bound}) and Lemma \ref{lem: technical lemma} applied conditional on $\{X_i\}$ gives
\begin{equation}\label{eq: variance derivations 1}
I_1(s)=(n h)^{2}h^{2k}\log p O_p(n^{1/2})
\end{equation}
uniformly over $s=(x,h)\in \MS_n$.

Consider $I_2(s)$. Note that 
$$
I_2(s)\leq I_{2,1}(s)+I_{2,2}(s)
$$
where
\begin{align*}
I_{2,1}(s)&=\sum_{1\leq i\leq n}(\hat{\sigma}_i-\eps_i)^2\(\sum_{1\leq j\leq n}\sign(X_j-X_i)Q(X_i,X_j,s)\)^2,\\
I_{2,2}(s)&=2\sum_{1\leq i\leq n}|\eps_i(\hat{\sigma}_i-\eps_i)|\(\sum_{1\leq j\leq n}\sign(X_j-X_i)Q(X_i,X_j,s)\)^2.
\end{align*}
Now, Assumption A\ref{as: function estimation} combined with (\ref{eq: variance calculations bound 1}) and (\ref{eq: variance calculations bound 2}) gives 
\begin{equation}\label{eq: variance derivations 2}
I_{2,1}(s)\leq \max_{1\leq i\leq n}(\hat{\sigma}_i-\eps_i)^2\sum_{1\leq i\leq n}\(\sum_{1\leq j\leq n}\sign(X_j-X_i)Q(X_i,X_j,s)\)^2\leq  (nh)^3h^{2k}o_p(n^{-2\kappa_1})
\end{equation}
uniformly over $s=(x,h)\in\MS_{n}$. In addition,
$$
I_{2,2}(s)\leq 2\max_{1\leq i\leq n}|\hat{\sigma}_i-\eps_i|\sum_{1\leq i\leq n}|\eps_i|\(\sum_{1\leq j\leq n}\sign(X_j-X_i)Q(X_i,X_j,s)\)^2,
$$
and since
$$
\Ep\left[\sum_{1\leq i\leq n}|\eps_i|\(\sum_{1\leq j\leq n}\sign(X_j-X_i)Q(X_i,X_j,s)\)^2|\{X_i\}\right]\leq C(nh)^3h^{2k}
$$
w.p.a.1 uniformly over $s=(x,h)\in\MS_{n}$ for some $C>0$, it follows that
\begin{equation}\label{eq: variance derivations 3}
I_{2,2}(s)= (n h)^3h^{2k}o_p(n^{-\kappa_1})
\end{equation}
uniformly over $s=(x,h)\in\MS_{n}$. Combining (\ref{eq: variance derivations 1})-(\ref{eq: variance derivations 3}) with the lower bound for $V(s)$ established above shows that under Assumption A\ref{as: function estimation}, $|\hat{V}(s)/V(s)-1|=o_p(n^{-\kappa_1})+O_p(\log p/(hn^{1/2}))$ uniformly over $s\in\MS_n$. This gives the asserted claim (c) and completes the proof of the theorem.
\end{proof}

\begin{proof}[Proof of Proposition \ref{lem: asumption 9 verification}]
Recall that $h_n=(\log p/n)^{1/(2\beta+3)}$. Therefore, for the basic set of weighting functions, there exists $c\in(0,1)$ such that for all $n$, there is $\tilde{h}\in H_n$ satisfying $\tilde{h}\in(ch_n/3,h_n/3)$. By Lemma \ref{lem: technical lemma} and Assumption A\ref{as: design difficult}, w.p.a.1, for all $[x_1,x_2]\subset[s_l,s_r]$ with $x_2-x_1=h_n$, there exists $i=\overline{1,n}$ such that $X_i\in[x_1+h_n/3,x_2-h_n/3]$. Then the weighting function with $(x,h)=(X_i,\tilde{h})$ satisfies conditions of Assumption A\ref{as: weighting functions difficult}.
\end{proof}

\section{Useful Lemmas}\label{sec: technical lemma}
\begin{lemma}\label{lem: technical lemma}
Let $W_1,\dots,W_n$ be an i.i.d. sequence of random variables with the support $[s_l,s_r]$ such that $c_1(x_2-x_1)\leq \Pr(W_1\in[x_1,x_2])\leq C_1(x_2-x_2)$ for some $c_1,C_1>0$ and all $[x_1,x_2]\subset[s_l,s_r]$. Then for any $c_2>0$, 
\begin{equation}\label{eq: observation drop bounds}
(1/2)c_1n(x_2-x_1)\leq\left|\{i=\overline{1,n}:W_i\in[x_1,x_2]\}\right|\leq (3C_1/2)n(x_2-x_1)
\end{equation}
simultaneously for all intervals $[x_1,x_2]\subset [s_l,s_r]$ satisfying $x_2-x_1\geq c_2(\log n)^2/n$ w.p.a.1.
\end{lemma}
\begin{proof}
Let $L=[(s_r-s_l)/(c_2(\log n)^2/n)]$ where $[x]$ denotes the largest integer smaller than or equal to  $x$. Denote $\Delta=(s_r-s_l)/L$. For $l=0,\dots,L$, denote $y_l=s_l+l\Delta/L$. It suffices to show that (\ref{eq: observation drop bounds}) holds simultaneously for all intervals $[x_1,x_2]$ of the form $[y_{l-1},y_l]$ for $l=\overline{1,L}$. 

For $l=\overline{1,L}$, let $I_{i,l}=1\{W_i\in[y_{l-1},y_l]\}$. Then 
$$
\sum_{1\leq i\leq n}I_{i,l}=\left|\{i=\overline{1,n}:W_i\in[y_{l-1},y_l]\}\right|.
$$
In addition, $\Ep[I_{i,l}]=\Ep[I_{i,l}^2]=\Pr(W_i\in[y_{l-1},y_l])$, so that $c_1n\Delta\leq \Ep[\sum_{1\leq i\leq n}I_{i,l}]\leq C_1n\Delta$ and $\text{Var}(\sum_{1\leq i\leq n}I_{i,l})\leq n\Ep[I_{i,l}^2]\leq C_1n\Delta$. Hence, Bernstein's inequality (see, for example, Lemma 2.2.9 in \cite{VaartWellner1996}) gives
\begin{eqnarray*}
\Pr\left(\sum_{1\leq i\leq n}I_{i,l}>(3/2)C_1n\Delta\right)&\leq& \exp(-C(\log n)^2),\\
\Pr\left(\sum_{1\leq i\leq n}I_{i,l}<(1/2)c_1n\Delta\right)&\leq& \exp(-C(\log n)^2)
\end{eqnarray*}
for some $C>0$. Therefore, by the union bound,
\begin{align*}
&\Pr\left(\sum_{1\leq i\leq n}I_{i,l}>(3/2)C_1n\Delta\,\text{ or }\sum_{1\leq i\leq n}I_{i,l}<(1/2)c_1n\Delta\text{ for some }l=\overline{1,L}\right)\\
&\leq 2L\exp(-C(\log n)^2)=o(1)
\end{align*}
where the last conclusion follows from $L\leq Cn$ for some $C>0$. This completes the proof of the lemma.
\end{proof}

\begin{lemma}\label{lem: maximal inequality CCK}
Let $W_1,\dots,W_n$ be independent random vectors in $\RR^p$ with $p\geq 2$. Let $W_{ij}$ denote the $j$th component of $W_i$, that is $W_i=(W_{i1},...,W_{ip})^T$. Define $M=\max_{1\leq j\leq n}\max_{1\leq j\leq p}|W_{ij}|$ and $\sigma^2=\max_{1\leq j\leq p}\sum_{1\leq i\leq n}\Ep[W_{ij}^2]$. Then
$$
\Ep\left[\max_{1\leq j\leq p}\left|\sum_{1\leq i\leq n}(W_{ij}-\Ep[W_{ij}])\right|\right]\leq C\left(\sigma\sqrt{\log p}+\sqrt{\Ep[M^2]}\log p\right)
$$
for some universal $C>0$.
\end{lemma}
\begin{proof}
See Lemma 8 in \cite{ChernozhukovKato2011}.
\end{proof}

\begin{lemma}\label{lem: gaussian comparison}
Let $W^1$ and $W^2$ be zero-mean Gaussian vectors in $\RR^p$ with covariances $\Sigma^1$ and $\Sigma^2$, respectively. Let $W^k_j$ denote the $j$th component of $W^k$, $k=1$ or $2$, that is $W^k=(W^k_1,...,W^k_p)^T$. Then for any $g\in\mathbb{C}^2(\RR,\RR)$,
$$
\left|\Ep\left[g\left(\max_{1\leq j\leq p}W^1_j\right)-g\(\max_{1\leq j\leq p}W^2_j\)\right]\right|\leq \Vert g^{(2)}\Vert_\infty\Delta_{\Sigma}/2+2\Vert g^{(1)}\Vert_\infty\sqrt{2\Delta_{\Sigma}\log p}
$$
where $\Delta_{\Sigma}=\max_{1\leq j,k\leq p}|\Sigma^1_{jk}-\Sigma^2_{jk}|$.
\end{lemma}
\begin{proof}
See Theorem 1 and following comments in \cite{ChernozhukovKato2011}.
\end{proof}

\begin{lemma}\label{lem: gaussian approximation original}
Let $x_1,...,x_n$ be a sequence of independent zero-mean vectors in $\RR^p$ with $x_{ij}$ denoting the $j$th component of $x_i$, that is $x_i=(x_{i1},...,x_{ip})^T$. Let $y_1,...,y_n$ be a sequence of independent zero-mean Gaussian vectors in $\RR^p$ with $y_{ij}$ denoting the $j$th component of $y_i$, that is $y_i=(y_{i1},...,y_{ip})^T$. Assume that $\Ep[x_ix_i^T]=\Ep[y_iy_i^T]$ for all $i=\overline{1,n}$. Further, assume that for all $i$ and $j$, $x_{ij}=z_{ij}u_{i}$ where $z_{ij}$'s are non-stochastic with $|z_{ij}|\leq B_n$ and $\sum_{1\leq i\leq n}z_{ij}^2/n=1$ where $\{B_n\}$ is a sequence of positive constants. Finally, assume that for some constants $c_1,C_1,c_2,C_2>0$, one of the following conditions holds: (i) $\Ep[u_{i}^2]\geq c_1$, $\Ep[\exp(|u_{i}|/C_1)]\leq 2$, and $B_n^2(\log(pn))^7/n\leq C_2n^{-c_2}$ or (ii) $\Ep[u_{i}^2]\geq c_1$, $\Ep[u_{i}^4]\leq C_1$, and $B_n^4(\log(pn))^7/n\leq C_2n^{-c_2}$. Then there exist constants $c,C>0$ depending only on $c_1,C_1,c_2,C_2$ such that
\begin{equation}\label{eq: gaussian approximation original}
\sup_{t\in\RR}\left|\Pr\left(\max_{1\leq j\leq p}\frac{1}{\sqrt{n}}\sum_{1\leq i\leq n}x_{ij}\leq t\right)-\Pr\left(\max_{1\leq j\leq p}\frac{1}{\sqrt{n}}\sum_{1\leq i\leq n}y_{ij}\leq t\right)\right|\leq Cn^{-c}
\end{equation}
for all $n$. In addition, if the terms $C_2n^{-c_2}$ above are replaced by $\eta_n$ where $\{\eta_n\}$ is a sequence of positive numbers converging to zero, then there exists another sequence $\{\eta_n'\}$ of positive numbers converging to zero and depending only on $\{\eta_n\}$ such that
\begin{equation}\label{eq: gaussian approximation modified}
\sup_{t\in\RR}\left|\Pr\left(\max_{1\leq j\leq p}\frac{1}{\sqrt{n}}\sum_{1\leq i\leq n}x_{ij}\leq t\right)-\Pr\left(\max_{1\leq j\leq p}\frac{1}{\sqrt{n}}\sum_{1\leq i\leq n}y_{ij}\leq t\right)\right|\leq \eta_n'
\end{equation}
for all $n$.
\end{lemma}
\begin{proof}
The result in (\ref{eq: gaussian approximation original}) is proven in Corollary 2.1 of \cite{ChernozhukovChetverikov12}. Further, inspecting the proof of Corollary 2.1 of \cite{ChernozhukovChetverikov12} shows that the sequences $C_2n^{-c_2}$ and $Cn^{-c}$ in (\ref{eq: gaussian approximation original}) can be replaced by general sequences $\{\eta_n\}$ and $\{\eta_n'\}$ of positive numbers converging to zero, and so the result in (\ref{eq: gaussian approximation modified}) holds as well.
\end{proof}

\section*{Supplementary Appendix}
This supplementary Appendix contains additional simulation results. In particular, I consider the test developed in this paper with weighting functions of the form given in equation (\ref{eq: weighting function}) with $k=1$. The simulation design is the same as in Section \ref{sec: monte carlo}. The results are presented in table 2. For ease of comparison, I also repeat the results for the tests of GSV, GHJK, and HH in this table. Overall, the simulation results in table 2 are similar to those in table 1, which confirms the robustness of the findings in this paper.
\begin{table}
\caption{Results of Monte Carlo Experiments}
\begin{tabular}{cccccccccccc}
\hline 
\multirow{2}{*}{{\small Noise}} & \multirow{2}{*}{{\small Case}} & \multirow{2}{*}{{\small Sample}} & \multicolumn{9}{c}{{\small Proportion of Rejections for}}\tabularnewline
\cline{4-12} 
&&& {\small GSV} & {\small GHJK} & {\small HH} & {\small CS-PI} & {\small CS-OS} & {\small CS-SD} & {\small IS-PI} & {\small IS-OS} & {\small IS-SD} \tabularnewline
\hline 
\multicolumn{1}{c}{} && {\small 100} & {\small .118} & {\small .078} & {\small .123} & {\small .129} & {\small .129} & {\small .129} & {\small .166} & {\small .166} & {\small .166}\tabularnewline
{\small normal}&{\small 1} & {\small 200} & {\small .091} & {\small .051} & {\small .108} & {\small .120} & {\small .120} & {\small .120} & {\small .144} & {\small .144} & {\small .144}\tabularnewline
&& {\small 500} & {\small .086} & {\small .078} & {\small .105} & {\small .121} & {\small .121} & {\small .121} & {\small .134} & {\small .134} & {\small .134}\tabularnewline
\hline
&& {\small 100} & {\small 0} & {\small .001} & {\small 0} & {\small .002} & {\small .009} & {\small .009} & {\small .006} & {\small .024} & {\small .024}\tabularnewline
{\small normal}&{\small 2}& {\small 200} & {\small 0} & {\small .002} & {\small 0} & {\small .001} & {\small .012} & {\small .012} & {\small .007} & {\small .016} & {\small .016}\tabularnewline
&& {\small 500} & {\small 0} & {\small .001} & {\small 0} & {\small .002} & {\small .005} & {\small .005} & {\small .005} & {\small .016} & {\small .016}\tabularnewline
\hline 
&& {\small 100} & {\small 0} & {\small .148} & {\small .033} & {\small .238} & {\small .423} & {\small .432} & {\small 0} & {\small 0} & {\small 0}\tabularnewline
{\small normal}&{\small 3}& {\small 200} & {\small .010} & {\small .284} & {\small .169} & {\small .639} & {\small .846} & {\small .851} & {\small .274} & {\small .615} & {\small .626}\tabularnewline
&& {\small 500} & {\small .841} & {\small .654} & {\small .947} & {\small .977} & {\small .995} & {\small .996} & {\small .966} & {\small .994} & {\small .994}\tabularnewline
\hline 
&& {\small 100} & {\small .037} & {\small .084} & {\small .135} & {\small .159} & {\small .228} & {\small .231} & {\small .020} & {\small .040} & {\small .040}\tabularnewline
{\small normal}&{\small 4}& {\small 200} & {\small .254} & {\small .133} & {\small .347} & {\small .384} & {\small .513} & {\small .515} & {\small .372} & {\small .507} & {\small .514}\tabularnewline
&& {\small 500} & {\small .810} & {\small .290} & {\small .789} & {\small .785} & {\small .833} & {\small .833} & {\small .782} & {\small .835} & {\small .836}\tabularnewline
\hline 
\multicolumn{1}{c}{} && {\small 100} & {\small .109} & {\small .079} & {\small .121} & {\small .120} & {\small .120} & {\small .120} & {\small .200} & {\small .200} & {\small .200}\tabularnewline
{\small uniform}&{\small 1} & {\small 200} & {\small .097} & {\small .063} & {\small .109} & {\small .111} & {\small .111} & {\small .111} & {\small .154} & {\small .154} & {\small .154}\tabularnewline
&& {\small 500} & {\small .077} & {\small .084} & {\small .107} & {\small .102} & {\small .102} & {\small .102} & {\small .125} & {\small .125} & {\small .125}\tabularnewline
\hline
&& {\small 100} & {\small .001} & {\small .001} & {\small 0} & {\small 0} & {\small .006} & {\small .006} & {\small .015} & {\small .031} & {\small .031}\tabularnewline
{\small uniform}&{\small 2}& {\small 200} & {\small 0} & {\small 0} & {\small 0} & {\small .001} & {\small .009} & {\small .009} & {\small .013} & {\small .021} & {\small .024}\tabularnewline
&& {\small 500} & {\small 0} & {\small .003} & {\small 0} & {\small .003} & {\small .012} & {\small .012} & {\small .011} & {\small .021} & {\small .021}\tabularnewline
\hline 
&& {\small 100} & {\small 0} & {\small .151} & {\small .038} & {\small .225} & {\small .423} & {\small .433} & {\small 0} & {\small 0} & {\small 0}\tabularnewline
{\small uniform}&{\small 3}& {\small 200} & {\small .009} & {\small .233} & {\small .140} & {\small .606} & {\small .802} & {\small .823} & {\small .261} & {\small .575} & {\small .590}\tabularnewline
&& {\small 500} & {\small .811} & {\small .582} & {\small .947} & {\small .976} & {\small .993} & {\small .994} & {\small .971} & {\small .990} & {\small .991}\tabularnewline
\hline 
&& {\small 100} & {\small .034} & {\small .084} & {\small .137} & {\small .150} & {\small .216} & {\small .219} & {\small .020} & {\small .046} & {\small .046}\tabularnewline
{\small uniform}&{\small 4}& {\small 200} & {\small .197} & {\small .116} & {\small .326} & {\small .355} & {\small .483} & {\small .488} & {\small .328} & {\small .466} & {\small .472}\tabularnewline
&& {\small 500} & {\small .803} & {\small .265} & {\small .789} & {\small .803} & {\small .852} & {\small .855} & {\small .796} & {\small .859} & {\small .861}\tabularnewline
\hline 
\end{tabular}

{\small Nominal Size is 0.1. GSV, GHJK, and HH stand for the tests of \cite{GSV2000},
\cite{Gijbels2000}, and \cite{HallHeckman2000} respectively. CS-PI, CS-OS, and CS-SD refer to the test developed in this paper with $\sigma_i$ estimated using Rice's formula and plug-in, one-step, and step-down critical values respectively. Finally, IS-PI, IS-OS, and IS-SD refer to the test developed in this paper with $\sigma_i$ estimated by $\hat{\sigma}_i=\hat{\eps}_i$ and plug-in, one-step, and step-down critical values respectively.}
\end{table}


\begin{thebibliography}{99}
\bibitem[Anderson and Schmittlein(1984)]{AndersonSchmittlein}
Anderson, E., and Schmittlein, D. (1984). Integration of the Sales Force: An Empirical Examination. {\em RAND Journal of Economics}, \textbf{15}, 385-395.






\bibitem[Andrews and Shi(2010)]{AndrewsandShi2010}
Andrews, D. W. K., and Shi, X. (2013). Inference Based on Conditional Moment Inequalities. {\em Econometrica}, \textbf{81}, 609-666.


\bibitem[Angeletos and Werning(2006)]{AngeletosWerning2006}
Angeletos, M., and Werning, I. (2006). Information Aggregation, Multiplicity, and Volatility. {\em American Economic Review}, \textbf{96}.




\bibitem[Athey(2002)]{Athey2002}
Athey, S. (2002). Monotone Comparative Statics under Uncertainty. {\em The Quarterly Journal of Economics}, \textbf{117}, 187-223.


\bibitem[Beraud, Huet, and Laurent(2005)]{Beraud2005}
Beraud, Y., Huet, S., and Laurent, B. (2005). Testing Convex Hypotheses on the Mean of a Gaussian Vector. Application to Testing Qualitative Hypotheses on a Regression Function. {\em The Annals of Statistics}, \textbf{33}, 214-257.


\bibitem[Blundell, Chen, and Kristensen(2007)]{BCK07}
Blundell, R., Chen, X., and Kristensen, D. (2007). Semi-Nonparametric IV Estimation of Shape-Invariant Engel Curves. {\em Econometrica}, \textbf{75}, 1613-1669.

\bibitem[Boucheron, Bousquet, and Lugosi(2004)]{BoucheronBousquetLugosi}
Boucheron, S., Bousquet, O., and Lugosi, G. (2004). Concentration Inequalities. {\em Advanced Lectures in Machine Learning}, 208-240.

\bibitem[Bowman, Jones, and Gijbels(1998)]{Bowman98}
Bowman, A. W., Jones, M. C., and Gijbels, I. (1998). Testing Monotonicity of Regression. {\em Journal of Computational and Graphical Statistics}, \textbf{7}, 489-500.


\bibitem[Cai and Wang(2008)]{Cai2008}
Cai, T., and Wang, L. (2008). Adaptive Variance Function Estimation in Heteroscedastic Nonparametric Regression. {\em The Annals of Statistics}, \textbf{36}, 2025-2054.






\bibitem[Chernozhukov, Chetverikov, and Kato(2011)]{ChernozhukovKato2011}
Chernozhukov, V., Chetverikov, D., and Kato, K. (2011). Comparison and Anti-Concentration Bounds for Maxima of Gaussian Random Vectors. {\em arXiv:1301.4807}.

\bibitem[Chernozhukov, Chetverikov, and Kato(2012)]{ChernozhukovChetverikov12}
Chernozhukov,V., Chetverikov, D., and Kato, K. (2012). Gaussian Approximations and Multiplier Bootstrap for Maxima of Sums of High-Dimensional Random Vectors. {\em The Annals of Statistics, forthcoming}.

\bibitem[Chernozhukov, Lee, and Rosen(2013)]{ChernozhukovLeeRosen2009}
Chernozhukov, V., Lee, S. and Rosen, A. (2013). Intersection Bounds: Estimation and Inference. {\em Econometrica}, \textbf{81}, 667-737.

\bibitem[Chetverikov(2012)]{Chetverikov2012}
Chetverikov, D. (2012). Adaptive Test of Conditional Moment Inequalities. {\em arXiv:1201.0167v2}.

\bibitem[Darolles el. al.(2011)]{DFR12}
Darolles, S., Fan, Y., Florens, J., Renault, E. (2011). Nonparametric Instrumental Regression. {\em Econometrica}, \textbf{79}, 1541-1565.

\bibitem[Das, Newey, and Vella(2003)]{DNV03}
Das, M., Newey, W., and Vella, F. (2003). Nonparametric Estimation of Sample Selection Models. {\em Review of Economic Studies}, \textbf{70}, 33-58.

\bibitem[Delgado and Escanciano(2010)]{DelgadoEscanciano2010}
Delgado, M., and Escanciano, J. (2010). Testing Conditional Monotonicity in the Absence of Smoothness. {\em Unpublished manuscript}.

\bibitem[Dudley(1999)]{Dudley1999}
Dudley, R. (1999). Uniform Central Limit Theorems. {\em Cambridge Studies in Advanced Mathematics}.

\bibitem[Dumbgen and Spokoiny(2001)]{Dumbgen2001}
Dumbgen, L., and Spokoiny, V. (2001). Multiscale Testing of Qualitative Hypotheses. {\em The Annals of Statistics}, \textbf{29}, 124-152.

\bibitem[Durot(2003)]{Durot2003}
Durot, C. (2003). A Kolmogorov-type Test for Monotonicity of Regression. {\em Statistics and Probability Letters}, \textbf{63}, 425-433.

\bibitem[Ellison and Ellison(2011)]{EllisonEllison}
Ellison, G, and Ellison, S. (2011). Strategic Entry Deterrence and the Behavior of Pharmaceutical Incumbents Prior to Patent Expiration. {\em American Economic Journal: Microeconomics}, \textbf{3}, 1-36.



\bibitem[Fan and Yao(1998)]{FanYao1998}
Fan, J., and Yao, Q. (1998). Efficient Estimation of Conditional Variance Functions in Stochastic Regression. {\em Biometrika}, \textbf{85}, 645-660.

\bibitem[Ghosal, Sen, and van der Vaart(2000)]{GSV2000}
Ghosal, S., Sen, A., and van der Vaart, A. (2000). Testing Monotonicity of Regression. {\em The Annals of Statistics}, \textbf{28}, 1054-1082.

\bibitem[Gijbels et. al.(2000)]{Gijbels2000}
Gijbels, I., Hall, P., Jones, M., and Koch, I. (2000). Tests for Monotonicity of a Regression Mean with Guaranteed Level. {\em Biometrika}, \textbf{87}, 663-673.

\bibitem[Gutknecht(2013)]{Gutknecht2013}
Gutknecht, D. (2013). Testing Monotonicity under Endogeneity. {\em Unpublished manuscript}.




\bibitem[Hall and Heckman(2000)]{HallHeckman2000}
Hall, P. and Heckman, N. (2000). Testing for Monotonicity of a Regression Mean by Calibrating for Linear Functions. {\em The Annals of Statistics}, \textbf{28}, 20-39.

\bibitem[Hall and Horowitz(2005)]{HH05}
Hall, P. and Horowitz, J. (2005). Nonparametric Methods for Inference in the Presence of Instrumental Variables. {\em The Annals of Statistics}, \textbf{33}, 2904-2929.

\bibitem[Hansen(2000)]{Hansen00}
Hansen, B. (2000). Econometrics. {\em Unpublished manuscript}.

\bibitem[Hardle and Mammen(1993)]{HardleMammen1993}
Hardle, W., and Mammen, E. (1993). Comparing Nonparametric Versus Parametric Regression Fits. {\em The Annals of Statistics}, \textbf{21}, 1926-1947.

\bibitem[Hardle and Tsybakov(2007)]{HardleTsybakov2007}
Hardle, W., and Tsybakov, A. (2007). Local Polinomial Estimators of the Volatility Function in Nonparametric Autoregression. {\em Journal of Econometrics}, \textbf{81}, 233-242.


\bibitem[Heckman(1979)]{Heckman79}
Heckman, J. (1979). Sample Selection Bias as a Specification Error. {\em Econometrica}, \textbf{47}, 153-161.

\bibitem[Holm(1979)]{Holm1979}
Holm, S. (1979). A simple sequentially rejective multiple test procedure. {\em Scandinavian Journal of Statistics}, \textbf{6}, 65-70.

\bibitem[Holmstrom and Milgrom(1994)]{HolmstromMilgrom94}
Holmstrom, B., and Milgrom, P. (1994). The Firm as an Incentive System. {\em The American Economic Review}, \textbf{84}, 972-991.

\bibitem[Horowitz(2009)]{Horowitz_book}
Horowitz, J. (2009). Semiparametric and Nonparametric Methods in Econometrics. {\em Springer Series in Statistics}.

\bibitem[Horowitz and Spokoiny(2001)]{Horowitz2001}
Horowitz, J. L., and Spokoiny, V. (2001). An Adaptive, Rate-Optimal Test of a Parametric Mean-Regression Model against a Nonparametric Alternative. {\em Econometrica}, \textbf{69}, 599-631.



\bibitem[Kasy(2012)]{Kasy12}
Kasy, M. (2012). Identification in Continuous Triangular Systems with Unrestricted Heterogeneity. {\em Unpublished manuscript}.






\bibitem[Leadbetter, Lindgren, and Rootzen(1983)]{Leadbetter_book}
Leadbetter, M., Lindgren, G., and Rootzen, H. (1983). Extremes and Related Properties of Random Sequences and Processes. {\em Springer}.

\bibitem[Lehmann and Romano(2005)]{LehmannRomano2005}
Lehmann, E.L., and Romano, J. (2005). Testing Statistical Hypotheses. {\em Springer}.



\bibitem[Lee, Linton, and Whang(2009)]{LLW2009}
Lee, S., Linton, O., and Whang, Y. (2009). Testing for Stochastic Monotonicity. {\em Econometrica}, \textbf{27}, 585-602.

\bibitem[Lee, Song, and Whang(2011)]{LeeandSongandWhang2011}
Lee, S., Song, K., and Whang, Y. (2011). Testing function inequalities. {\em CEMMAP working paper CWP 12/11}.

\bibitem[Lee, Song, and Whang(2011b)]{LSW1}
Lee, S., Song, K., and Whang, Y. (2011). Nonparametric Tests of Monotonicity: an $L_p$ approach. {\em Unpublished manuscript}.



\bibitem[Liu(1988)]{Liu88}
Liu, R. (1988). Bootstrap Procedures under niid Models. {\em The Annals of Statistics}, \textbf{16}, 1696-1708.

\bibitem[Mammen(1993)]{Mammen93}
Mammen, E. (1993). Bootstrap and Wild Bootstrap for High Dimensional Linear Models. {\em The Annals of Statistics}, \textbf{21}, 255-285.

\bibitem[Manski and Pepper(2000)]{ManskiPepper}
Manski, C. and Pepper, J. (2000). Monotone Instrumental Variables: With an Application to Returns to Schooling. {\em Econometrica}, \textbf{68}, 997-1010.


\bibitem[Matzkin(1994)]{Matzkin94}
Matzkin, R. (1994). Restrictions of Economic Theory in Nonparametric Methods {\em Handbook of Econometrics, chapter 42}.


\bibitem[Merton(1974)]{Merton74}
Merton, R. (1974). On the Pricing of Corporate Debt: the Risk Structure of Interest Rates. {\em Journal of Finance}, \textbf{29}, 449-470.


\bibitem[Milgrom and Shannon(1994)]{MilgromShannon}
Milgrom, P. and Shannon, C. (1994). Monotone Comparative Statics. {\em Econometrica}, \textbf{62}, 157-180.


\bibitem[Morris and Shin(1998)]{MorrisShin98}
Morris, S. and Shin, H. (1998). Unique Equilibrium in a Model of Self-Fulfilling Currency Attacks. {\em The American Economic Review}, \textbf{88}, 587-597.

\bibitem[Morris and Shin(2003)]{MorrisShin2001}
Morris, S. and Shin, H. (2003). Global Games: Theory and Application. {\em Advances in Economics and Econometrics, Eight World Congress}, Volume 1, Chapter 3, 56-114.

\bibitem[Morris and Shin(2004)]{MorrisShin2004}
Morris, S. and Shin, H. (2004). Coordination Risk and the Price of Debt. {\em European Economic Review}, \textbf{48}, 133-153.

\bibitem[Muller(1991)]{Muller91}
Muller, H. (1991). Smooth Optimum Kernel Estimators near Endpoints. {\em Biometrika}, \textbf{78}, 521-530.

\bibitem[Muller and Stadtmuller(1987)]{Muller1987}
Muller, H. and Stadtmuller, U. (1987). Estimation of Heteroscedasticity in Regression Analysis. {\em The Annals of Statistics}, \textbf{15}, 610-625.

\bibitem[Newey(1997)]{Newey97}
Newey, W. (1997). Convergence Rates and Asymptotic Normality for Series Estimators. {\em Journal of Econometrics}, \textbf{79}, 147-168.

\bibitem[Newey and Powell(2003)]{NP03}
Newey, W. and Powell, J. (2003). Instrumental Variable Estimation of Nonparametric Models. {\em Econometrica}, \textbf{71}, 1565-1578.

\bibitem[Newey, Powell, and Vella(1999)]{NPV99}
Newey, W., Powell, J., and Vella, F. (1999). Nonparametric Estimation of Triangular Simultaneous Equations Models. {\em Econometrica}, \textbf{67}, 565-603.




\bibitem[Piterbarg(1996)]{Piterbarg_book}
Piterbarg, D. (1996). Asymptotic Methods in the Theory of Gaussian Processes and Fields. {\em American Mathematical Society}, \textbf{148}, Translation of Mathematical Monographs.

\bibitem[Pollard(1984)]{Pollard_book}
Pollard, D. (1984). Convergence of Stochastic Processes. {\em Springer-Verlag}.

\bibitem[Poppo and Zenger(1998)]{PoppoZenger98}
Poppo, L., and Zenger, T. (1998). Testing Alternative Theories of the Firm: Transaction Cost, Knowlegde-Based, and Measurement Explanations of Make-or-Buy Decisions in Information Services. {\em Strategic Management Journal}, \textbf{19}, 853-877.

\bibitem[Rice(1984)]{Rice1984}
Rice, J. (1984). Bandwidth Choice for Nonparametric Kernel Regression. {\em The Annals of Statistics}, \textbf{12}, 1215-1230.

\bibitem[Rio(1994)]{Rio94}
Rio, E. (1994). Local Invariance Principles and Their Application to Density Estimation. {\em Probability Theory and Related Fields}, \textbf{98}, 21-45.

\bibitem[Robinson(1988)]{Robinson88}
Robinson, P. M. (1988). Root-N-Consistent Semiparametric Regression. {\em Econometrica}, \textbf{56}, 931-954.

\bibitem[Romano and Shaikh(2010)]{RomanoShaikh2010}
Romano, J. and Shaikh, A. (2010). Inference for the Identified Sets in Partially Identified Econometric Models. {\em Econometrica}, \textbf{78}, 169-211.


\bibitem[Romano and Wolf(2005a)]{RomanoWolf2005}
Romano, J. and Wolf, M. (2005a). Exact and Approximate Stepdown Methods for Multiple Hypothesis Testing. {\em Journal of American Statistical Association}, \textbf{100}, 94-108.

\bibitem[Romano and Wolf(2005b)]{RomanoWolf2005ECMA}
Romano, J. and Wolf, M. (2005b). Stepwise Multiple Testing as Formalized Data Snooping. {\em Econometrica}, \textbf{73}, 1237-1282.

\bibitem[Romano and Wolf(2011)]{RomanoWolf11}
Romano, J. and Wolf, M. (2011). Alternative Tests for Monotonicity in Expected Asset Returns. {\em Unpublished manuscript}.


\bibitem[Schlee(1982)]{Schlee82}
Schlee, W. (1982). Nonparametric Tests of the Monotony and Convexity of Regression. {\em In Nonparametric Statistical Inference, Amsterdam: North-Holland}.



\bibitem[Tirole(1988)]{Tirole_book}
Tirole, J. (1988). The Theory of Industrial Organization. {\em Cambridge, MA: MIT Press}.

\bibitem[Tsybakov(2009)]{Tsybakov_book}
Tsybakov, A. (2009). Introduction to Nonparametric Estimation. {\em Springer}.

\bibitem[Van der Vaart and Wellner(1996)]{VaartWellner1996}
Van der Vaart, A. and Wellner, J. (1996). Weak Convergence and Empirical Processes with Applications to Statistics. {\em Springer}.

\bibitem[Wang and Meyer(2011)]{WangAndMeyer2011}
Wang, J. and Meyer, M. (2011). Testing the Monotonicity or Convexity of a Function Using Regression Splines. {\em The Canadian Journal of Statistics}, \textbf{39}, 89-107.

\bibitem[Wu(1986)]{Wu1986}
Wu, C. (1986). Jacknife, Bootstrap, and Other Resampling Methods in Regression Analysis. {\em The Annals of Statistics}, \textbf{14}, 1261-1295.


\end{thebibliography}
\end{document}